\documentclass[12pt, leqno]{amsart}

\usepackage{graphicx}
\usepackage{amssymb,amsmath,amsthm,amscd}
\usepackage{mathrsfs}
\usepackage{lscape}
\usepackage{enumerate}
\usepackage[usenames,dvipsnames]{color}
\usepackage[colorlinks=true, pdfstartview=FitV,
 linkcolor=blue,citecolor=blue,urlcolor=blue]{hyperref}
\usepackage{tikz}

\usepackage[all]{xy}

\allowdisplaybreaks[3]

\numberwithin{equation}{section}

\newenvironment{rouge}{\relax\color{red}}{\relax}

\newcommand{\ber}{\begin{rouge}}
\newcommand{\er}{\end{rouge}}

\theoremstyle{plain}
\newtheorem{lemma}{Lemma}[section]
\newtheorem{proposition}[lemma]{Proposition}
\newtheorem{theorem}[lemma]{Theorem}

\theoremstyle{definition}
\newtheorem{remark}[lemma]{Remark}
\newtheorem{example}[lemma]{Example}
\newtheorem{algorithm}[lemma]{Algorithm}

\newtheorem{definition}[lemma]{Definition}
\newtheorem{corollary}[lemma]{Corollary}

\newtheorem{convention}[lemma]{Convention}

\newtheorem{observation}[lemma]{Observation}

\newcommand{\soc}{{\operatorname{soc}}}

\newcommand{\C}{\mathbb{C}}
\newcommand{\Z}{\mathbb{Z}}
\newcommand{\N}{\mathsf{N}}

\newcommand{\cmA}{\mathsf{A}}
\newcommand{\seteq}{\mathbin{:=}}

\newcommand{\al}{\alpha}

\newcommand{\bi}{\mathbf{i}}

\newcommand{\PP}{ \textbf{\textit{P}}}
\newcommand{\QQ}{ \textbf{\textit{Q}}}
\newcommand{\RR}{ \textbf{\textit{R}}}
\newcommand{\Qd}{{Q^{\diamondsuit}}}
\newcommand{\PPi}{\PP_{\lf \Qd \rf}}

\newcommand{\ko}{ \textbf{k}}

\newcommand{\ii}{ \textbf{\textit{i}}}
\newcommand{\jj}{ \textbf{\textit{j}}}

\newcommand{\ut}{ \textbf{\textit{t}}}

\newcommand{\be}{\beta}
\newcommand{\ga}{\gamma}

\newcommand{\shp}{\hspace{-0.4ex}+\hspace{-0.4ex}}

\newcommand{\PR}{\Phi^+}
\newcommand{\tw}{{\widetilde{w}}}

\newcommand{\redex}{{\widetilde{w}}}
\newcommand{\redez}{{\widetilde{w}_0}}
\newcommand{\um}{\underline{m}}

\newcommand{\us}{\underline{s}}

\newcommand{\up}{\underline{p}}

\newcommand{\tb}{\mathtt{b}}
\newcommand{\rl}{\mathsf{Q}}
\newcommand{\wt}{{\rm wt}}
\newcommand{\rev}{{\rm rev}}

\newcommand{\gdist}{{\rm gdist}}

\newcommand{\rds}{{\rm rds}}

\newcommand{\oUp}{\overline{\Upsilon}}
\newcommand{\uUp}{\underline{\Upsilon}}
\newcommand{\lf}{[\hspace{-0.3ex}[}
\newcommand{\rf}{]\hspace{-0.3ex}]}
\newcommand{\htau}{\ut\hspace{-1.5ex}-}

\newcommand{\wUp}{ \widehat{\Upsilon}}
\newcommand{\wpl}{ \widehat{+}}
\newcommand{\wmi}{ \widehat{-}}
\newcommand{\wpm}{ \widehat{\pm}}

\newcommand*\ov[1]{\overline{#1}}

\newcommand{\rectangle} {\xy (3,-3)*{}="T"; (5,-25)*{}="B";(15,-15)*{}="C";
"T"; "C" **\dir{-}; "C"; "B" **\dir{-}; "T"; "T"+(-5,-5) **\dir{-};
"T"+(-5,-5); "T"+(-10,-10) **\dir{-}; "B"; "B"+(-5,5) **\dir{-};
"B"+(-5,5); "B"+(-12,12) **\dir{-}; "C"*{\bullet}; "B"*{\bullet};
"T"*{\bullet}; "B"+(-12,12)*{\bullet}; "B"+(-17,12)*{\scriptstyle
\be}; "B"+(0,-3)*{\scriptstyle \ga}; "C"+(5,0)*{\scriptstyle
\al}; "T"+(0,3)*{\scriptstyle \eta};
\endxy}

\newlength{\mylength}
\title[Twisted Coxeter elements and folded AR-quivers: I]{Twisted Coxeter elements and folded AR-quivers via Dynkin diagram automorphisms: I}

\author[S.-j. Oh, U. Suh]{Se-jin Oh$^\dagger$,  Uhi Rinn  Suh$^\ddagger$}

\address{Department of Mathematics Ewha Womans University Seoul 120-750, Korea}
\email{sejin092@gmail.com}
\address{Department of Mathematical Sciences, KAIST, 291 Daehak-ro Yuseonggu
Daejeon, 305-701 South Korea}
\email{uhrisu@gmail.com}

\thanks{$^\dagger$ This work was supported by NRF Grant \# 2016R1C1B2013135.}
\thanks{$^\ddagger$ This work was supported by NRF Grant \# 2016R1C1B1010721.}

\keywords{longest element, $r$-cluster point, twisted Coxeter elements,
twisted AR-quivers, twisted additive property, folded AR-quivers, folded distance polynomials, denominator formulas}
\subjclass[2010]{81R50, 05E10, 16T30, 17B37}

\begin{document}

\begin{abstract}
We introduce and study the twisted adapted $r$-cluster point and its
combinatorial Auslander-Reiten quivers, called twisted AR-quivers and folded AR-quivers, of type $A_{2n+1}$ which are closely related to twisted Coxeter elements and
associated with the non-trivial Dynkin diagram automorphism. As applications of the study,
we prove that any folded AR-quiver of type $A_{2n+1}$ and hence the twisted adapted $r$-cluster point encode crucial information on the representation theory of quantum affine algebra $U_q'(B^{(1)}_{n+1})$ such as
Dorey's rule and denominator formulas.
\end{abstract}

\maketitle

\section*{Introduction}

Let $Q$ be a Dynkin quiver of finite type $ADE$.
By the Gabriel theorem \cite{Gab80}, it is well-known that the information on the representation theory for the path algebra $\C Q$ is encoded in
its Auslander-Reiten quiver $\Gamma_Q$. More precisely, each vertex in $\Gamma_Q$ corresponds to an
indecomposable module over $\C Q$ and each arrow in $\Gamma_Q$ corresponds to an irreducible morphism. Interestingly,
{\rm (i)}  the set of all vertices in $\Gamma_Q$ can be identified with the set $\PR$ of positive roots  of finite type $ADE$ via
dimension vectors, {\rm (ii)} the paths in $\Gamma_Q$ represent the convex partial order $\prec_Q$ on $\PR$, {\rm (iii)} $\Gamma_Q$
has a canonical coordinate system.
It is also well-known that $\Gamma_Q$ satisfies {\it additive property} and $\Gamma_Q$ can be constructed by using
its unique {\it Coxeter element} $\phi_Q$.

On the other hand, for each commutation class $[\redez]$ of reduced expressions of the longest element $w_0$ for a Weyl group $W$ of finite type,
there exists a convex partial order $\prec_{[\redez]}$ on $\PR$. In particular, {\rm (i)} $\prec_Q$ corresponds to the commutation class $[Q]$ of $w_0$ which is
{\it adapted} to the Dynkin quiver $Q$ of finite type $ADE$ and {\rm (ii)} all commutations classes $[Q]$ can be grouped into one {\it $r$-cluster point} $\lf Q\rf$
via reflection equivalence relations arising from reflection functor. In \cite{OS15}, the authors introduced the combinatorial AR-quiver $\Upsilon_{[\redez]}$
for every class $[\redez]$ of $w_0$ for any finite type and its combinatorial properties as a realization of $\prec_{[\redez]}$. The construction
of $\Upsilon_{[\redez]}$ was given in an inductive way and $\Upsilon_{[\redez]}$ does not have a coordinate system in general.

\medskip

In the representation theory of quantum affine algebras, Coxeter elements and their AR-quivers appear in interesting ways. More precisely,
(1) Chari-Pressley \cite{CP96} determined the conditions for (see \cite[\S 1.3]{AK} for definition of $V(\varpi_i)_x$)
$$ {\rm Hom}(V(\varpi_i)_x\otimes V(\varpi_j)_y, V(\varpi_k)_z) \ne 0, $$
referred as Dorey's rule untwisted quantum affine algebras of type $A^{(1)}_n$, $D^{(1)}_n$, $B^{(1)}_n$ and $C^{(1)}_n$. For type $A^{(1)}_n$ and $D^{(1)}_n$, they used Coxeter elements and, for type $B^{(1)}_n$ and $C^{(1)}_n$, they used twisted Coxeter elements. (The Dorey's rules for $A^{(2)}_n$ and $D^{(2)}_n$ were studied in \cite{KKKO14D,Oh14}).
(2) With newly introduced notions in \cite{Oh15E}, the first named author
 proved that we can {\it read} the denominator formulas $d_{k,l}(z)$ for $U_q'(A^{(t)}_{n})$ and $U_q'(D^{(t)}_{n})$ $(t=1,2)$ from {\it any} $\Gamma_Q$ of finite type $AD_n$,
which {\it control} their representation theory (see \cite[Theorem 2.2.1]{KKK13A}).
(3) In \cite{Oh14A,Oh14D}, he also proved that Dorey's rule for $U_q'(A^{(t)}_{n})$ and $U_q'(D^{(t)}_{n})$ $(t=1,2)$ can be interpreted as the {\it coordinates}
 of $(\al,\be,\ga)$ in some $\Gamma_Q$ where $(\al,\be)$ is a pair for $\ga=\al+\be \in \PR$. Furthermore, every pair $(\al,\be)$ of $\ga$ is $[Q]$-minimal
 when $Q$ is of type $A$.

Now, we have interesting relationships can be described the following conceptual diagram $(t=1,2)$:
\begin{align} \label{eq: concept}
\scalebox{0.84}{\raisebox{2.8em}{\xymatrix@C=6ex@R=2ex{ && U_q'(AD^{(t)}_n) \ar@{<->}[ddr]\ar@{<->}[ddl]  \\ \\
\{ [\redez] \} & \{ [Q] \in \lf Q \rf \}  \ar@{_{(}->}[l] \ar@{<->}[rr] &&  \{ \Gamma_Q \ | \ [Q] \in \lf Q \rf \} \ar@{<->}[r] & \{ \phi_Q \} }}}
\end{align}

The goal of this paper is to make analogue of the diagram \eqref{eq: concept} by using twisted Coxeter elements of finite type $A_{2n+1}$,
as Chari-Pressley used them to investigate Dorey's rule for $U_q'(B_{n+1}^{(1)})$. Here,
the twisted Coxeter elements can be understood as a generalization of Coxeter elements via the non-trivial Dynkin diagram automorphism $\vee$.
Note that the orbit of the automorphism yields the Dynkin diagram of type $B_{n+1}$
(see Definition \ref{def: twisted Coxeter}):
$$\xymatrix@R=0.5ex@C=4ex{ *{\circ}<3pt>  \ar@{<->}@/^1pc/[rrrrr] \ar@{-}[r]_<{1 \ \ }  &*{\circ}<3pt> \ar@{<->}@/^/[rrr]
\ar@{-}[r]_<{2 \ \ }  &   {}
\ar@{.}[r] & *{\circ}<3pt> \ar@{-}[r]_>{\,\,\,\ 2n} &*{\circ}<3pt>\ar@{-}[r]_>{\,\,\,\, 2n+1} &*{\circ}<3pt> }
\overset{\vee}{\longrightarrow}
\xymatrix@R=0.5ex@C=4ex{ *{\circ}<3pt>  \ar@{-}[r]_<{1 \ \ }  &*{\circ}<3pt>
\ar@{-}[r]_<{2 \ \ }  &   {}
\ar@{.}[r] & *{\circ}<3pt> \ar@{-}[r]_>{\,\,\,\ n} &*{\circ}<3pt>\ar@{=>}[r]_>{\,\,\,\, n+1} &*{\circ}<3pt> }
$$

The various approaches to study on non-simply laced type via Dynkin diagram folding were well-developed in many contexts (for examples, see \cite{DR1,DR2,GLS,Lus93}). In this paper, we do {\it not} study
on the finite type $B_{n+1}$ whose Dynkin diagram is an orbit of the Dynkin diagram $A_{2n+1}$ via the automorphism. We {\it do} study certain family of reduced expressions of the longest element of Weyl group of the same type $A_{2n+1}$, which is characterizable and related to the representation theory of {\it quantum affine algebra of type $B^{(1)}_{n+1}$}, by {\it folding its combinatorial AR-quivers via the automorphism $\vee$}.

\medskip

We first construct and characterize the cluster point $\lf \Qd \rf$ of $w_0$ {\it arising from any} twisted Coxeter elements $\phi_{[\ii_0]}\vee$
of finite type $A_{2n+1}$ (Definition \ref{def: vee-foldable}). The cluster point $\lf \Qd \rf$ can be distinguished from $\lf Q \rf$ via their {\it $\vee$-Coxeter compositions} (Definition \ref{Def:Coxeter_composition}):
$$\mathsf{C}^{\vee}_{\lf \Qd \rf}= ( \underbrace{2n+1, \ldots ,2n+1}_{ n+1\text{-times}} ) \text{ and }
\mathsf{C}^{\vee}_{\lf Q \rf}= (\underbrace{2n+2, \ldots, 2n+2}_{ n\text{-times}}, n+1)$$
(see Example \ref{ex: Coxeter composition} also). The cluster point $\lf \Qd \rf$ called the {\it twisted adapted cluster point} and its commutation classes are said to
be {\it twisted adapted} (Definition \ref{def: twisted cluster}). Surprisingly, the point $\lf \Qd \rf$ of type $A_{2n+1}$ is closely related to the point $\lf Q \rf$ of type $A_{2n}$ and
each commutation class $[\ii_0]$ in $\lf \Qd \rf$ can be obtained from some class $[Q]$ in $\lf Q \rf$ by performing a special {\it surgery}
(Theorem \ref{thm: surgery}). Thus, for $[\ii_0] \in \lf \Qd \rf$, $\Upsilon_{[\ii_0]}$ has a {\it natural coordinate system} and its construction is {\it not} inductive anymore.
The surgery can be described as the way of obtaining $\Upsilon_{[Q^<]}$ and $\Upsilon_{[Q^>]}$ from $\Gamma_Q$ by {\it inserting} new vertices in each direction:
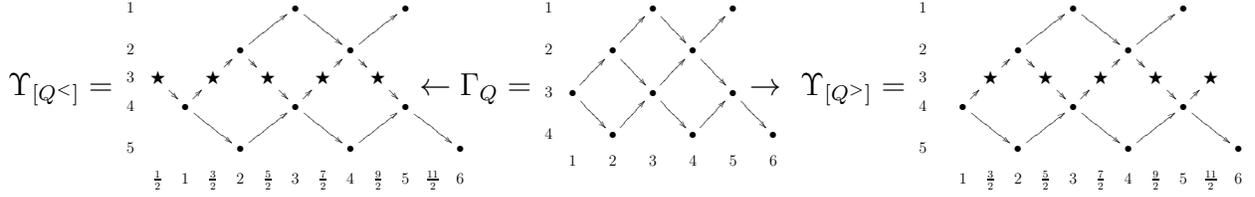
\begin{align*}
\Upsilon_{[Q^<]} = \raisebox{2.8em}{\scalebox{0.5}{\xymatrix@C=1ex@R=1ex{
1&&&&&& \bullet \ar@{->}[ddrr] &&&& \bullet  \\  & \\
2&&&&\bullet\ar@{->}[uurr]\ar@{->}[dr] && && \bullet \ar@{->}[uurr]\ar@{->}[dr]\\
3&\bigstar\ar@{->}[dr]&& \bigstar\ar@{->}[ur] && \bigstar\ar@{->}[dr] && \bigstar\ar@{->}[ur] && \bigstar\ar@{->}[dr]  \\
4&& \bullet\ar@{->}[ur]\ar@{->}[ddrr] &&&& \bullet\ar@{->}[ur]\ar@{->}[ddrr] &&&& \bullet\ar@{->}[ddrr]\\ & \\
5&&&& \bullet \ar@{->}[uurr] &&&& \bullet \ar@{->}[uurr] &&&& \bullet \\
&\frac{1}{2}& 1 &\frac{3}{2}& 2 &\frac{5}{2}& 3 &\frac{7}{2}& 4 &\frac{9}{2}& 5 &\frac{11}{2}& 6}}}
\ \hspace{-5ex} \gets
\Gamma_Q = \raisebox{2.8em}{\scalebox{0.5}{\xymatrix@C=1ex@R=1ex{
1&&&&& \bullet \ar@{->}[ddrr] &&&& \bullet\\  & \\
2&&&\bullet\ar@{->}[uurr]\ar@{->}[ddrr] && && \bullet \ar@{->}[uurr]\ar@{->}[ddrr]\\ & \\
3& \bullet\ar@{->}[uurr]\ar@{->}[ddrr] &&&& \bullet\ar@{->}[uurr]\ar@{->}[ddrr] &&&& \bullet\ar@{->}[ddrr]\\ & \\
4&&& \bullet \ar@{->}[uurr] &&&& \bullet \ar@{->}[uurr] &&&& \bullet \\
& 1 && 2 && 3 && 4 && 5 && 6 }}}
\hspace{-3ex} \to  \
\Upsilon_{[Q^>]} = \raisebox{2.8em}{\scalebox{0.5}{\xymatrix@C=1ex@R=1ex{
1&&&&&& \bullet \ar@{->}[ddrr] &&&& \bullet  \\  & \\
2&&&&\bullet\ar@{->}[uurr]\ar@{->}[dr] && && \bullet \ar@{->}[uurr]\ar@{->}[dr]\\
3&&& \bigstar\ar@{->}[ur] && \bigstar\ar@{->}[dr] && \bigstar\ar@{->}[ur] && \bigstar\ar@{->}[dr] && \bigstar \\
4&& \bullet\ar@{->}[ur]\ar@{->}[ddrr] &&&& \bullet\ar@{->}[ur]\ar@{->}[ddrr] &&&& \bullet\ar@{->}[ddrr] \ar@{->}[ur]\\ & \\
5&&&& \bullet \ar@{->}[uurr] &&&& \bullet \ar@{->}[uurr] &&&& \bullet \\
&& 1 &\frac{3}{2}& 2 &\frac{5}{2}& 3 &\frac{7}{2}& 4 &\frac{9}{2}& 5 &\frac{11}{2}& 6}}}
\end{align*}
Moreover, we can
conclude that the number of distinct classes in $\lf \Qd \rf$ as $2^{2n}$ via the two to one and onto map $ \PP : \lf \Qd \rf \to \lf Q \rf$ (Theorem \ref{thm: 22n}).

Secondly, we study the combinatorial properties of $\Upsilon_{[\ii_0]}$ ($ [\ii_0] \in \lf \Qd \rf$), called a {\it twisted AR-quiver},
by using $\Gamma_Q$ for $\PP([\ii_0])=[Q]$ and applying notions in \cite{Oh15E}. More precisely, we suggest the ways how to label the vertices of
$\Upsilon_{[\ii_0]}$ in an efficient way (Theorem \ref{them: comp for length k ge 0})
and how to assign a {\it coordinate} to each vertex, by using the combinatorial properties of $\Gamma_Q$ and the surgery (Algorithm \ref{Alg: surgery}).
As applications, we prove the {\it twisted additive property} for $\Upsilon_{[\ii_0]}$ (Theorem \ref{thm:twisted additive} and Proposition \ref{Prop:twisted_adapted_2}),
observe the interesting properties of $\PR$ with respect to $\prec_{[\ii_0]}$ (Section \ref{subsec: dis radi}) and
characterize each {\it $[\ii_0]$-minimal pair} $(\al,\be)$ of $\ga$ in terms of coordinates in $\Upsilon_{[\ii_0]}$ (Section \ref{subsec: minimal twisted}).

In the last part, we introduce new quiver
$\widehat{\Upsilon}_{[\ii_0]}$, called a {\it folded AR-quiver} by
folding $\Upsilon_{[\ii_0]}$ via the automorphism $\vee$:
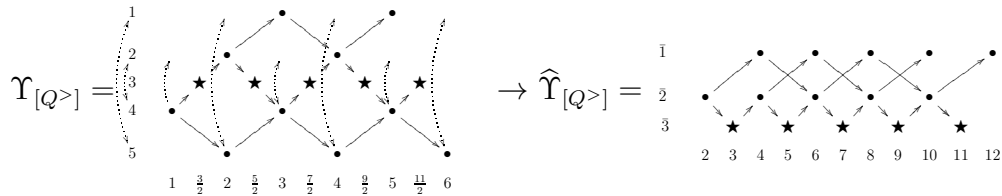
$$
\Upsilon_{[Q^>]} = \raisebox{2.8em}{\scalebox{0.5}{\xymatrix@C=1ex@R=1ex{
1&&&&&& \bullet \ar@{->}[ddrr] &&&& \bullet && && \\  & \\
2&&&&\bullet\ar@{->}[uurr]\ar@{->}[dr] && && \bullet \ar@{->}[uurr]\ar@{->}[dr] && &&\\
3&&& \bigstar\ar@{->}[ur] && \bigstar\ar@{->}[dr] && \bigstar\ar@{->}[ur] && \bigstar\ar@{->}[dr] && \bigstar \\
4 \ar@/^/@{<.>}[uu] &&  \bullet\ar@{->}[ur]\ar@{->}[ddrr] \ar@/^/@{.>}[uu] &&&& \bullet\ar@{->}[ur]\ar@{->}[ddrr] \ar@/^/@{<.>}[uu] &&&& \bullet\ar@{->}[ddrr] \ar@/^/@{<.>}[uu] \ar@{->}[ur]\\ & \\
5 \ar@/^1pc/@{<.>}[uuuuuu]  &&&& \bullet\ar@/^1pc/@{<.>}[uuuuuu] \ar@{->}[uurr] &&&& \bullet \ar@/^1pc/@{<.>}[uuuuuu]  \ar@{->}[uurr] &&&& \bullet \ar@/^1pc/@{<.>}[uuuuuu]  \\
&& 1 &\frac{3}{2}& 2 &\frac{5}{2}& 3 &\frac{7}{2}& 4 &\frac{9}{2}& 5 &\frac{11}{2}& 6}}}
\to
\widehat{\Upsilon}_{[Q^>]} = \raisebox{1.5em}{\scalebox{0.5}{\xymatrix@C=1ex@R=1ex{
\bar{1}&&&&\bullet\ar@{->}[ddrr] && \bullet \ar@{->}[ddrr] &&\bullet\ar@{->}[ddrr]&& \bullet && \bullet  \\  & \\
\bar{2}&&\bullet\ar@{->}[dr]\ar@{->}[uurr]&&\bullet\ar@{->}[uurr]\ar@{->}[dr] &&\bullet\ar@{->}[dr]\ar@{->}[uurr]&& \bullet \ar@{->}[uurr]\ar@{->}[dr] && \bullet \ar@{->}[dr]\ar@{->}[uurr]\\\
\bar{3}&&& \bigstar\ar@{->}[ur] && \bigstar\ar@{->}[ur] && \bigstar\ar@{->}[ur] && \bigstar\ar@{->}[ur] && \bigstar \\
&& 2 & 3 & 4 & 5 & 6 & 7 & 8 &9 & 10 & 11& 12}}}
$$
Then we
have a natural {\it folded} coordinate system for
$\widehat{\Upsilon}_{[\ii_0]}$. Surprisingly, any
$\widehat{\Upsilon}_{[\ii_0]}$ of $A_{2n+1}$ encodes the information
on the representation theory for $U_q'(B^{(1)}_{n+1})$ in the
following sense: {\rm (i)} By defining {\it folded distance
polynomial} $\widehat{D}_{k,l}^{[\ii_0]}(z)$ for $1\le k,l \le n$ on
$\widehat{\Upsilon}_{[\ii_0]}$, we can prove that
$\widehat{D}_{k,l}^{[\ii_0]}(z)=\widehat{D}_{k,l}^{[\ii'_0]}(z)$ for
any $[\ii_0],[\ii'_0]\in \lf\Qd\rf$ and the denominator formulas
$d^{B^{(1)}_{n+1}}_{k,l}(z)$ for $U_q'(B^{(1)}_{n+1})$ coincide with
$\widehat{D}_{k,l}^{[\ii_0]}(z)\times (z-q^{2n+1})^{\delta_{kl}}$
(Theorem \ref{thm: dist denom}) by using the results in Section \ref{subsec: dis radi}. {\rm (ii)} The Dorey's rule for
$U_q'(B^{(1)}_{n+1})$ can be interpreted as the {\it coordinates} of
$(\al,\be,\ga)$ in some $\widehat{\Upsilon}_{[\ii_0]}$ where
$(\al,\be)$ is a {\it $[\ii_0]$-minimal pair} for $\ga=\al+\be \in \PR$ of
type $A_{2n+1}$ (Theorem \ref{thm: Dorey}). Thus we can complete the twisted analogue of \eqref{eq:
concept} related to the representation theory for
$U_q'(B^{(1)}_{n+1})$:
\begin{align} \label{eq: concept2}
\scalebox{0.76}{\raisebox{2.8em}{\xymatrix@C=6ex@R=2ex{ && U_q'(B^{(1)}_{n+1}) \ar@{<->}[ddr]\ar@{<->}[ddl]  \\ \\
\{ [\redez] \} & \{ [\ii_0] \in \lf \Qd \rf \}  \ar@{_{(}->}[l] \ar@{<->}[rr] &&
\{ \widehat{\Upsilon}_{[\ii_0]}\ | \ [\ii_0] \in \lf \Qd \rf \}  & \ar@{->}[l] \{ \phi_{[\ii_0]}\vee \} }}}
\end{align}
where $\lf \Qd \rf$ and $\widehat{\Upsilon}_{[\ii_0]}$ are of type $A_{2n+1}$. Recently, in \cite{KKK13A,KKK13b,KKKO14D}, the denominator formulas $d^{A^{(t)}_{2n+1}}_{k,l}(z)$, the Dorey rule for $U_q'(A^{(t)}_{2n+1})$
and AR-quiver $\Gamma_Q$ of type $A_{2n+1}$ plays an important role for constructing the exact functor between the category $\mathcal{C}^{(t)}_Q$ of modules over $U_q'(A_{2n+1}^{(t)})$ $(t=1,2)$ (see \cite[\S 5.11]{HL11} and \cite[Definition 2.4]{KKKO14D}) and the category $\mathrm{Rep}(\mathsf{R})$ of finite dimensional modules over KLR-algebra $\mathsf{R}$ (\cite{KL09,KL11,R08}) of finite type $A_{2n+1}$.
With the results in this paper, we can define the category $\mathscr{C}_{\Qd}$ of modules over $U_q'(B_{n+1}^{(1)})$ (Definition \ref{def: [ii0] module category}).
Note that $\mathcal{C}^{(t)}_Q$ and  $\mathscr{C}_{\Qd}$ share similar property (see \cite{HL11}, \cite[Theorem 5.2]{Oh14A} and Corollary \ref{Cor: stable}):
\begin{eqnarray*} &&
\parbox{84ex}{
They are the smallest subcategories containing  $\{ V(\al_i) \ | \ i \in \Pi_{A_{2n+1}}  \}$ stable under taking subquotient, tensor product and extension.}
\end{eqnarray*}

In the forthcoming work \cite{OS16D}, we will study the analogue of \eqref{eq: concept} associated with the non-trivial Dynkin diagram automorphism of
type $D_{n+1}$. Also, the results in this paper will be crucial ingredients for observing the mysterious categorical relation
between quantum affine algebras $U_q'(A^{(t)}_{2n+1})$ $(t=1,2)$ and
$U_q'(B^{(1)}_{n+1})$ where generalized Cartan matrices of $U_q'(A^{(2)}_{2n+1})$ and $U_q'(B^{(1)}_{n+1})$ are transpose to each other \cite{KO16}.
Such observation was initiated in \cite{FH11A}.

\section{Foldable $r$-cluster point and its reduced expressions}

Let us consider the Dynkin diagram $\Delta_n$ of finite type
ADE$_n$, labeled by the index set $I_n$ of rank $n$, and their
$s$-order automorphisms $^{\vee(n,s)}$. Let $W_n$ be a Weyl group,
generated by the set of simple reflections $\{ s_i \ | \ i \in I_n
\}$ associated to $\Delta_n$ and ${}_{n}w_0$ be the longest element
of $W_n$. We usually drop $n$ if there is no danger of confusion.
Now we recall the non-trivial Dynkin diagram automorphisms $\vee$ whose orbits yield other Dynkin diagrams:
\begin{align}
B_{n+1}   &\longleftrightarrow
\big( A_{2n+1}: \xymatrix@R=0.5ex@C=4ex{ *{\circ}<3pt> \ar@{-}[r]_<{1 \ \ }  &*{\circ}<3pt>
\ar@{-}[r]_<{2 \ \ }  &   {}
\ar@{.}[r] & *{\circ}<3pt> \ar@{-}[r]_>{\,\,\,\ 2n} &*{\circ}<3pt>\ar@{-}[r]_>{\,\,\,\, 2n+1} &*{\circ}<3pt> }, \ i^{\vee(2n+1,2)} = 2n+2-i \big) \label{eq: B_n} \\
C_n     &\longleftrightarrow \left( D_{n+1}:
\raisebox{1em}{\xymatrix@R=0.5ex@C=4ex{
& & &  *{\circ}<3pt>\ar@{-}[dl]^<{ \  n} \\
*{\circ}<3pt> \ar@{-}[r]_<{1 \ \ }  &*{\circ}<3pt>
\ar@{.}[r]_<{2 \ \ } & *{\circ}<3pt> \ar@{.}[l]^<{ \ \ n-1}  \\
& & &   *{\circ}<3pt>\ar@{-}[ul]^<{\quad \ \  n+1} \\
}}, \ i^{\vee(n+1,2)} = \begin{cases} i & \text{ if } i \le n-1, \\ n+1 & \text{ if } i = n, \\ n & \text{ if } i = n+1. \end{cases} \right) \label{eq: C_n} \\
F_4   &\longleftrightarrow \left( E_{6}:
\raisebox{2em}{\xymatrix@R=3ex@C=4ex{
& & *{\circ}<3pt>\ar@{-}[d]_<{\quad \ \  6} \\
*{\circ}<3pt> \ar@{-}[r]_<{1 \ \ }  &
*{\circ}<3pt> \ar@{-}[r]_<{2 \ \ }  &
*{\circ}<3pt> \ar@{-}[r]_<{3 \ \ }  &
*{\circ}<3pt> \ar@{-}[r]_<{4 \ \ }  &
*{\circ}<3pt> \ar@{-}[l]^<{ \ \ 5 } }}, \begin{cases} 1^{\vee(6,2)}=5, \ 5^{\vee(6,2)}=1 \\ 2^{\vee(6,2)}=4, \ 4^{\vee(6,2)}=2, \\ 3^{\vee(6,2)}=3, \ 6^{\vee(6,2)}=6 \end{cases}  \right) \label{eq: F_4} \\
G_2   &\longleftrightarrow \left( D_{4}:\raisebox{1em}{
\xymatrix@R=0.5ex@C=4ex{
& &   *{\circ}<3pt>\ar@{-}[dl]^<{ \ 2} \\
*{\circ}<3pt> \ar@{-}[r]_<{1 \ \ }  &*{\circ}<3pt>
\ar@{-}[l]^<{4 \ \ }   \\
& &    *{\circ}<3pt>\ar@{-}[ul]^<{\quad \ \  3} \\
}}, \ \begin{cases} 1^{\vee(4,3)}=2, \ 2^{\vee(4,3)}=3, \
3^{\vee(4,3)}=1, \\ 4^{\vee(4,3)}=4. \end{cases} \right) \label{eq:
G_2}
\end{align}
Note that $^{\vee(n,s)}$ is not necessarily same as $^{*(n)}$ of $I$ induced by the longest element $w_0$. For example, if $n$ is even then $^{*(n)}$ for type $D_n$   is the identity so that it is not the same as $^{\vee(n,s)}$ in (\ref{eq: C_n}).  Also,  \eqref{eq: G_2} is not an involution. Usually we
abbreviate $^{(n,s)}$ to $^{(n)}$ or skip if there is no danger of
confusion.

Recall the Weyl group $W$ is a group generated by $\{ s_i  \ | \ i \in I\}$ subject to the following relations:
$$(1) \ \  s_i^2=1 \quad \quad (2) \ \  (s_is_j)^{d_{ij}+2}=1,$$
where $d_{ij}=d_{ji} \ (i,j \in I)$ denotes the number of arrows between the vertices $i$ and $j$ in $\Delta$.
The relations (2) are so-called {\it braid relations}. When $d_{ij}=d_{ji}=0$, they are simply named {\it commutation relations}.

Let $\langle I \rangle$ be the free monoid generated by $I$ and $\langle I \rangle_{{\mathrm r}}$ be the set of reduced words of $\langle I \rangle$
in the sense of Weyl group representation. We usually denote by $\bi$ for words and by $\ii$ for reduced words.

\begin{definition}
For a word $\bi$ and $J \subset I$, we define a subword $\bi_{|J}$ of $\bi=i_1\cdots i_l$ as follows:
$$  \bi_{|J} \seteq i_{t_1} \cdots i_{t_s} \text{ such that } \left\{
\begin{array}{l}
i_{t_x} \in J \text{ and } t_x < t_y \text{ for all } 1 \le x<y < s,\\
\text{if } t\not \in \{t_1, t_2, \cdots, t_s\} \text{ then } i_t\not \in J.
\end{array}\right.$$
\end{definition}

\begin{definition} \hfill
\begin{enumerate}
\item We say that two reduced words $\ii$ and $\jj$ representing $w \in W$ are {\it commutation equivalent}, denoted by $\ii
\sim \jj \in [\ii]$, if $\ii$ is obtained from $\jj$ by applying only commutation relations.
\item We say that an element $w \in W$ is {\it fully commutative} if the number of commutation classes for $w$ is the same as $1$. Thus
 we denote its unique commutation class by $[w]$.
\end{enumerate}
\end{definition}

\begin{definition} Fix a Dynkin diagram $\Delta$ of finite type.
For a commutation class $[\ii_0]$ representing $w_0$, we say that an index $i$ is a {\it sink $($resp. {\it source}$)$ of
$[\ii_0]$} if there is a reduced word $\jj_0 \in [\ii_0]$ of $w_0$ starting with $i$  (resp. ending with $i$).
\end{definition}

The following proposition may be known for experts on the research area but it is hard to find the same statement. In
\cite{OS15}, the authors proved the following proposition by using positive root systems and {\it convex total orders} induced from
reduced words for $w_0$.

\medskip

Recall the involution $^*$ on $I$ induced by $w_0$ (\cite{Bour}).

\begin{proposition} %\cite{OS15}
For the reduced word $\ii_0=i_1 i_2 \cdots i_{\N-1} i_\N$ of $w_0$, the word $\ii'_0= i_{\N}^* i_1 i_2 \cdots i_{\N-1}$
is again a reduced word of $w_0$ such that $[\ii'_0] \ne [\ii_0]$.
Similarly, $\ii''_0= i_2 \cdots i_{\N-1} i_\N i^*_1$ is again a
reduced word of $w_0$ such that $[\ii''_0] \ne [\ii_0]$.
\end{proposition}

By applying the above proposition, we can obtain new reduced words for $w_0$ by applying the operations so called
 {\it reflection functors} in the Auslander-Reiten theory, from a reduced word for $w_0$.

\begin{definition} \cite{OS15}
The right action of {\it reflection functor} $r_i$ on $[\ii_0]$ is defined by
$$[\ii_0]\, r_i =
\left\{
\begin{array}{ll}
[ i_2 \cdots i_\N i^*] & \text{ if $i$ is a sink and } \ii'_0=i_{ \ }   i_2 \cdots i_\N \in [\ii_0],\\
\ [\ii_0] &  \text{ if $i$ is not a sink of $[\ii_0]$}.
\end{array}
\right.
$$
On the other hand, the left action of {\it reflection functor} $r_i$ on $[\ii_0]$ is defined by
$$
r_i \, [\ii_0]=
\left\{
\begin{array}{ll}
[ i^* i_1 \cdots i_{\N-1} ] &  \text{ if $i$ is a source and } \ii'_0= i_1 \cdots i_{\N-1} i \in [\ii_0],\\
\ [\ii_0] &  \text{ if $i$ is not a source of $[\ii_0]$}.
\end{array}
\right.
$$
\end{definition}

For the word $\mathbf{w}= i_1 \cdots i_k$, the right  (resp. left) action of the reflection functor  $r_{\mathbf{w}}$ is defined by
$$ [\ii_0] \, r_{\mathbf{w}}=[\ii_0] r_{i_1} \cdots r_{i_k} \qquad (\text{resp. } r_{\mathbf{w}}[\ii_0]=r_{i_k}\cdots r_{i_1} [\ii_0]).$$

\begin{definition} \cite{OS15}  \label{def: ref equi}
Let  $[\ii_0]$ and $[\ii'_0]$ be two commutation classes. We say that $[\ii_0]$ and $[\ii'_0]$ are
{\it reflection equivalent} and write $[\ii_0]\overset{r}{\sim} [\ii'_0]$ if $[\ii_0']$ can be obtained from $[\ii_0]$ by a
sequence of reflection functors. The family of commutation classes
$\lf \ii_0 \rf \seteq \{\, [\ii_0]\, |\, [\ii_0]\overset{r}{\sim}[\ii_0']\, \}$ is called an {\it r-cluster point}.
\end{definition}

\begin{remark}
For any Dynkin diagram automorphism $\vee:I \to I$, $\vee$ is {\it compatible} with the involution $*$ in the sense that, for any $i\in I$, indices $i$ and $i^*$ are in the same orbit class $\bar{i}$ induced by $\vee.$
\end{remark}

\begin{definition} \label{Def:Coxeter_composition}
Fix an automorphism $\vee$ on $I$ which is compatible with the involution $^*$,
and denote by $\overline{I} \seteq \{\bar{i} \ | \ i \in I \}$
the set of orbit classes of $I$ induced by $^\vee$. For a reduced word
$\ii_0= i_1 i_2  \cdots  i_{\mathsf{N}}$ of $w_0$ and $\bar{k} \in \overline{I}$, define
$$\mathsf{C}^\vee_{[\ii_0]}(\bar{k}) = |\{ i_s \ | \ i_s \in \overline{k}, \  1 \le s \le \mathsf{N} \} |  \quad (\mathsf{N}=\ell(w_0)).$$
Here $|\vee|$ is the order of $\vee$. We call the composition
\[\mathsf{C}^{\vee}_{[\ii_0]}=\big (\mathsf{C}^\vee_{[\ii_0]}(\overline{1}),\mathsf{C}^\vee_{[\ii_0]}(\overline{2}), \ldots, \mathsf{C}^{\vee}_{[\ii_0]}(|\overline{I}|)\big),\]
 ordered by the smallest representative of orbit classes, the {\it $\vee$-Coxeter composition of $\lf \ii_0 \rf$}.
\end{definition}

Note that the notion of $\vee$-Coxeter composition is well-defined for each $r$-cluster point; i.e.,
$$\mathsf{C}^{\vee}_{\lf\ii_0\rf}\seteq \mathsf{C}^{\vee}_{[\ii_0']}=\mathsf{C}^{\vee}_{[\ii_0'']} \quad \text{ for any } \ii_0',\ii_0'' \in \lf \ii_0 \rf.$$

\begin{convention}
If there is no danger of confusion, we identify the index set $\bar{I}$ with the set $\{1,2,\cdots, |\bar{I}|\}\in \Z$ by the correspondence  $\bar{j}_s \mapsto s$ for $\bar{j}_s$ in Definition \ref{Def:Coxeter_composition}. Hence the sum (resp. subtraction) $\bar{j}_s + \bar{j}_{s'}$ (resp.  $\bar{j}_s - \bar{j}_{s'}$) is considered as the sum (resp. subtraction) on $\Z$ and, similarly,  the sum (resp. subtraction) between the index $\bar{j}_s$ and $t$ for $t\in \Z$ is defined by that on $\Z$.
\end{convention}

\begin{example} \label{ex: Coxeter composition}
\begin{enumerate}
\item For a Weyl group $W$ of type $A_5$, $\vee$ in \eqref{eq: B_n} and a reduced word
$$  \ii_0' = 5 \ 4 \ 3 \ 2 \ 1 \ 5 \ 4 \ 3 \ 2 \ 5 \ 4 \ 3 \ 5 \ 4 \ 5 $$
we have
\begin{itemize}
\item the number of $1$ and $5$ appearing in $\ii_0'$ is 6,
\item the number of $2$ and $4$ appearing in $\ii_0'$ is 6,
\item the number of $3$ appearing in $\ii_0'$ is 3.
\end{itemize}
Hence
$$\mathsf{C}^{\vee}_{\lf\ii_0'\rf} = (6,6,3).$$
\item For a Weyl group $W$ of type $A_5$, $\vee$ in \eqref{eq: B_n} and a reduced word
$$  \ii_0 = 1 \ 2 \ 3 \ 5 \ 4 \ 3 \ 1 \ 2 \ 3 \ 5 \ 4 \ 3 \ 1 \ 2 \ 3 $$
we have
$$\mathsf{C}^{\vee}_{\lf\ii_0\rf} = (5,5,5).$$
\end{enumerate}
\end{example}

\begin{definition} \label{def: vee-foldable}
For an automorphism $^\vee$ and an $r$-cluster point $\lf \redez \rf$ of finite ADE types,
we say that an $r$-cluster point $\lf \ii_0 \rf$ is {\it $\vee$-foldable} if
$$\mathsf{C}^{\vee}_{\lf \ii_0 \rf}(k)= \mathsf{C}^{\vee}_{\lf \ii_0 \rf}(l)  \qquad \text{ for any } k, l \in \overline{I}.$$
\end{definition}

If a $\vee$-foldable $r$-cluster point exists, then its $\vee$-Coxeter composition should be the following form:
The $\vee$-Coxeter composition of any $\vee$-foldable cluster point $\lf \ii_0 \rf$ is
\begin{align*}
\mathsf{C}^{\vee}_{\lf \ii_0 \rf}= \begin{cases}
( \underbrace{2n+1, \ldots ,2n+1}_{ n+1\text{-times}} ) & \text{ if } \vee=\eqref{eq: B_n}, \\
( \underbrace{n+1, \ldots ,n+1}_{ n\text{-times}} ) & \text{ if } \vee=\eqref{eq: C_n}, \\
\hspace{4.5ex} (9,9,9,9)& \text{ if } \vee=\eqref{eq: F_4},\\
\hspace{6.5ex} (6,6)& \text{ if } \vee=\eqref{eq: G_2}.
\end{cases}
\end{align*}

\begin{remark}
We remark here an interesting relationships between the dual Coxeter numbers $\mathsf{h}_X^\vee$ for type $X$:
\begin{align} \label{eq: dual Coxeter}
\mathsf{h}_{B_{n+1}}^\vee=\mathsf{h}_{A_{2n}}^\vee=2n+1, \ \mathsf{h}_{C_n}^\vee=\dfrac{1}{2}\mathsf{h}_{D_{n+1}}^\vee=\mathsf{h}^\vee_{A_n}=n+1.
\end{align}
\end{remark}

\section{Twisted Coxeter elements and induced reduced expressions}

\begin{convention}
We can view $W$ as a
subgroup of $GL(\C\Phi)$ generated by the set of simple reflections $\{ s_i \ | \ i \in I \}$.
In this case, we use the term ``reduced expression" instead of ``reduced word". Moreover, we sometimes abuse the notation
$\ii$ to represent reduced expressions.
\end{convention}

Let $\sigma \in GL(\C\Phi)$ be a linear transformation of finite order
which preserves some base $\Pi$ of $\Phi$. Hence $\sigma$ preserves
$\Phi$ itself and normalizes $W$ and so $W$ acts by conjugation on the
coset $W\sigma$.

\begin{definition} \label{def: twisted Coxeter} \hfill
\begin{enumerate}
\item Let $\{ \Pi_{i_1}, \ldots,\Pi_{i_k} \}$ be the all orbits of $\Pi$ in
$\Phi$ with respect to $\sigma$. For each $r\in \{1, \cdots, k \}$, choose $\alpha_{i_r} \in \Pi_{i_r}$ arbitrarily,
and let $s_{i_r} \in W$ denote the corresponding reflection. Let $w$
be the product of $s_{i_1}, \ldots , s_{i_k}$ in any order. The
element $w\sigma \in W\sigma$ thus obtained is called a {\it
$\sigma$-Coxeter element}.
\item If $\sigma$ in (1) is $\vee$ in \eqref{eq: B_n}, \eqref{eq: C_n}, \eqref{eq: F_4},
then  {\it $\sigma$-Coxeter element} is also called a {\it twisted Coxeter element}.
\item If $\sigma$ in (1) is $\vee$ or $\vee^2$ in \eqref{eq: G_2}
then  {\it $\sigma$-Coxeter element} is also called a {\it triply twisted Coxeter element}.
\end{enumerate}
\end{definition}

\begin{example}
Take $\sigma$ as $\vee$ in \eqref{eq: B_n} for $A_{5}$. There are
$12$ distinct twisted Coxeter elements in $W \vee$ given as follows:
\begin{align*}
& s_1s_2s_3\vee, \ s_2s_1s_3\vee, \ s_3s_1s_2\vee, \ s_3s_2s_1\vee, \ s_5s_2s_3\vee, \ s_3s_2s_5\vee, \\
& s_1s_4s_3 \vee, \ s_3s_1s_4\vee, \ s_5s_4s_3\vee, \ s_4s_5s_3\vee,
\ s_3s_5s_4\vee, \ s_3s_4s_5\vee.
\end{align*}
\end{example}

\begin{example}
Take $\sigma$ as $\vee$ in \eqref{eq: C_n} for $D_{4}$: \\
There are $8$ distinct twisted Coxeter elements in $W \vee$ given as
follows:
\begin{align*}
& s_1s_2s_3\vee, \ s_1s_3s_2\vee, \ s_2s_1s_3\vee, \ s_3s_2s_1\vee, \\
& s_1s_2s_4\vee, \ s_1s_4s_2\vee, \ s_2s_1s_4\vee, \ s_4s_2s_1\vee.
\end{align*}
\end{example}

\begin{remark} \label{rem: basic prop twist}
Note that $w$ in ${\rm (1)}$ of Definition \ref{def: twisted Coxeter} is fully commutative of length $|\overline{I}|$. Thus, for $i_r$ contained in
the orbit of an extremal vertex of $\Delta$, $s_{i_r}$ is a sink or a source of $[w]$.
\end{remark}

In the proof of the following proposition, we shall use the notion of Coxeter elements and their properties will be reviewed in Subsection \ref{subsec: Adapted cluster}:

\begin{proposition}\label{prop: number tCox elts} \hfill
\begin{enumerate}
\item[{\rm (1)}] The number of twisted Coxeter elements associated to \eqref{eq: B_n} is $4\times 3^{n-1}.$
\item[{\rm (2)}] The number of twisted Coxeter elements associated to \eqref{eq: C_n} is $2^n.$
\item[{\rm (3)}] The number of twisted Coxeter elements associated to \eqref{eq: F_4} is $24.$
\item[{\rm (4)}] The number of triply twisted Coxeter elements associated to \eqref{eq: G_2} is $12$.
\end{enumerate}
\end{proposition}

\begin{proof}
(1) For $A_3$ type, there are four twisted Coxeter elements $s_1s_2 \vee=[1\ 2]\vee$,  $s_2s_1\vee=[2\ 1]\vee$, $s_3s_2\vee=[3\ 2]\vee$, $s_2s_3\vee=[2\ 3]\vee$.

Assume that our assertion is true for type $A_{2n-1}$. Take a twisted Coxeter element
\[  [i_1\ i_2 \cdots i_{n}]\vee  \text{ of type } A_{2n-1}.\]
Let $J= \{2, 3, \cdots, 2n\} \subset I_{2n+1}$. We shall find twisted Coxeter element $\ii\vee$ of $A_{2n+1}$ satisfying \[\ii_{|J}= i_1^+ \ i_2^+\ \cdots i_n^+, \text{ where } i_k^+=i_k+1\] for $k=1, \cdots, n.$
\begin{enumerate}
\item[{\rm (a)}]
Assume that there is $t\in \{1, \cdots, n\}$ such that $i_t^+=2.$ Note that, in this case, $\ii$ does not have the index $2n$. Then $[\ii]$ can be three of following classes:
\[ \ [ 1\ \ii_{|J}]  , \quad [\ii_{|J}\ 1], \quad [2n+1\ \ii_{|J}].\ \]
Here we use the property that $s_1s_{i^+_k} = s_{i^+_k} s_{1}$ for $k\in \{1, \cdots, n\}\backslash\{t\}$ and $s_{2n+1} s_{i^+_k}= s_{i^+_k}s_{ 2n+1}$ for $k\in \{1, \cdots, n\}.$
\item[{\rm (b)}]
Assume that there is $t\in \{1, \cdots, n\}$ such that $i_t^+=2n.$ Note that, in this case, $\ii$ does not have the index $2$. Then $[\ii]$ can be three of following classes:
\[ \ [ 2n+1\ \ii_{|J}]  , \quad [\ii_{|J}\ 2n+1], \quad [1\ \ii_{|J}]\]
Here we use the property that $s_{2n+1}s_{i^+_k} = s_{i^+_k} s_{2n+1}$ for $k\in \{1, \cdots, n\}\backslash\{t\}$ and $s_{1} s_{i^+_k}= s_{i^+_k}s_{1}$ for $k\in \{1, \cdots, n\}.$
\end{enumerate}

Hence we can get three distinct twisted Coxeter elements of type $A_{2n+1}$ induced from a twisted Coxeter element of type $A_{2n-1}.$
Therefore the number of twisted Coxeter elements of type $A_{2n+1}$ is $4\times 3^{n-1}$ by Remark \ref{rem: basic prop twist}.\\
(2) There exists a canonical one-to-one correspondence between Coxeter elements of type $A_n$ and twisted Coxeter elements of $D_{n+1}$ with $s_n$ defined by
\[ s_{i_1}s_{i_2}\cdots s_{i_n} \leftrightarrow s_{i_1}s_{i_2}\cdots s_{i_n}\vee.\]
On the other hand, there is the one-to-one correspondence between  twisted Coxeter elements of $D_{n+1}$ with $s_n$ and  twisted Coxeter elements of $D_{n+1}$ with $s_{n+1}$ defined by
\[s_{i_1}s_{i_2}\cdots s_{i_n}\vee \leftrightarrow s_{i'_1}s_{i'_2}\cdots s_{i'_n}\vee\]
where $i'_k= i_k+\delta_{i_k, n}.$
Since we know the number of Coxeter elements of $A_n$ is $2^{n-1}$, the number of twisted Coxeter elements of $D_{n+1}$ is $2^n$ by Remark \ref{rem: basic prop twist}.

\noindent
(3) We list all the twisted Coxeter elements of $E_6$:
\begin{align*}
& s_6s_5s_4s_3\vee, \  s_6s_4s_5s_3\vee, \  s_6s_3s_5s_4\vee, \  s_6s_3s_4s_5\vee, \  s_6s_1s_4s_3\vee, \  s_6s_3s_4s_1\vee,
 s_6s_5s_2s_3 \vee, \  s_6s_3s_5s_1\vee, \\ &   s_6s_1s_2s_3\vee, \  s_6s_2s_1s_3\vee, \  s_6s_3s_1s_2\vee, \  s_6s_3s_2s_1\vee,
 s_5s_4s_3 s_6\vee, \ s_4s_5s_3 s_6\vee, \ s_3s_5s_4 s_6\vee, \ s_3s_4s_5 s_6\vee, \\ & s_1s_4s_3 s_6\vee, \ s_3s_4s_1 s_6\vee,
 s_5s_2s_3 s_6 \vee, \ s_3s_5s_1 s_6\vee, \ s_1s_2s_3 s_6\vee, \ s_2s_1s_3 s_6\vee, \ s_3s_1s_2 s_6\vee, \ s_3s_2s_1 s_6\vee .
\end{align*}
(4) We list all the triply twisted Coxeter elements of $D_4$ for \eqref{eq: G_2}:
\begin{align*}
& s_1s_4\vee, \ s_2 s_4\vee, \ s_3 s_4\vee,\ s_4 s_1\vee, \ s_4 s_2\vee, \ s_4 s_3\vee \\
& s_1s_4\vee^2, \ s_2 s_4\vee^2, \ s_3 s_4\vee^2,\ s_4 s_1\vee^2, \ s_4 s_2\vee^2, \ s_4 s_3\vee^2.
\end{align*}
\end{proof}

\section{Adapted cluster point of type $ADE$, convex orders and (combinatorial) Auslander-Reiten quiver}

In this section, we shortly review the theories on the reduced expressions of $w_0$ which are adapted to some Dynkin quiver $Q$,
convex orders on the set of positive roots of finite types and (combinatorial) Auslander-Reiten quivers.
For the precise reference, we mainly refer \cite{ARS,ASS,Gab80,Oh14A,Oh15E,OS15}.

\subsection{Adapted cluster point of type ADE} \label{subsec: Adapted cluster}
In this subsection, we consider the case when the Dynkin diagram $\Delta$ is simply laced; that is, $\Delta$
is of type ADE$_n$.

\medskip

A Dynkin quiver $Q$ is obtained by orienting all edges of $\Delta$. We denote by $\Phi^+$ the set of all positive roots
and $\Pi=\{ \al_i \ | \ i \in I\, \}$ the set of simple roots associated to $\Delta$.
The Gabriel theorem
states that there exists a bijection between {\rm (i)}  the set of isomorphism classes of indecomposable modules ${\rm Ind}Q$ in the finite dimensional module category over
the path algebra $\mathbb{C}Q$ and {\rm (ii)} $\Phi^+$, via taking dimension vectors on ${\rm Ind}Q$ (\cite{Gab80}).

We say that a vertex is a {\it sink} (resp. {\it source}) of $Q$ if and only if there are only entering arrows into (resp. exiting arrows out of) it.
Let $i$ be a sink or a source of $Q$. We denote by $i \ Q$ the quiver obtained from $Q$ by reversing the orientation of each arrow incident with $i$ in $Q$.
For a reduced word $\ii$, we say that $\ii=i_1 i_2 \cdots i_l$ is {\it adapted to} $Q$ if
\begin{align*}
\text{ $i_k$ is a sink of $ i_{k-1} \cdots i_2 i_1 \ Q$ for all $1 \le k \le l$.}
\end{align*}

The followings are well-known and we record here for later use:

\begin{theorem} \label{thm: Qs} \hfill
\begin{enumerate}
\item[{\rm (1)}] A reduced word $\ii_0$ of $w_0$ is adapted at most one Dynkin quiver $Q$.
\item[{\rm (2)}] For each Dynkin quiver $Q$, there is a reduced word $\ii_0$ of $w_0$ adapted to $Q$.
\item[{\rm (3)}] Reduced expressions $\ii_0$ and $\ii_0'$ are adapted to $Q$ if and only if $\ii_0 \sim \ii_0'$ and $\ii_0$ is adapted to $Q$.
Hence the commutation class $[Q]$ of $w_0$ adapted to $Q$ is well-defined.
\item[{\rm (4)}] For every commutation class $[Q]$, there exists a unique Coxeter element $\phi_Q$. We have {\rm (i)} $\phi_Q$ is a product of all simple reflections,
{\rm (ii)} $\phi_Q$ is fully commutative, {\rm (iii)} $\phi_Q \cdot Q =Q$, {\rm (iv)} $\phi_Q$ is adapted to $Q$.
\item[{\rm (5)}] For every Coxeter element $\phi$, there exists a unique Dynkin quiver $Q$ such that $\phi=\phi_Q$.
\end{enumerate}
\end{theorem}

\vskip -1em

\begin{eqnarray} &&
\parbox{79ex}{
Furthermore, all commutation classes in $\{ [Q] \}$ are reflective equivalent via reflection functors and form an $r$-cluster point $\lf Q \rf$.
Since there are $2^{|I|-1}$-many Dynkin quivers, the number of commutation classes inside of $\lf Q \rf$ and of Coxeter elements are the same as $2^{|I|-1}$.
More precisely, we have one-to-one correspondences
$$
\xymatrix@C=8ex@R=4ex{
\{ [Q] \} \ar@{<->}[r]^{1-1} & \{ \phi_Q \} \ar@{<->}[r]^{1-1} & \{ Q \}}
$$
between order $2^{|I|-1}$ sets and  call $\lf Q \rf$ {\it the adapted cluster point}.
}\label{eq: 2n-1many}
\end{eqnarray}

\subsection{Convex orders on $\Phi^+$}
Recall that $W$ acts on the set of all roots $\Phi$. For $w \in W$, consider the subset $\Phi(w)$ of $\Phi$ defined as follows:
\begin{align} \label{eq: Phiw}
\Phi(w) \seteq \{ \beta \in \Phi^+ \ | \ w^{-1}(\beta) \in \Phi^- \}.
\end{align}

\begin{lemma} \cite{Bour} For $w \in W$, the followings hold:
\begin{enumerate}
\item[{\rm (1)}] $\Phi(w) = \{ s_{i_1}s_{i_2} \cdots s_{i_{k-1}}(\alpha_{i_k}) \ | \ k =1,\ldots, l\}$ for any reduced word $\ii=i_1i_2 \cdots i_l$ of $w$.
\item[{\rm (2)}] $|\Phi(w)|=\ell(w)$ and hence $\Phi(w_0)=\Phi^+$.
\end{enumerate}
\end{lemma}

In this paper, we mainly focus on when $w=w_0$. By the above lemma, we can define a total order $<_{\ii_0}$ on $\Phi^+$ depending on $\ii_0=i_1i_2 \cdots i_\N$ of $w_0$ as follows:
$$ \be^{\ii_0}_k \seteq s_{i_1}s_{i_2} \cdots s_{i_{k-1}}(\alpha_{i_k}) \ (1 \le k \le \N) \quad \text{ and } \quad \be^{\ii_0}_k <_{\ii_0} \be^{\ii_0}_l \iff k <l.$$
Interestingly, the order $<_{\ii_0}$ is convex in the following sense.

\begin{definition}
We say that an order $\prec$ on $\Phi^+$ is {\it convex} if it satisfies the following property: For $\al,\be \in \Phi^+$ satisfying $\al+\be \in \Phi^+$, we have either
\begin{align*}
\al \prec \al+\be \prec \be \quad\text{ or }\quad \be \prec \al+\be \prec \al.
\end{align*}
\end{definition}

By considering  $<_{\ii'_0}$ for all $\ii_0' \in [\ii_0]$, we can construct a partial order $\prec_{[\ii_0]}$ depending only on $[\ii_0]$: For $\al,\be \in \Phi^+$,
\begin{align}
\al \prec_{[\ii_0]}  \be \quad \text{ if } \al <_{\ii_0'} \be \text{ for all } \ii_0' \in [\ii_0].
\end{align}
The order $\prec_{[\ii_0]}$ is well-defined and also {\it convex} due to the following observation: For reduced expressions
$\ii_0 \sim \ii_0'$ and any pair $(\al,\be)$ such that $\al+\be \in \Phi^+$, we have
\[ \al <_{\ii_0} \al+\be <_{\ii_0} \be \ \text{ if  and only if } \ \al <_{\ii'_0} \al+\be <_{\ii'_0} \be.\]

\subsection{(Combinatorial) Auslander-Reiten quivers}

\begin{definition} For a Dynkin quiver $Q$ of finite type $ADE$, let us choose any reduced word $\ii_0$ in $[Q]$. The quiver
$\Gamma_Q=(\Gamma_Q^0,\Gamma_Q^1)$ is called the {\it Auslander-Reiten quiver} (AR-quiver) if
\begin{enumerate}
\item each vertex in $\Gamma_Q^0$ corresponds to an isomorphism class $[M]$ in ${\rm Ind}Q$,
\item an arrow $[M] \to [M']$ in $\Gamma_Q^1$ represents an irreducible morphism $M \to M'$.
\end{enumerate}
\end{definition}

By the Gabriel theorem \cite{Gab80}, we can replace the vertex set ${\rm Ind}Q$ of $\Gamma_Q$ with $\Phi^+$. On the other hand, we can construct the AR-quiver $\Gamma_Q$
in a purely combinatorial way by using only its Coxeter element $\phi_Q$.

\medskip

(A) As we observed in \eqref{eq: Phiw}, $\Phi(\phi_Q)$ is well-defined for the Coxeter element $\phi_Q=s_{i_1}s_{i_2}\cdots s_{i_n}$:
$$  \Phi(\phi_Q)=\{ \be^{\phi_Q}_1=\al_{i_1}, \be^{\phi_Q}_2=s_{i_1}(\al_{i_1}),\ldots,\be^{\phi_Q}_n =s_{i_1}\cdots s_{i_{n-1}}(\al_{i_n}) \}.$$
Furthermore, the residue $i_k$ on $\be^{\phi_Q}_k$ is also well-defined and does not depend on the choice of reduced word for $\phi_Q$. (Recall that $\phi_Q$ is fully commutative.)

(B) The {\it height function} $\xi$ associated to $Q$ is a map on $Q$ satisfying $\xi(j)=\xi(i)+1$ if there exists an arrow $i \to j$ in $Q$. Note that the connectedness of
$Q$ implies the uniqueness of $\xi$ up to constant. In this paper, we fix $\xi$.

\medskip

Hence we can assign $\be^{\phi_Q}_k$ to $(i_k,\xi(i_k)) \in I \times \Z$ which does depend only on $Q$ and hence $\phi_Q$:

\begin{algorithm}
The AR-quiver $\Gamma_Q$, whose set of vertices is also identified with a subset of  $I \times \Z$,
can be constructed by the following injective map $\Omega_Q:\Phi \to I \times \Z$ in an inductive way
$($cf. \cite{HL11}$):$
\begin{enumerate}
\item[{\rm (1)}] $\Omega_Q(\be^{\phi_Q}_k) \seteq (i_k,\xi(i_k))$.
\item[{\rm (2)}] If $\Omega_Q(\be)$ is already assigned as $(i,p)$ and $\phi_Q(\be) \in \Phi^+$, then $$\Omega_Q(\phi_Q(\be)) = (i,p-a_{ii})=(i,p-2).$$
\item[{\rm (3)}] For $(i,p),(j,q) \in {\rm Im}(\Omega_Q)$, there exists an arrow  $(i,p) \to (j,q)$ if and only if $i$ and $j$ are adjacent in $\Delta$ and
$$p-q=a_{ij}=-1.$$
\end{enumerate}
\end{algorithm}
\noindent
For $\be$ with $\Omega_Q(\be)=(i,p)$, we call $i$ the residue of $\beta$ with respect to $[Q]$ and $(i,p)$ is a {\it coordinate of $\be$} in $\Gamma_Q$.

Also,  the AR-quiver $\Gamma_Q$  has the following property.

\begin{proposition} \cite{B99, Gab80, R80} \label{Prop:AR}
Let us denote by $\mathsf{h}^\vee$ the dual Coxeter number associated to $Q$. For an index $i\in I$, let
\[ r^Q_i= \frac{\mathsf{h}^\vee+a_i-b_i}{2}\]
where $a^Q_i$ is the number of arrows in $Q$ between $i$ and $i^*$ directed toward $i$ and $b^Q_i$ is the number of arrows in $Q$ between $i$ and $i^*$
directed toward $i^*$. Then the number of vertex in $\Gamma_Q \cap \{i\} \times \Z$ is $r^Q_i$ and
\[ \Gamma_Q \cap \{i\} \times \Z= \{ \, (i, \xi(i)-2k)\, | \, k=0, \cdots, r_i-1\}.\]
\end{proposition}

The reflection functor $r_i: [Q] \mapsto [Q]r_i$ for a sink $i$ of $[Q]$ can be understood by the map from $\Gamma_Q$ to $\Gamma_{i Q}$ described
using coordinates and the dual Coxeter number $\mathsf{h}^\vee$ as follows:
\begin{algorithm} \label{alg: Ref Q}
Let $\mathsf{h}^\vee$ be the dual Coxeter number associated to $Q$.
\begin{enumerate}
\item[{\rm (A1)}] Remove the vertex $(i,p)$ such that $\Omega_Q(\al_i)=(i,p)$ and the arrows entering into $(i,p)$ in $\Gamma_Q$.
\item[{\rm (A2)}] Add the vertex $(i^*,p-\mathsf{h}^\vee)$ and the arrows to all $(j,p-\mathsf{h}^\vee+1)$ for $j$ adjacent to $i^*$ in $\Delta$.
\item[{\rm (A3)}] Label the vertex $(i^*,p-\mathsf{h}^\vee)$ with $\al_i$ and change the labels $\be$ to $s_i(\be)$ for all $\be \in \Gamma_Q \setminus \{\al_i\}$.
\end{enumerate}
\end{algorithm}

For $\be \in \Phi^+$ with $\phi_Q(\be) \in \Phi^+$, the AR-quiver $\Gamma_Q$ satisfies {\it the additive property} in the following sense:
\begin{align}\label{eq: addtive property}
\be+\phi_Q(\be) = \sum_{\Omega_Q(\ga)=(j,p-1)} \ga,
\end{align}
where $\Omega_Q(\be)=(i,p-2)$ and $j$ runs over all vertices  with coordinates $(j,p-1)$ such that $j$ is adjacent with $i$ in $\Delta$.

Interestingly, $\Gamma_Q$ can be also understood as a visualization of $\prec_{Q} \seteq \prec_{[Q]}$ and is closely related to $[Q]$:

\begin{theorem} \cite{B99,OS15,R96}
\begin{enumerate}
\item[{\rm (1)}] $\al \prec_Q \be$ if and only if there exists a path from $\be$ to $\al$ inside of $\Gamma_Q$.
\item[{\rm (2)}] By reading residues of vertices in a way {\it compatible with} arrows, we can obtain all reduced words $\ii_0 \in [Q]$.
\end{enumerate}
\end{theorem}

\begin{example}
The AR-quiver $\Gamma_Q$ associated to $\xymatrix@R=3ex{ *{ \bullet }<3pt> \ar@{<-}[r]_<{1}  &*{\bullet}<3pt>
\ar@{->}[r]_<{2}  &*{\bullet}<3pt>
\ar@{<-}[r]_<{3} &*{\bullet}<3pt>
\ar@{<-}[r]_<{4}  & *{\bullet}<3pt> \ar@{-}[l]^<{\ \ 5}
}$ with the height function such that $\xi(1)=0$ is given as follows:
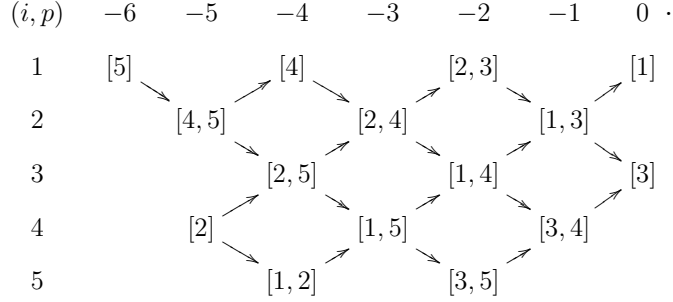
\[ \scalebox{0.84}{\ \xymatrix@C=2ex@R=1ex{
( i,p ) &-6&-5&-4&-3&-2&-1&0\\
1& [5]\ar@{->}[dr] & & [4]\ar@{->}[dr] & &[2,3] \ar@{->}[dr] & &[1]  \\
2&&[4,5]\ar@{->}[ur]\ar@{->}[dr]& &[2,4] \ar@{->}[ur]\ar@{->}[dr]&& [1,3] \ar@{->}[ur]\ar@{->}[dr]\\
3&&& [2,5]\ar@{->}[ur]\ar@{->}[dr] && [1,4]\ar@{->}[ur]\ar@{->}[dr] && [3]\\
4&&  [2] \ar@{->}[ur]\ar@{->}[dr]  && [1,5] \ar@{->}[ur] \ar@{->}[dr] && [3,4] \ar@{->}[ur]\\
5&&&[1,2] \ar@{->}[ur] &&[3,5] \ar@{->}[ur]}}.
\]
Here $[a,b]$ $(1 \le a,b \le 5)$ denotes the positive root $\sum_{k=a}^b \al_k$.
\end{example}

With the above theorem, we can extend the correspondences in \eqref{eq: 2n-1many} with $\{ \prec_Q \}$ and $\{ \Gamma_Q \}$:
\begin{align}\label{eq: 2n-1many 2}
\raisebox{3.6em}{\xymatrix@C=8ex@R=4ex{
& \{ \prec_Q \} \ar@{<->}[dl]_{1-1} \ar@{<->}[d]_{1-1} \ar@{<->}[dr]^{1-1}  \\
\{ [Q] \} \ar@{<->}[r]^{1-1} & \{ \phi_Q \} \ar@{<->}[r]^{1-1} & \{ Q \} \\
& \{ \Gamma_Q \} \ar@{<->}[ul]^{1-1}\ar@{<->}[ur]_{1-1} \ar@{<->}[u]^{1-1}}} \text{for any $[Q] \in \lf Q \rf$}.
\end{align}

Note that there are reduced words which are not adapted to any Dynkin quivers (see $\ii_0$ in Example \ref{ex: Coxeter composition}). In \cite{OS15}, the authors introduced combinatorial AR-quivers $\Upsilon_{[\ii_0]}$,
for {\it any} commutation class $[\ii_0]$ of any finite type. The quiver can be understood as a generalization of $\Gamma_Q$:

\begin{algorithm} \label{Alg_AbsAR}
Let $\ii_0=(i_1 i_2 i_3 \cdots i_{\N})$ be
a reduced expression of the longest element $w_0\in W$. Then we can label
$\PR$ as follows:
\begin{align}\label{eq: computing for label}
\beta^{\ii_0}_k = s_{i_1}\cdots s_{i_{k-1}}(\al_{k}) \quad \text{ for } 1 \le k \le \N.
\end{align}
The quiver $\Upsilon_{\ii_0}=(\Upsilon^0_{\ii_0},
\Upsilon^1_{\ii_0})$ associated to $\ii_0$ is constructed in the
following algorithm:
\begin{enumerate}
\item[{\rm (Q1)}] $\Upsilon_{\ii_0}^0$ consists of $\N$ vertices labeled by $\beta^{\ii_0}_1, \cdots, \beta^{\ii_0}_{\N}$.
\item[{\rm (Q2)}] The quiver $\Upsilon_{\ii_0}$ consists of $|I|$ residues and each vertex $\beta^{\ii_0}_k\in \Upsilon^0_{\ii_0}$ lies in the $i_k$-th residue.
\item[{\rm (Q3)}] There is an arrow from $\beta^{\ii_0}_k$ to $\beta^{\ii_0}_j$ if the followings hold:
\begin{enumerate}
\item[{\rm (Ar1)}] two vertices $i_k$ and $i_j$ are connected in the Dynkin diagram,
\item[{\rm (Ar2)}] $ j= \max \{ j'\, |\, j' <k, \, i_{j'}=i_j \} $,
\item[{\rm (Ar3)}] $ k= \min \{ k'\, |\, k' >j, \, i_{k'}=i_k \} $.
\end{enumerate}
\item[{\rm (Q4)}] Assign the color $m_{jk}=-(\alpha_{i_j}, \alpha_{i_k})$ to each arrow $\beta^{\ii_0}_k\to \beta^{\ii_0}_j$ in {\rm (Q3)}; that is,
$\beta^\redex_k \xrightarrow{m_{jk}} \beta^\redex_j$.  Replace
$\xrightarrow{1}$ by $\rightarrow$,  $\xrightarrow{2}$ by
$\Rightarrow$ and  $\xrightarrow{3}$ by  $\Rrightarrow$.
\end{enumerate}
\end{algorithm}

\begin{theorem} \cite{OS15} \label{thm: OS14}
Let us choose any commutation class $[\ii_0]$ and a reduced word $\ii_0$ in $[\ii_0]$.
\begin{enumerate}
\item[{\rm (1)}] The construction of $\Upsilon_{\ii_0}$ does depend only on its commutation class $[\ii_0]$ and hence $\Upsilon_{[\ii_0]}$
is well-defined.
\item[{\rm (2)}] $\al \prec_{[\ii_0]} \be$ if and only if there exists a path from $\be$ to $\al$ in $\Upsilon_{[\ii_0]}$.
\item[{\rm (3)}] By defining the notion, standard tableaux of shape $\Upsilon_{[\ii_0]}$, every reduced word $\ii'_0 \in [\ii_0]$
corresponds to a standard tableau of shape $\Upsilon_{[\ii_0]}$ and can be obtained by reading residues in a way compatible with the tableau.
\item[{\rm (4)}] When $[\ii_0]=[Q]$, $\Upsilon_{[Q]}$ is isomorphic to $\Gamma_Q$ as quivers.
\end{enumerate}
\end{theorem}

\begin{example} \label{D4 non-adapted D-1}
Let $\ii_0 = (123124123124)$ be a reduced word of $w_0$ of type $D_4$. Note that $\ii_0$ is not adapted to any Dynkin quiver of
type $D_4$. We can draw the combinatorial AR-quiver
$\Upsilon_{[\ii_0]}$ as follows:
$$\scalebox{0.84}
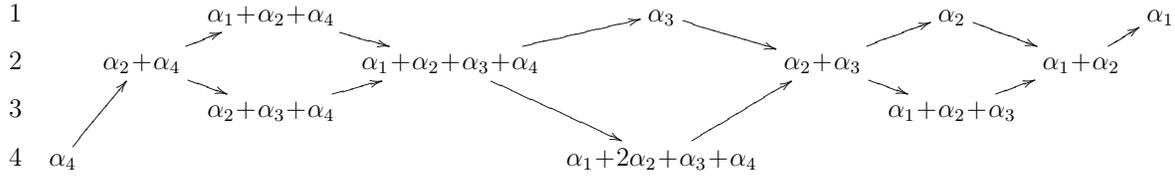
{\xymatrix@C=1ex@R=1ex{
1&&& \alpha_1 \shp \alpha_2 \shp \alpha_4\ar@{->}[dr] && \alpha_3 \ar@{->}[dr]&& \alpha_2 \ar@{->}[dr] && \alpha_1\\
2&& \alpha_2 \shp \alpha_4 \ar@{->}[ur]\ar@{->}[dr] && \alpha_1 \shp
\alpha_2 \shp \alpha_3 \shp \alpha_4 \ar@{->}[ur]\ar@{->}[ddr]
 && \alpha_2 \shp \alpha_3 \ar@{->}[ur]\ar@{->}[dr]&& \alpha_1 \shp \alpha_2 \ar@{->}[ur]\\
3& && \alpha_2 \shp \alpha_3 \shp \alpha_4 \ar@{->}[ur]&& && \alpha_1 \shp \alpha_2 \shp \alpha_3 \ar@{->}[ur] \\
4& \alpha_4 \ar@{->}[uur] &&  && \alpha_1 \shp 2\alpha_2 \shp
\alpha_3 \shp \alpha_4 \ar@{->}[uur]&& }}
$$
\end{example}

\section{Twisted adapted cluster point of type $A_{2n+1}$}

In this section, we shall introduce a special $r$-cluster point, say the {\it twisted cluster point}, associated to the set of twisted Coxeter elements related to $\vee$
in \eqref{eq: B_n}.
As we have seen in Proposition \ref{prop: number tCox elts}, the number of all twisted Coxeter elements of $W_{2n+1}$ related to $\vee$ in \eqref{eq: B_n}
is the same as $4 \times 3^{n-1}$ while the number of all Coxeter elements and hence the number of distinct commutation classes in $\lf Q \rf$ is
the same as $2^{2n}$. In this section, we show that the number of distinct commutation classes in the twisted cluster point is also $2^{2n}$
even though the number of twisted Coxeter elements is $4 \times 3^{n-1}$.

\medskip

Let us consider the following word $\bi_0$ of $W$ of the finite type $A_{2n+1}$:
\[ \bi_0 = \prod_{k=0}^{2n} (1\ 2\ 3\cdots n+1)^{k\vee}.\]
Here
\begin{align} \label{eq: vee def}
& (j_1 \cdots j_n)^\vee \seteq j^\vee_1 \cdots j^\vee_n \text{ and }
(j_1 \cdots j_n)^{k \vee} \seteq  ( \cdots ((j_1 \cdots j_n \underbrace{ )^\vee )^\vee \cdots )^\vee}_{ \text{ $k$-times} }.
\end{align}

\begin{observation}
Let us denote by $\ut=s_1s_2\cdots s_{n}s_{n+1}s_{2n+1}s_{2n}\cdots s_{n+1}$ and
$\htau=s_1s_2\cdots s_{n}s_{n+1}$. Then one can easily check that
\begin{align*}
& \htau = \left( \begin{matrix} 1& 2 & \ldots & n+1 &n+2& n+3 & \ldots & 2n+1 & 2n+2 \\ 2 & 3 & \ldots & n+2 & 1 & n+3 & \ldots & 2n+1 & 2n+2  \end{matrix} \right), \\
& \ut = \left( \begin{matrix} 1& 2 & \ldots & n+1 & n+2 & n+3 & \ldots & 2n+1 & 2n+2 \\ 2 & 3 & \ldots & 2n+2 & n+2 & 1 & \ldots & 2n & 2n+1  \end{matrix} \right),
\end{align*}
with the two-line notation. Then one can easily check that
$$ \bi_0= \ut^n\htau = \left( \begin{matrix} 1& 2 & \ldots & 2n-1 & 2n \\ 2n & 2n-1 & \ldots & 2 & 1   \end{matrix} \right),$$
which is the same as $w_0$. Note that $\bi_0$ is reduced since the length of the word $\bi_0$
is the same as the number of positive roots. Thus, from now on, we denote by $\ii_0$ instead of $\bi_0$, i.e.,
\[ \ii_0 = \prod_{k=0}^{2n} (1\ 2\ 3\cdots n+1)^{k\vee} \]
Furthermore, for $\vee$ in \eqref{eq: B_n},
\begin{itemize}
\item the $r$-cluster $\lf \ii_0 \rf$ is $\vee$-foldable and any reduced word in $\lf \ii_0 \rf$
is not adapted to any Dynkin quiver since
$$\mathsf{C}^{\vee}_{\lf \ii_0 \rf}= ( \underbrace{2n+1, \ldots ,2n+1}_{ n+1\text{-times}} ) \text{ and }
\mathsf{C}^{\vee}_{\lf Q \rf}= (2n+2, \ldots, 2n+2, n+1), $$
\item $s_1s_2\cdots s_{n}s_{n+1} \vee$ is a twisted Coxeter element.
\end{itemize}
\end{observation}

In Example \ref{ex: Coxeter composition}, $\ii_0'$ is adapted to
$$  Q = \xymatrix@R=3ex{ *{ \bullet }<3pt> \ar@{->}[r]_<{1}  &*{\bullet}<3pt>
\ar@{->}[r]_<{2}  &*{\bullet}<3pt>
\ar@{->}[r]_<{3} &*{\bullet}<3pt>
\ar@{->}[r]_<{4}  & *{\bullet}<3pt> \ar@{-}[l]^<{\ \ 5}
} \quad \text{ and } \quad \mathsf{C}^{\vee}_{\lf\ii_0\rf} = (6,6,3)$$
while $\ii_0$ is not adapted to any $Q$,
$$   \ii_0 = \prod_{k=0}^{4} (1\ 2\ 3)^{k\vee} \quad \text{ and } \quad  \mathsf{C}^{\vee}_{\lf\ii_0\rf} = (5,5,5)$$ which implies that $\lf \ii_0' \rf$ is foldable.

\begin{definition} \label{def: twisted cluster} \hfill
\begin{enumerate}
\item The $r$-cluster point $\lf \Qd \rf \seteq \lf \ii_0 \rf$ is called the {\it twisted adapted cluster point} of type $A_{2n+1}.$
\item A class $[\ii_0']\in \lf \Qd \rf$ is called a {\it twisted adapted class} of type $A_{2n+1}.$
\end{enumerate}
\end{definition}

Consider the monoid homomorphism
\[ \PP: \langle  I_{2n+1} \rangle \to \langle I_{2n} \rangle, \]
such that
\[ \PP(i)=i, \ \PP(j)=j-1, \ \PP(n+1)={\rm id} \quad\text{ for } i=1, \cdots, n \text{ and } j=n+2, \cdots, 2n+1.\]

The following lemma is obvious and useful to prove propositions in this section.

\begin{lemma} \label{Lemma} Let $m$ be any non-negative integer.
\begin{enumerate}
\item[{\rm (1)}]
Let $\ii=i_1\  i_2\cdots i_l$ be a reduced word for $w \in W$ of type $A_m$ and let $1\leq k_1<k_2\leq l $.
If $\ii' = i_{\sigma(1)}i_{\sigma(2)}\cdots i_{\sigma(l)}\in [\ii]$ for $\sigma \in \mathfrak{S}_l$ satisfies
\[ \sigma^{-1}(k_1)>\sigma^{-1}(k_2)\]
then, for all $ k_1<k_3<k_2$, we have
\begin{equation*}
s_{ i_{k_1}}s_{i_{k_3}}=s_{i_{k_3}}s_{i_{k_1}}\quad \text{ or } \quad s_{i_{k_2}}s_{i_{k_3}}=s_{i_{k_3}}s_{i_{k_2}}
\end{equation*}
and
\begin{equation*}
s_{i_{k_1}}s_{i_{k_2}}=s_{i_{k_2}}s_{i_{k_1}}.
\end{equation*}
\item[{\rm (2)}]
Let $J=\{i,i+1,i+2\}\subset I_m$ of $A_m$ and $\ii$ be a reduced word for $w$. If there is an $i+1$ between any adjacent $i$ and $i+2$ in $\ii$ then
\[ \ii_{|J} = \ii'_{|J},\quad  \text{ for any } \ii'\in [\ii].\]
In other words, any $\ii'\in [\ii]$ can be obtained from $\ii$ by other commutation relations than $s_i s_{i+2} = s_{i+2} s_{i}.$
\end{enumerate}
\end{lemma}

\begin{proposition}\label{Prop:1_1102} \hfill
\begin{enumerate}
\item[{\rm (1)}] $\PP(\ii_0)= (1\ 2\ 3\cdots n\ 2n\ 2n-1\cdots n+1)^n(1\, 2\, 3\cdots n)$ is a reduced word of the longest element ${}_{2n}w_0$ of $A_{2n}.$
\item[{\rm (2)}]  For $\ii'_0 \in [\ii_0]$, we have $[\PP(\ii'_0)] = [\PP(\ii_0)].$
\end{enumerate}
\end{proposition}

\begin{proof}
(1) One can easily check that $\PP(\ii_0)$ is a reduced expression of ${}_{2n}w_0$ and is adapted to the Dynkin quiver
\[ Q = \scalebox{0.84}{\xymatrix@R=3ex{ \bullet
\ar@{<-}[r]_<{ \ 1} &  \bullet
\ar@{<-}[r]_<{ \ 2}  & \cdots &  \bullet
\ar@{<-}[r]_<{ \ n} &\bullet \ar@{->}[r]_<{ \ n+1}
&\bullet \ar@{->}[r]_<{ \ n+2} & \cdots
& \bullet \ar@{<-}[l]^<{\ \ \ \ \ \ 2n} }}. \]
(2) If  $i,j\in I_{2n+1}$ satisfies
\[ \, [\PP(ij)]\neq [\PP(ji)] \text{ and } i-j \geq 2\]
then $i=n+2$ and $j=n.$ Hence it is enough to check that
\begin{eqnarray} &&
  \parbox{80ex}{
any $\ii'_0\in [\ii_0]$ can be obtained by a sequence of commutation relations $ij=ji$ for $(i,j)\neq(n,n+2), (n+2, n)$.
}\label{Eqn:n,n+2}
\end{eqnarray}

To see this property, observe that the word (not reduced) $\ii_{0|J}$ for $J=\{n, n+1, n+2\}\subset I_{2n+1}$ is
\[\ii_{0|J}=(n\ n+1\ n+2\ n+1)^n(n\ n+1).\]
Since there is $n+1$ between every adjacent $n$ and $n+2$, (\ref{Eqn:n,n+2}) holds by Lemma \ref{Lemma}.
\end{proof}

\begin{proposition} \label{Prop:2_1102}
Let $[\ii'_0]$ be a class of reduced expressions in $\lf \Qd \rf$. Then we have the following properties.
\begin{enumerate}
\item[{\rm (1)}] There is $n+1$ between every adjacent $n$ and $n+2$ in $\ii'_0$.
\item[{\rm (2)}] Let $J=\{n,n+1, n+2\} \in I_{2n+1}.$ We have $\PP(\ii'_{0|J})= (n\ n+1)^n n \ \text{ or } \ (n+1 \ n)^n n+1.$
\end{enumerate}
\end{proposition}

\begin{proof}
We know the following facts:
\begin{enumerate}
\item[{\rm (a)}] Any reduced expression $\jj_0$ in $[\ii_0]$ satisfies $\jj_{0|J}=(n\ n+1\ n+2\ n+1)^n(n\ n+1)$,
\item[{\rm (b)}]  $n^\vee= n+2$, $n+1^\vee= n+1$ and $n+2^\vee=n$.
\end{enumerate}
Hence $\ii'_{0|J}$ is one of the followings:
\begin{eqnarray} &&
  \parbox{85ex}{
\begin{enumerate}
\item[{\rm (i)}] $(n\ n+1\ n+2\ n+1)^n(n\ n+1)$,
\item[{\rm (ii)}] $(n+1\ n+2\ n+1)(n\ n+1\ n+2\ n+1)^{n-1}(n\ n+1\ n^\vee)=(n+1\ n+2)(n+1\ n\ n+1\ n+2)^n$,
\item[{\rm (iii)}] $n+2(n+1\ n\ n+1\ n+2)^n n+1^\vee =(n+2 \ n+1\ n\ n+1)^n(n+2\  n+1),$
\item[{\rm (iv)}] $( n+1\ n\ n+1)(n+2 \ n+1\ n\ n+1)^{n-1}(n+2\  n+1\ n)=(n+1\ n) (n+1\ n+2\ n+1\ n)^n.$
\end{enumerate}
}\label{eq: 4cases}
\end{eqnarray}
We can check that (i), (ii), (iii), (iv) satisfy (1) and (2). Hence we proved the proposition.
\end{proof}

\begin{remark}
In \eqref{eq: 4cases}, one can observe that $(n+1)^\vee=n+1$ becomes a sink or a source (but not both) for any $[\ii'_0]\in \lf \Qd \rf$.
\end{remark}

\begin{proposition} \label{Prop:3_1102} \hfill
\begin{enumerate}
\item[{\rm (1)}] If $\ii'_0 ,\ii''_0  \in  [\ii_0] \in \lf \Qd \rf$ then $[\PP(\ii'_0)]=[\PP(\ii''_0)].$
Hence we can denote
\[ \PP([\ii'_0]):= [\PP(\ii'_0)].\]
\item[{\rm (2)}]  Let $[\ii'_0]$ be in the cluster point $\lf \Qd \rf$. Then $[\PP(\ii'_0)]$ is in the cluster point $\lf Q \rf$ for a Dynkin quiver $Q$ of $A_{2n}.$
\end{enumerate}
\end{proposition}

\begin{proof}
(1) Using Proposition \ref{Prop:2_1102} and Lemma \ref{Lemma}, the analogous argument to the proof of Proposition \ref{Prop:1_1102} (2) works.\\
(2) By (1), we can see that $i\in I_{2n+1}\backslash\{n+1\}$ is a sink (resp. source) in $\ii'_0$ if and only if $\PP(i)\in I_{2n}$ is a sink (resp. source) in $\PP(\ii'_0).$ Also, $\PP(i^\vee) = (\PP(i))^\vee$. Hence If we regard $r_{id}$ as the identity map then
\[ \PP([\ii'_0] \cdot r_i) = [\PP(\ii'_0)]\cdot r_{\PP(i)}.\]
In Proposition \ref{Prop:1_1102}, we showed $\PP([\ii_0])$ is adapted to a quiver $Q$ of $A_{2n}$. Since every adapted reduced expression consists of the unique cluster point $\lf Q \rf$, we proved (2).
\end{proof}

\begin{example} For $\ii_0$ in Example \ref{ex: Coxeter composition},
$$\PP([\ii_0]) =1 \ 2 \ 4 \ 3 \ 1 \ 2 \ 4 \ 3 \ 1 \ 2$$
which is a reduced expression of ${}_4w_0$ and adapted to
$$  Q = \xymatrix@R=3ex{ *{ \bullet }<3pt> \ar@{<-}[r]_<{1}  &*{\bullet}<3pt>
\ar@{<-}[r]_<{2}  &*{\bullet}<3pt>
\ar@{->}[r]_<{3} &*{\bullet}<3pt>
\ar@{-}[l]^<{ \ \ 4} }.$$
\end{example}

\begin{remark}
By Proposition \ref{Prop:3_1102}, if we restrict $\PP$ to reduced expressions in $\lf \Qd \rf$  then the map can be considered as a map between classes in $\lf \Qd \rf$ of $A_{2n+1}$ and $\lf Q \rf$ of $A_{2n}$. We denote
\[ \PPi\ : \ \lf \Qd \rf \to \lf Q \rf, \quad [\ii'_0]\mapsto [\PP(\ii'_0)]=: \PPi([\ii'_0]).  \]
\end{remark}

\begin{proposition} \label{Prop:5.7_1103}
For a given Dynkin quiver $Q'$ of $A_{2n}$, there are at least two distinct classes $[\ii'_0], [\ii''_0]\in \lf \Qd \rf$ such that $\PPi([\ii'_0])=\PPi([\ii''_0])=Q'.$
\end{proposition}

\begin{proof}
Consider the monoid homomorphism
\[ \RR:\langle I_{2n} \rangle \to \langle I_{2n+1} \rangle \]
such that
\[ i \mapsto \left\{ \begin{array}{ll} i & \text{ if }i=1, \cdots, n-1; \\ i+1 & \text{ if } i=n+2, \cdots, 2n; \\ n \ n+1 & \text{ if } i=n; \\ n+2\ n+1 & \text{ if } i=n+1. \end{array}\right.\]
Then
\begin{enumerate}[(i)]
\item  $\PP\circ \RR= {\rm id}.$
\item $\RR$ preserves commutation relations, that is $\RR(i)\RR(j)= \RR(j)\RR(i)$ if $|i-j|\geq 2.$
\item If $\RR(i) \ii$ is a reduced expression of $w_0$ then $[\RR(i) \ii]\cdot r_{\RR(i)}=[ \ii  \RR(i^\vee)]$.
\end{enumerate}

If we denote the reduced expression ${}_{2n}\ii_0=(1\ 2\ \cdots n \ 2n \  2n-1\ \cdots n+1)^n  (1\ 2\ \cdots n)$ of the longest element
${}_{2n}w_0$ of $A_{2n}$ then $\ii_0= \RR({}_{2n}\ii_0).$ Now we assume that
 \[\ [{}_{2n}\ii'_0]=[{}_{2n}\ii_0] \cdot r_{\bi}=[i_1i_2\cdots i_l] \text{ is a class of reduced expressions of } {}_{2n}w_0  \] and
 \[ \ [\ii'_0]= [\ii_0]\cdot r_{\RR(\bi)}=[\RR({}_{2n}\ii'_0)]= [\RR(i_1\cdots i_l)] \text{ is a class of  reduced expressions of }w_0 .\]
 Then, by (iii), $i\in I_{2n}$ is a source of $[{}_{2n}\ii'_0]$ if and only if there is a reduced expression in  $[\ii'_0]$ which starts with $\RR(i)$. Moreover, we have
 \[ \ [\ii'_0]\cdot r_{\RR(i_1)}=[\RR(i_2\cdots i_l) \ \RR(i^\vee_1) ]=[ \RR(i_2\cdots i_l \ i^\vee_1)]. \]
As a conclusion, for any word $\bi$ consisting of $\{1, \cdots, 2n\}$, the class of reduced expressions
\[\ [\ii_0]\cdot r_{\RR(\bi)}=[\RR(i_1\cdots i_l)] \text{ satisfies }\PPi([\ii_0]\cdot r_{\RR(\bi)})=[i_1, \cdots, i_l] =[{}_{2n}\ii_0] \cdot r_{\bi}.\]
Since every reduced expression adapted to a quiver consists of one cluster point, we proved that $\PPi$ is an onto map.

In addition, since  $[\ii_0]\cdot r_{\RR(\bi)}=[\RR(i_1\cdots i_l)]$ has $n+1$ as a source and  $r_{n+1}\cdot( [\ii_0]\cdot r_{\RR(\bi)})$ has $n+1$ as a sink,
these two are distinct classes. The following equation is easy to check:
\[ \PPi([\ii_0]\cdot r_{\RR(\bi)})= \PPi(r_{n+1}\cdot( [\ii_0]\cdot r_{\RR(\bi)})).\]
Hence we proved the proposition.
\end{proof}

\begin{proposition} \label{prop: 2to1}
The map $\PPi$ is a two-to-one and onto map. More precisely, for each Dynkin quiver $Q'$ of $A_{2n}$,
there is a unique class of reduced expressions in $\lf \Qd \rf$ which has $n+1$ as a source $($resp. sink$)$.
\end{proposition}

\begin{proof}
Suppose $\ii'_0$ and $\ii''_0$ are two distinct reduced expressions in $\lf \Qd \rf$ such that
\begin{enumerate}[(a)]
\item both $[\ii'_0]$ and $[\ii''_0]$ have $n+1$ as a source
\item $\PPi([\ii'_0])=\PPi([\ii''_0])=[Q']$ for a quiver $Q'$ of type $A_{2n}.$
\end{enumerate}

As we saw in Proposition \ref{Prop:2_1102}, we have
\[ \ii'_{0|J}=\ii''_{0|J}=(\RR(n)\RR(n+1))^n\RR(n) \text{ or } (\RR(n+1)\RR(n))^n\RR(n+1)\]
for $J=\{n,n+1, n+2\}$. Since $s_{i+1}s_{j}= s_{j}s_{i+1}$ for any $j\in I_{2n+1}\backslash J,$ using these commutation relations, we can find $\jj'_0\in [\ii'_0]$ and  $\jj''_0\in [\ii''_0]$ such that every $n+1$ exists right after $n$ or $n+2.$ Hence we have
\[ \text{ (i) both $[\jj'_0]$ and $[\jj''_0]$ have $n+1$ as a source, (ii) $\PPi([\jj'_0])=\PPi([\jj''_0])=Q'.$ }\]
Let us denote $\PP(\jj'_0)= {}_{2n}\ii'_0$ and  $\PP(\jj''_0)= {}_{2n}\ii''_0.$ Then
\begin{enumerate}[(i)]
\item both ${}_{2n}\ii'_0$ and ${}_{2n}\ii''_0$ are reduced expression of ${}_{2n}w_0$ adapted to $Q'$
\item $\RR({}_{2n}\ii'_0)=\jj'_0$ and $\RR({}_{2n}\ii''_0)=\jj''_0.$
\end{enumerate}
If ${}_{2n}\ii''_0$ can be obtained from ${}_{2n}\ii'_0$ by a sequence of commutation relations $s_{i_k}s_{j_k}=s_{j_k}s_{i_k}$ ($k=1, \cdots, t$) then
$\jj''_0$ can be obtained from $\jj'_0$ by a sequence of commutation relations $s_{\RR(i_k)}s_{\RR(j_k)}=s_{\RR(j_k)}s_{\RR(i_k)}$ for $k=1, \cdots, t,$
where $s_{n\ n+1}= s_n s_{n+1}.$ Hence $[\ii'_0]=[\jj'_0]=[\jj''_0]=[\ii''_0],$ i.e., there is a unique class $[\ii'_0]$ of reduced expressions of $w_0$ such that
\begin{enumerate}[(a)]
\item  $[\ii'_0]$ has $n+1$ as a source,
\item $\PPi([\ii'_0])=Q'$ for a quiver $Q'$ of type $A_{2n}.$
\end{enumerate}

Similarly, we can show that there is a unique class $[\ii'_0]$ of reduced expressions such that
\begin{enumerate}[(a)]
\item  $[\ii'_0]$ has $n+1$ as a sink,
\item $\PPi([\ii'_0])=Q'$ for a quiver $Q'$ of type $A_{2n}.$
\end{enumerate}

Recall that we showed $\PPi$ is an onto map in Proposition \ref{Prop:5.7_1103}. So we proved the proposition.
\end{proof}

\begin{theorem} \label{thm: 22n}
The number of classes in $\lf \Qd \rf$ is $2^{2n}.$
\end{theorem}

\begin{proof}
Since there are $2n-1$ arrows in a Dynkin quiver of $A_{2n}$ and each arrow has two possible directions, the number of  Dynkin quivers is $2^{2n-1}$. Since the map $\PPi$ is a two-to-one and onto map, the number of classes in $\lf \Qd \rf$ is $2 \times 2^{2n-1}=2^{2n}.$
\end{proof}

Now we focus on classes of reduced expressions related to twisted Coxeter elements.

\begin{proposition} \label{Prop:twisted 2n}
Let $\vee: i\mapsto n+1-i$ for $i\in I_{2n}$ and let $i_1 i_2 \cdots i_n \vee$ be a twisted Coxeter element of $W_{2n}$.
Then
 \begin{equation} \label{Eqn:3_1104}
\prod_{k=0}^{2n} (i_1\ i_2\ i_3\cdots i_n)^{k\vee}
 \end{equation}
is a reduced expression of ${}_{2n}w_0$ adapted to a Dynkin quiver $Q'$ of type $A_{2n}.$
\end{proposition}

\begin{proof}
It is enough to see that, for the reduced word $i_1  i_2  \cdots  i_n$ of the twisted Coxeter element, there is a Dynkin quiver $Q'$ of type $A_{2n}$ such that
\begin{enumerate}
\item[{\rm (1)}]  $i_1\  i_2 \cdots i_n\ i^\vee_1\ i^\vee_2\ \cdots i^\vee_n$ is adapted to $Q',$
\item[{\rm (2)}] $(i_1\  i_2 \cdots i_n\ i^\vee_1\ i^\vee_2\ \cdots i^\vee_n)^n (i_1\  i_2 \cdots i_n)$ is a reduced expression of ${}_{2n}w_0.$
\end{enumerate}
Since $i_1\  i_2 \cdots i_n\ i^\vee_1\ i^\vee_2\ \cdots i^\vee_n$ is a Coxeter element of $W_{2n}$,
it is well known that there exists a unique quiver $Q'$ determined by (1). Also, by Proposition \ref{Prop:AR}, an adapted reduced expression to $Q'$ has the property that
\[ \text{ $\#$ of $i$'s in the reduced expression} =\frac{ 2n+1 + a_i -b_i }{2} \]
where $a_i$ is the number of arrows toward $i$ in $Q'$ and $b_i$ is the number of arrows toward $i^\vee$ in $Q'$ between the vertices $i$
and $i^\vee.$ Hence instead of (2), we shall prove the quiver $Q'$ in (1) has the property that
\begin{equation} \label{Eqn:1_1104}
 \frac{ 2n+1 + a_{i_t} -b_{i_t} }{2}-\frac{ 2n+1 + a_{i^\vee_t} -b_{i^\vee_t}}{2}=1
 \end{equation}
for $t=1, \cdots, n.$ Since $a_{i^\vee_t}=b_{i_t}$ and $b_{i^\vee_t}=a_{i_t}$, (\ref{Eqn:1_1104}) is equivalent to
\begin{enumerate}
\item[{\rm (2$'$)}]  $a_{i_t}-b_{i_t}=1$ for all $1 \le t \le n$.
\end{enumerate}

We use an induction to prove this proposition. For $A_2$,  two reduced expressions $121$ and $212$ which are adapted to the Dynkin quivers
\[  \scalebox{0.84}{\xymatrix@R=3ex{ \bullet
\ar@{<-}[r]_<{1} &
 \bullet \ar@{-}[l]^< {\ \ \ \ 2} }}\qquad \text{ and } \qquad \scalebox{0.84}{\xymatrix@R=3ex{ \bullet
\ar@{->}[r]_<{1} &
 \bullet \ar@{-}[l]^< {\ \ \ \ 2} }},  \]
 respectively.

 Suppose we proved for $A_{2n-2}$ and take a reduced word $i_1 i_2 \cdots i_n$ of some twisted Coxeter element of type  $A_{2n}$: \\

\noindent
(Case 1)  Suppose $i_\alpha=1$ and $i_\beta=2$ for  $1\leq \alpha < \beta\leq n$. Then
\[(i_1\ i_2\ \cdots\ i_{\alpha-1}\ i_{\alpha+1}\ \cdots \ i_n) (i_1\ i_2\ \cdots\ i_{\alpha-1}\ i_{\alpha+1}\ \cdots \ i_n)^\vee\] is adapted to a
quiver $Q^-_{+1}$ which is isomorphic to a quiver $Q^-$ of type $A_{2n-2},$ via $i+1\mapsto i$ for $i\in I_{2n-2}.$
Note that by the induction hypothesis, the quiver $Q^-$ satisfies (1) and (2$'$). Now consider the quiver
 \[ Q'= \scalebox{0.84}{\xymatrix@R=3ex{ \bullet
\ar@{<-}[r]_<{1} &
 [\bullet \ar@{-}[l]^< {\ \ \ \ 2} & &Q^-_{+1}  &  &  \bullet]
\ar@{->}[r]_<{2n-1} &
 \bullet \ar@{-}[l]^< {\ \ \ \ 2n } } } \]
 of $A_{2n}.$ We can see that $Q'$ is the Dynkin quiver satisfying (1) and (2$'$).

\noindent
 (Case 2)
  Suppose $i_\alpha=1$ and $i_\beta=2$ for  $1\leq \beta < \alpha \leq n$.  Let $Q^-_{+1}$ and $Q^-$ be quivers defined in similar ways with those of (Case 1).   Then
  \[ Q'= \scalebox{0.84}{\xymatrix@R=3ex{ \bullet
\ar@{->}[r]_<{1} &
 [\bullet \ar@{-}[l]^< {\ \ \ \ 2} & &Q^-_{+1}  &  &  \bullet]
\ar@{<-}[r]_<{2n-1} &
 \bullet \ar@{-}[l]^< {\ \ \ \ 2n } } } \]
 is the Dynkin quiver satisfying (1) and (2$'$).

\noindent
 (Case 3)
  Suppose $i_\alpha=2n$ and $i_\beta=2$ for  $1\leq \alpha, \beta \leq n$.  Let $Q^-_{+1}$ and $Q^-$ be quivers defined in similar ways with those of (Case 1).   Then
  \[ Q'= \scalebox{0.84}{\xymatrix@R=3ex{ \bullet
\ar@{->}[r]_<{1} &
 [\bullet \ar@{-}[l]^< {\ \ \ \ 2} & &Q^-_{+1}  &  &  \bullet]
\ar@{->}[r]_<{2n-1} &
 \bullet \ar@{-}[l]^< {\ \ \ \ 2n } } } \]
 is the Dynkin quiver satisfying (1) and (2$'$).

We omit the cases when  there is $1\leq \beta \leq n$ such that $i_\beta=2n-1$. The same argument works for these cases. Hence we proved the proposition.
\end{proof}

\begin{theorem} \label{thm:twisted longest}
Let $[i_1\ i_2\ \cdots i_{n+1}]\vee$ be a twisted Coxeter element of $A_{2n+1}.$ Then
\begin{equation}\label{Eqn: 2n+1 twsited}
 \ii'_0 = \prod_{k=0}^{2n} (i_1\ i_2\ \cdots i_{n+1})^{k\vee}
 \end{equation}
is a reduced expression of $w_0$ in $\lf \Qd \rf.$
\end{theorem}

\begin{proof}
We recall that the following fact: In the proof of Proposition \ref{Prop:5.7_1103}, we showed that if ${}_{2n}\ii''_0$ is a reduced expression of ${}_{2n}w_0$
adapted to $Q'$ of $A_{2n}$ then $\RR({}_{2n}\ii''_0)$ is a reduced expression of $w_0$ in $\lf \Qd \rf.$ There are two types of twisted Coxeter elements
of type $A_{2n+1}$:
\begin{enumerate}
\item[{\rm (1)}] The twisted Coxeter element $[i_1\ i_2\ \cdots i_{n+1}]\vee$ is said to be type $1$ if there is $\alpha$ and $\beta$ such that
$1\leq \alpha<\beta\leq n+1$ and $i_{\alpha}=n$ or $n+2$ and $i_\beta=n+1$.
\item[{\rm (2)}] The twisted Coxeter element $[i_1\ i_2\ \cdots i_{n+1}]\vee$ is said to be type $2$ if there is $\alpha$ and $\beta$ such that
$1\leq \alpha<\beta\leq n+1$ and $i_{\alpha}=n+1$ and $i_\beta=n$ or $n+2$.
\end{enumerate}

If $[i_1\ i_2\ \cdots i_{n+1}]\vee$ is type $1$, by using the commutation relations $s_{n+1}s_j=s_j s_{n+1}$ for $j\in I_{2n+1}\backslash\{n,n+1, n+2\}$, we can find $i'_1, i'_2, \cdots, i'_{n+1}$ such that
\begin{enumerate}[(i)]
\item $[i'_1\ i'_2\ \cdots i'_{n+1}]=[i_1\ i_2\ \cdots i_{n+1}],$
\item $i'_{\alpha}=n$ or $n+2$ and $i'_{\alpha+1}=n+1$ for some $\alpha \in \{1, \cdots, n\}.$
\end{enumerate}
Then
\begin{enumerate}[(i)]
\item $\left[\ \prod_{k=0}^{2n} (i'_1\ i'_2\ \cdots i'_{n+1})^{k\vee}\ \right]=\left[\ \prod_{k=0}^{2n} (i_1\ i_2\ \cdots i_{n+1})^{k\vee}\ \right]$,
\item $\prod_{k=0}^{2n} (i'_1\ i'_2\ \cdots i'_{n+1})^{k\vee}=\RR \left( \prod(i'_1\ i'_2\ \cdots i'_{n})^{k\vee} \right)$.
\end{enumerate}
By Proposition \ref{Prop:twisted 2n},  $\prod(i'_1\ i'_2\ \cdots i'_{n})^{k\vee}$ is a reduced expression of ${}_{2n}w_0$ adapted to a Dynkin quiver of type $A_{2n}$ and by Proposition \ref{Prop:5.7_1103},  $\RR( \prod(i'_1\ i'_2\ \cdots i'_{n})^{k\vee})$ is a reduced expression of $w_0$ in $\lf \Qd \rf.$

By (ii), we know that  $\prod_{k=0}^{2n} (i'_1\ i'_2\ \cdots i'_{n+1})^{k\vee}$ is  a reduced expression of $w_0$ in $\lf \Qd \rf$ and hence, by (i), $\prod_{k=0}^{2n} (i_1\ i_2\ \cdots i_{n+1})^{k\vee}$ is  a reduced expression of $w_0$ in $\lf \Qd \rf.$

\vskip 2mm

If $[i_1\ i_2\ \cdots i_{n+1}]\vee$ is type $2$, by using the commutation relations $s_{n+1}s_j=s_j s_{n+1}$ for $j\in I_{2n+1}\backslash\{n,n+1, n+2\}$, we can find $i'_1, i'_2, \cdots, i'_{n+1}$ such that
\begin{enumerate}[(i)]
\item $[i'_1\ i'_2\ \cdots i'_{n+1}]=[i_1\ i_2\ \cdots i_{n+1}],$
\item $i'_1=n+1$.
\end{enumerate}
Then
\[ \ \left[\ \prod_{k=0}^{2n} (i_1\ i_2\ \cdots i_{n+1})^{k\vee}\ \right] =\left[\ \prod_{k=0}^{2n} (i'_1\ i'_2\ \cdots i'_{n+1})^{k\vee}\ \right]\] and
\[ \ \left[\ \prod_{k=0}^{2n} (i'_1\ i'_2\ \cdots i'_{n+1})^{k\vee}\ \right] \cdot r_{n+1}= \left[\ \prod_{k=0}^{2n} (i'_2\ i'_3\ \cdots i'_{n+1}\ i'_1)^{k\vee}\ \right]. \]
Since $[ i'_2\ i'_3\ \cdots i'_{n+1}\ i'_1]\vee$ is a type 1 twisted Coxeter element, we conclude  $[\ \prod_{k=0}^{2n} (i_1\ i_2\ \cdots i_{n+1})^{k\vee}\ ] $ is a reduced expression of $w_0$ in $\lf \Qd \rf.$
\end{proof}

\begin{proposition} \label{Prop:Dynkin quiver 2n twisted}
The number of Dynkin quiver $Q'$of type $A_{2n}$ such that
\begin{equation} \label{Eqn:char}
|a_{i}-b_{i}|=1, \quad i\in I_{2n},
\end{equation}
where $a_i$ is the number of arrows toward $i$ and $b_i$ is the number of arrows toward $i^\vee$ between the vertices $i$ and $i^\vee$ is $2 \times 3^{n-1}.$
\end{proposition}

\begin{proof}
There are two Dynkin quivers of type $A_2$ and both satisfies the condition (\ref{Eqn:char}). Suppose we proved for $A_{2n-2}.$

If a Dynkin quiver $Q'$ of type $A_{2n}$ satisfies the condition (\ref{Eqn:char}), the subquiver $Q'|_{J}$ for $J=\{2, \cdots, 2n-1\}$ of $Q$ consists of vertices $2, \cdots, 2n-1$ is isomorphic to a quiver of type $A_{2n-2}$ satisfying (\ref{Eqn:char}). There are two types of Dynkin quivers satisfying (\ref{Eqn:char}). We call the quiver $Q'$ by type 1 if $a_2-b_2=1$ and by type 2 if  $a_2-b_2=-1.$

A type 1 quiver $Q'$ has  one of the following forms:
 \[ Q'= \scalebox{0.84}{\xymatrix@R=3ex{ \bullet
\ar@{<-}[r]_<{1} &
 [\bullet \ar@{-}[l]^< {\ \ \ \ 2} & &Q'|_{J}  &  &  \bullet]
\ar@{->}[r]_<{2n-1} &
 \bullet \ar@{-}[l]^< {\ \ \ \ 2n } } } \]
  \[ Q'= \scalebox{0.84}{\xymatrix@R=3ex{ \bullet
\ar@{->}[r]_<{1} &
 [\bullet \ar@{-}[l]^< {\ \ \ \ 2} & &Q'|_{J}  &  &  \bullet]
\ar@{<-}[r]_<{2n-1} &
 \bullet \ar@{-}[l]^< {\ \ \ \ 2n } } } \]
  \[ Q'= \scalebox{0.84}{\xymatrix@R=3ex{ \bullet
\ar@{->}[r]_<{1} &
 [\bullet \ar@{-}[l]^< {\ \ \ \ 2} & &Q'|_{J} &  &  \bullet]
\ar@{->}[r]_<{2n-1} &
 \bullet \ar@{-}[l]^< {\ \ \ \ 2n } } } \]

 A type 2 quiver $Q'$ has  one of the following forms:
 \[ Q'= \scalebox{0.84}{\xymatrix@R=3ex{ \bullet
\ar@{<-}[r]_<{1} &
 [\bullet \ar@{-}[l]^< {\ \ \ \ 2} & &Q'|_{J}  &  &  \bullet]
\ar@{->}[r]_<{2n-1} &
 \bullet \ar@{-}[l]^< {\ \ \ \ 2n } } } \]
  \[ Q'= \scalebox{0.84}{\xymatrix@R=3ex{ \bullet
\ar@{->}[r]_<{1} &
 [\bullet \ar@{-}[l]^< {\ \ \ \ 2} & &Q'|_{J}  &  &  \bullet]
\ar@{<-}[r]_<{2n-1} &
 \bullet \ar@{-}[l]^< {\ \ \ \ 2n } } } \]
  \[ Q'= \scalebox{0.84}{\xymatrix@R=3ex{ \bullet
\ar@{<-}[r]_<{1} &
 [\bullet \ar@{-}[l]^< {\ \ \ \ 2} & &Q'|_{J} &  &  \bullet]
\ar@{<-}[r]_<{2n-1} &
 \bullet \ar@{-}[l]^< {\ \ \ \ 2n } } } \]

In other words, for each Dynkin quiver of type $A_{2n-2}$ satisfying (\ref{Eqn:char}), we get three different Dynkin quivers of type $A_{2n}$ satisfying (\ref{Eqn:char}). Hence there are $2 \times 3^{n-1}$ Dynkin quivers of type $A_{2n}$ satisfying (\ref{Eqn:char}).
\end{proof}

\begin{proposition}\label{Prop:6.14_1104}
Let
\[\Omega=\left\{ \left[  \left.\ \prod_{k=0}^{2n} \ii^{k\vee} \right] \ \right| \ [\ii]\vee \text{ is a twisted Coxeter element of type $A_{2n+1}$}\subset \lf \Qd \rf \ \right\}\]
and
\[\QQ=\{ \text{ Dynkin quiver $Q'$ of type $A_{2n}$ satisfying the condition \eqref{Eqn:char} }\}\subset \lf Q\rf \]
The map
\[ \left.\PPi\right|_{\Omega}: \Omega \to \QQ \]
is a two-to-one and onto map.
\end{proposition}

\begin{proof}
By Theorem \ref{thm:twisted longest}, we know that $\PPi(\Omega)\subset\lf Q\rf. $ By counting the indices in $[\ \prod_{k=0}^{2n} \ii^{k\vee}\ ]\in \Omega $, we can easily see that  $\PPi(\Omega)\subset\QQ.$ Since the map $\PPi$ is a two-to-one map,
$\left.\PPi\right|_{\Omega}$ is also a two-to-one map. Moreover, since by Proposition \ref{prop: number tCox elts} and  Proposition \ref{Prop:Dynkin quiver 2n twisted}, we know that
\[ |\Omega|= 4 \times 3^{n-1} \quad \text{ and } \quad |\QQ|=2 \times 3^{n-1}.\]
Hence $ \left.\PPi\right|_{\Omega}$ is onto.
\end{proof}

Now we summarize results in this section, which can be understood as a {\it surgery} for
obtaining $[\Qd] \in \lf \Qd \rf$ of type $A_{2n+1}$ from $[Q] \in \lf Q \rf$ of type $A_{2n}$:

\begin{theorem} \label{thm: surgery}
Let $\lf \Qd \rf$ be the twisted adapted cluster point of type $A_{2n+1}$ and $\lf Q\rf$ be the adapted cluster point of type $A_{2n}$. \hfill
\begin{enumerate}
\item[{\rm (1)}]   The map $\PPi:\lf \Qd \rf \to \lf Q\rf$ is a two-to-one and onto map.
\item[{\rm (2)}] Let $\ii'_0$ be a reduced expression such that  $[\ii'_0]\in \lf\ii_0\rf$. Then $[\ii'_0]$ has $n+1$ as a source or a sink, but not both.
\item[{\rm (3)}] For each Dynkin quiver $Q'$ of type $A_{2n}$, we get $\PPi^{-1}(Q')$ by the following procedure.
\begin{enumerate}[(i)]
\item[{\rm (i)}] Let ${}_{2n}\ii'_0$ be a reduced expression of ${}_{2n}w_0$ adapted to $Q'$.
\item[{\rm (ii)}] Substitute $i\in \{n+1, \cdots, 2n\}$ in ${}_{2n}\ii'_0$ by $i^+=i+1$.
\item[{\rm (iii)}] Between each adjacent $n$ and $n+2$, insert $n+1$.
\item[{\rm (iv)}] Insert another $n+1$ at the beginning or at the end $($not both$)$ of the sequence obtained in {\rm (iii)}.
\end{enumerate}
Note that, in {\rm (iii)}, the positions of $n+1$'s are not unique. However, by commutation relations, we get a unique class of reduced expressions.
In {\rm (iv)}, we get two distinct classes of reduced expressions.
\item[{\rm (4)}] The number of classes in $\lf\ii_0\rf$ is $2^{2n}$.
\item[{\rm (5)}]  The number of twisted Coxeter elements is $4 \times 3^{n-1}$
\item[{\rm (6)}] We can associate a  twisted Coxeter element to a reduced expression of $w_0$ in $\lf \Qd \rf$, in the sense of {\rm Theorem \ref{thm:twisted longest}}.
\item[{\rm (7)}] A class of reduced expressions associated to a twisted Coxeter element is associated to a Dynkin quiver
satisfying \eqref{Eqn:char}, via the map  $\PPi:\lf \Qd \rf \to \lf Q\rf$.
\end{enumerate}
\end{theorem}

The two distinct classes in {\rm (iii)} obtained from $Q$ of type $A_{2n}$ are denoted by
\begin{align} \label{eq: two classes}
\begin{cases}
[Q^<] & \text{ if $n+1$ is a source of the class}, \\
[Q^>] & \text{ if $n+1$ is a sink of the class}.
\end{cases}
\end{align}

\begin{example} \label{Ex: twisted Coxeter}
Consider a Dynkin quiver $Q$ of $A_6$ :
$$ Q = {\xymatrix@C=4ex{ \bullet
\ar@{->}[r]_<{ \ 1} &  \bullet
\ar@{<-}[r]_<{ \ 2}  &  \bullet
\ar@{->}[r]_<{ \ 3} &\bullet \ar@{->}[r]_<{ \ 4}
&\bullet \ar@{<-}[r]_<{ \ 5}
& \bullet \ar@{-}[l]^<{\ \ \ \ \ \ 6} }}.$$
The quiver $Q$ satisfies (\ref{Eqn:char}) and $s_5 s_4 s_6 s_2 s_3 s_1$ is the corresponding Coxeter element. The commutation class adapted to $Q$ is
\[ [Q]= [(5\ 4\ 6\ 2\ 3\ 1)^3\ 5\ 4\ 6].\]
Then
\[\PPi^{-1}( [Q] )=\{\ [Q^<] \seteq [\ ( 6\ 4\ 5\ 7\ 2\ 4\  3\ 1)^3\ \ 6\ 4\ 5\ 7\ ],  \ [Q^>] \seteq [\ ( 6\ 5\ 4\ 7\ 2\  3\ 4\ 1)^3\ \ 6\  5\ 4 \ 7\ ] \ \} ,  \]
whose twisted Coxeter elements are $6\ 4\ 5\ 7 \vee$ and $6\ 5\ 4\ 7 \vee.$
The AR-quiver $\Gamma_Q$ has the shape of
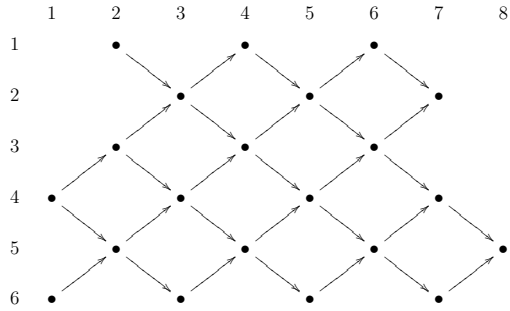
\[ \ \ \ \scalebox{0.6}{\xymatrix@C=2ex@R=1ex{
& 1  && 2 && 3 && 4 && 5 && 6 && 7 && 8  \\
1 & &&  \bullet\ar@{->}[ddrr] && &&  \bullet\ar@{->}[ddrr]  && && \bullet\ar@{->}[ddrr]  && && \\
& &&&&& &&&& &&&& &&\\
2 & && &&\bullet\ar@{->}[ddrr] \ar@{->}[uurr]  && && \bullet \ar@{->}[ddrr]\ar@{->}[uurr] && && \bullet &&  \\
&&&&&& &&&&&&&&\\
3 & && \bullet \ar@{->}[ddrr] \ar@{->}[uurr] && && \bullet\ar@{->}[ddrr] \ar@{->}[uurr]  && && \bullet\ar@{->}[ddrr] \ar@{->}[uurr]  && && \\
&&&&&& &&&& &&&& &&\\
4 & \bullet\ar@{->}[ddrr] \ar@{->}[uurr]  && && \bullet\ar@{->}[ddrr] \ar@{->}[uurr]  && && \bullet \ar@{->}[ddrr] \ar@{->}[uurr] && && \bullet\ar@{->}[ddrr]   &&\\
&&&&&&&&&&&&&&\\
5 & && \bullet\ar@{->}[ddrr] \ar@{->}[uurr]  && && \bullet\ar@{->}[ddrr] \ar@{->}[uurr]  && && \bullet\ar@{->}[ddrr] \ar@{->}[uurr]  && && \bullet  \\
 &&&&&& &&&&&&&&&&\\
6 & \bullet \ar@{->}[uurr]  && && \bullet \ar@{->}[uurr]  && && \bullet  \ar@{->}[uurr] && && \bullet  \ar@{->}[uurr] &&
}} \]
and a combinatorial AR-quiver related to $\PPi^{-1}( [Q] )$ has the shape of one of the followings:
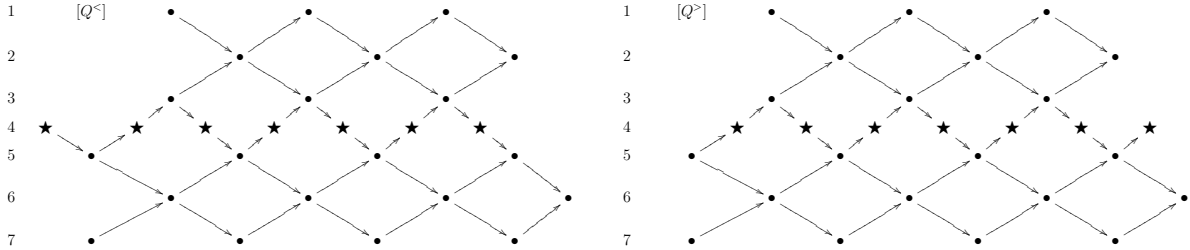
\[ \ \ \ \scalebox{0.5}{\xymatrix@C=2ex@R=1ex{
1 && [Q^<] &&  \bullet\ar@{->}[ddrr] && &&  \bullet\ar@{->}[ddrr]  && && \bullet\ar@{->}[ddrr]  && && \\
& &&&&& &&&& &&&& &&\\
2 && && &&\bullet\ar@{->}[ddrr] \ar@{->}[uurr]  && && \bullet \ar@{->}[ddrr]\ar@{->}[uurr] && && \bullet &&  \\
&&&&&& &&&&&&&&\\
3 & & && \bullet \ar@{->}[dr] \ar@{->}[uurr] && && \bullet\ar@{->}[dr] \ar@{->}[uurr]  && && \bullet\ar@{->}[dr] \ar@{->}[uurr]  && && \\
4& \bigstar\ar@{->}[dr]&&\bigstar\ar@{->}[ur]&&\bigstar\ar@{->}[dr]& &\bigstar\ar@{->}[ur]&&\bigstar\ar@{->}[dr]& &\bigstar\ar@{->}[ur]&&\bigstar\ar@{->}[dr]& &&\\
5 && \bullet\ar@{->}[ddrr] \ar@{->}[ur]  && && \bullet\ar@{->}[ddrr] \ar@{->}[ur]  && && \bullet \ar@{->}[ddrr] \ar@{->}[ur] && && \bullet\ar@{->}[ddrr]   &&\\
&&&&&&&&&&&&&&\\
6 && && \bullet\ar@{->}[ddrr] \ar@{->}[uurr]  && && \bullet\ar@{->}[ddrr] \ar@{->}[uurr]  && && \bullet\ar@{->}[ddrr] \ar@{->}[uurr]  && && \bullet  \\
 &&&&&& &&&&&&&&&&\\
7 && \bullet \ar@{->}[uurr]  && && \bullet \ar@{->}[uurr]  && && \bullet  \ar@{->}[uurr] && && \bullet  \ar@{->}[uurr] &&
}} \ \ \ \
 \scalebox{0.5}{\xymatrix@C=2ex@R=1ex{
1 && [Q^>] &&  \bullet\ar@{->}[ddrr] && &&  \bullet\ar@{->}[ddrr]  && && \bullet\ar@{->}[ddrr]  && && \\
& &&&&& &&&& &&&& &&\\
2 && && &&\bullet\ar@{->}[ddrr] \ar@{->}[uurr]  && && \bullet \ar@{->}[ddrr]\ar@{->}[uurr] && && \bullet &&  \\
&&&&&& &&&&&&&&\\
3 & & && \bullet \ar@{->}[dr] \ar@{->}[uurr] && && \bullet\ar@{->}[dr] \ar@{->}[uurr]  && && \bullet\ar@{->}[dr] \ar@{->}[uurr]  && && \\
4& &&\bigstar\ar@{->}[ur]&&\bigstar\ar@{->}[dr]& &\bigstar\ar@{->}[ur]&&\bigstar\ar@{->}[dr]& &\bigstar\ar@{->}[ur]&&\bigstar\ar@{->}[dr]& &  \bigstar&\\
5 && \bullet\ar@{->}[ddrr] \ar@{->}[ur]  && && \bullet\ar@{->}[ddrr] \ar@{->}[ur]  && && \bullet \ar@{->}[ddrr] \ar@{->}[ur] && && \bullet\ar@{->}[ddrr] \ar@{->}[ur]  &&\\
&&&&&&&&&&&&&&\\
6 && && \bullet\ar@{->}[ddrr] \ar@{->}[uurr]  && && \bullet\ar@{->}[ddrr] \ar@{->}[uurr]  && && \bullet\ar@{->}[ddrr] \ar@{->}[uurr]  && && \bullet  \\
 &&&&&& &&&&&&&&&&\\
7 && \bullet \ar@{->}[uurr]  && && \bullet \ar@{->}[uurr]  && && \bullet  \ar@{->}[uurr] && && \bullet  \ar@{->}[uurr] &&
}}
 \]
\end{example}

\begin{example} \label{Ex: non twisted Coxeter}
Consider a Dynkin quiver $Q$ of $A_6$
\[ Q = {\xymatrix@R=3ex{ \bullet
\ar@{->}[r]_<{ \ 1} &  \bullet
\ar@{->}[r]_<{ \ 2}  &  \bullet
\ar@{->}[r]_<{ \ 3} &\bullet \ar@{->}[r]_<{ \ 4}
&\bullet \ar@{<-}[r]_<{ \ 5}
& \bullet \ar@{-}[l]^<{\ \ \ \ \ \ 6} }}. \]
The quiver $Q$ does not satisfy (\ref{Eqn:char}). A reduced expression adapted to $Q$ is
\[ [Q]= [5\ 6\ 4\ 3\ 2\ 1\ 5\ 6\ 4\ 3\ 2\ 1\ 5\ 6\ 4\ 3\ 5\ 6\ 4\ 5\ 6] .\]
Then $\PPi^{-1}( [Q] )$ is
\begin{equation*}
\begin{aligned}
& \left\{ \ [Q^{<}] \seteq [ \ 6\ 7 \ 4\ 5\ 4 \ 3\ 2\ 1\ 6\ 7\ 4\ 5\ 4\ 3\ 2\ 1\ 6\ 7\ 4\ 5\ 4\ 3\ 6\ 7\ 4\ 5\ 6\ 7  \ ] \right., \\
& \left. \ \ \ \ \ \   [Q^{>}] \seteq [ \ 6\ 7 \ 5\ 4\ 3 \ 4\ 2\ 1\ 6\ 7\  5\ 4\ 3\ 4\ 2\ 1\ 6\ 7\ 5\ 4\ 3\ 4\ 6\ 7\ 5\ 4 \ 6\ 7  \ ]  \right\}.
\end{aligned}
\end{equation*}
These two classes of reduced expressions do {\it not} have twisted Coxeter elements.
The AR-quiver $\Gamma_Q$ has the shape of
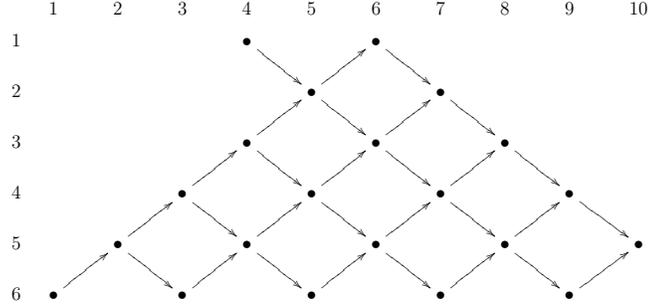
\[ \ \ \ \scalebox{0.6}{\xymatrix@C=2ex@R=1ex{
& 1 && 2 && 3 && 4 && 5  && 6 && 7 && 8 && 9 && 10 \\
1 && & &&  &&\bullet\ar@{->}[ddrr]  && && \bullet\ar@{->}[ddrr]  && &&  && && \\
& &&&&& &&&& &&&& &&\\
2  && & && && &&\bullet\ar@{->}[ddrr] \ar@{->}[uurr]  && && \bullet \ar@{->}[ddrr] && && &&  \\
&&&&&& &&&&&&&&\\
3 && & && && \bullet \ar@{->}[ddrr] \ar@{->}[uurr] && && \bullet\ar@{->}[ddrr] \ar@{->}[uurr]  && && \bullet\ar@{->}[ddrr]  && && \\
&&&&&& &&&& &&&& &&\\
4 && & && \bullet\ar@{->}[ddrr] \ar@{->}[uurr]  && && \bullet\ar@{->}[ddrr] \ar@{->}[uurr]  && && \bullet \ar@{->}[ddrr] \ar@{->}[uurr] && && \bullet\ar@{->}[ddrr]   &&\\
&&&&&&&&&&&&&&\\
5 && &  \bullet\ar@{->}[ddrr] \ar@{->}[uurr]  && && \bullet\ar@{->}[ddrr] \ar@{->}[uurr]  && && \bullet\ar@{->}[ddrr] \ar@{->}[uurr]  && && \bullet\ar@{->}[ddrr] \ar@{->}[uurr]  && && \bullet  \\
 &&&&&& &&&&&&&&&&\\
6  &  \bullet \ar@{->}[uurr]  & && & \bullet \ar@{->}[uurr]  && && \bullet \ar@{->}[uurr]  && && \bullet  \ar@{->}[uurr] && && \bullet  \ar@{->}[uurr] &&
}} \]
and the combinatorial AR-quiver related to $\PPi^{-1}( [Q] )$ has the shape of one of the followings:
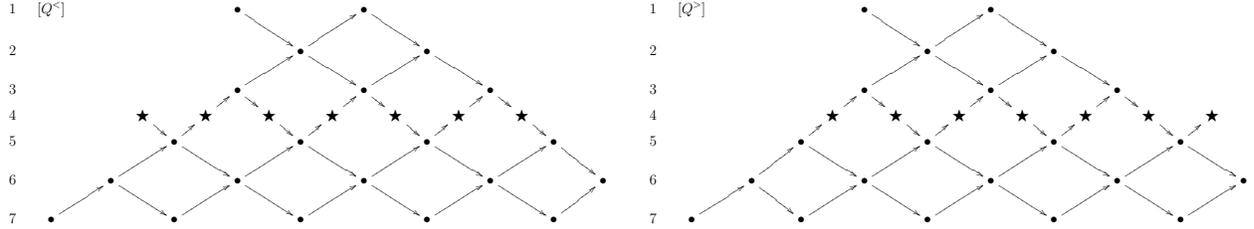
\[ \ \ \ \scalebox{0.46}{\xymatrix@C=2ex@R=1ex{
1 &[Q^<] &  & &&  &&\bullet\ar@{->}[ddrr]  && && \bullet\ar@{->}[ddrr]  && &&  && && \\
& &&&&& &&&& &&&& &&\\
2  && & && && &&\bullet\ar@{->}[ddrr] \ar@{->}[uurr]  && && \bullet \ar@{->}[ddrr] && && &&  \\
&&&&&& &&&&&&&&\\
3 && & && && \bullet \ar@{->}[dr] \ar@{->}[uurr] && && \bullet\ar@{->}[dr] \ar@{->}[uurr]  && && \bullet\ar@{->}[dr]  && && \\
4 &&&&\bigstar\ar@{->}[dr]&& \bigstar \ar@{->}[ur]  &&\bigstar\ar@{->}[dr]&& \bigstar \ar@{->}[ur] &&\bigstar\ar@{->}[dr]&& \bigstar \ar@{->}[ur] &&\bigstar\ar@{->}[dr]\\
5 && & && \bullet\ar@{->}[ddrr] \ar@{->}[ur]  && && \bullet\ar@{->}[ddrr] \ar@{->}[ur]  && && \bullet \ar@{->}[ddrr] \ar@{->}[ur] && && \bullet\ar@{->}[ddrr]   &&\\
&&&&&&&&&&&&&&\\
6 && &  \bullet\ar@{->}[ddrr] \ar@{->}[uurr]  && && \bullet\ar@{->}[ddrr] \ar@{->}[uurr]  && && \bullet\ar@{->}[ddrr] \ar@{->}[uurr]  && && \bullet\ar@{->}[ddrr] \ar@{->}[uurr]  && && \bullet  \\
 &&&&&& &&&&&&&&&&\\
7  &  \bullet \ar@{->}[uurr]  & && & \bullet \ar@{->}[uurr]  && && \bullet \ar@{->}[uurr]  && && \bullet  \ar@{->}[uurr] && && \bullet  \ar@{->}[uurr] &&
}}
\ \ \ \ \  \scalebox{0.46}{\xymatrix@C=2ex@R=1ex{
1 &[Q^>]&  & &&  &&\bullet\ar@{->}[ddrr]  && && \bullet\ar@{->}[ddrr]  && &&  && && \\
& &&&&& &&&& &&&& &&\\
2  && & && && &&\bullet\ar@{->}[ddrr] \ar@{->}[uurr]  && && \bullet \ar@{->}[ddrr] && && &&  \\
&&&&&& &&&&&&&&\\
3 && & && && \bullet \ar@{->}[dr] \ar@{->}[uurr] && && \bullet\ar@{->}[dr] \ar@{->}[uurr]  && && \bullet\ar@{->}[dr]  && && \\
4 &&&&&& \bigstar \ar@{->}[ur]  &&\bigstar\ar@{->}[dr]&& \bigstar \ar@{->}[ur] &&\bigstar\ar@{->}[dr]&& \bigstar \ar@{->}[ur] &&\bigstar\ar@{->}[dr]&& \bigstar\\
5 && & && \bullet\ar@{->}[ddrr] \ar@{->}[ur]  && && \bullet\ar@{->}[ddrr] \ar@{->}[ur]  && && \bullet \ar@{->}[ddrr] \ar@{->}[ur] && && \bullet\ar@{->}[ddrr]\ar@{->}[ur]   &&\\
&&&&&&&&&&&&&&\\
6 && &  \bullet\ar@{->}[ddrr] \ar@{->}[uurr]  && && \bullet\ar@{->}[ddrr] \ar@{->}[uurr]  && && \bullet\ar@{->}[ddrr] \ar@{->}[uurr]  && && \bullet\ar@{->}[ddrr] \ar@{->}[uurr]  && && \bullet  \\
 &&&&&& &&&&&&&&&&\\
7  &  \bullet \ar@{->}[uurr]  & && & \bullet \ar@{->}[uurr]  && && \bullet \ar@{->}[uurr]  && && \bullet  \ar@{->}[uurr] && && \bullet  \ar@{->}[uurr] &&
}} \]
\end{example}

\begin{remark} \label{Rem:surgery}
As we can see in Example \ref{Ex: twisted Coxeter} and Example \ref{Ex: non twisted Coxeter}, the combinatorial AR-quiver $\Upsilon_{[\ii_0]}$ such that  $[\ii_0]\in \lf \Qd \rf$ of type $A_{2n+1}$ is closely related to the AR-quiver $\Gamma_Q$ for $Q \seteq {\PPi([\ii_0])}$ of type $A_{2n}$.
\begin{enumerate}
\item From $\Upsilon_{[\ii_0]}$ to  $\Gamma_Q$ :  \\
By deleting vertices in the $n+1$-th residue of  $\Upsilon_{[\ii_0]}$ and unifying every  length two paths passing through $n+1$-th residue into an arrow,
we obtain a quiver isomorphic to $\Gamma_Q$. We rename the $m$-th residue $(m \ge n+2)$ to $m-1$-th residue.
\item From   $\Gamma_Q$ to  $\Upsilon_{[\ii_0]}$ :
\begin{enumerate}
\item we put a vertex on every arrow between the $n$-th residue to the $n+1$-th residue,
\item since vertices obtained in (a) positioned in the $n+1$-th residue, the original $m$-th residue for $m\geq n+1$ should be renamed by the $m+1$-th residue,
\item break every arrow with the new vertex in (a) into two arrows, from the $n$-th residue to the $n+1$-th residue and from the $n+1$-th residue to the $n+2$-th residue  (resp. from the $n+2$-th residue to the $n+1$-th residue and from the $n+1$-th residue to the $n$-th residue),
\item put another vertex in the $n+1$-th residue when $\alpha_i$ is a source or a sink $\Upsilon_{[\ii_0]}$.
\end{enumerate}
\end{enumerate}
\end{remark}

\begin{definition} For $[\ii_0] \in \lf \Qd \rf$ such that $\PPi([\ii_0])=[Q]$,
we denote by $\Gamma_Q \cap \Upsilon_{[\ii_0]}$ the set of all vertices in $\Upsilon_{[\ii_0]}$
whose residues are contained in $I_{2n+1} \setminus \{ n+1 \}$ and by $\Upsilon_{[\ii_0]} \setminus \Gamma_Q$ the set of all vertices in
$\Upsilon_{[\ii_0]}$ whose residues are $n+1$.
\end{definition}

\begin{remark}
By Remark \ref{Rem:surgery}, we sometimes identify the vertex in $\Gamma_Q \cap \Upsilon_{[\ii_0]}$ as a vertex in $\Gamma_Q$ also, and call it {\it induced}.
Also,we call a subquiver $\rho$ in $\Upsilon_{[\ii_0]}$ {\it induced} if it contains a vertex in $\Gamma_Q \cap \Upsilon_{[\ii_0]}$ and a subquiver $\rho$
in $\Upsilon_{[\ii_0]}$ {\it totally induced} if all vertices $\rho$ are contained in $\Gamma_Q \cap \Upsilon_{[\ii_0]}$.
\end{remark}

\section{ Twisted AR-quivers}

\subsection{Coordinate system} Since every class $[\ii_0]$ in $\lf \Qd \rf$ is induced from some Dynkin quiver $Q$ of type $A_{2n}$, we can assign
coordinates for vertices in $\Upsilon_{[\ii_0]}$ using those of $\Gamma_Q$. To describe the algorithm for $\Upsilon_{[\ii_0]}$, we use the observations in
Example \ref{Ex: twisted Coxeter}, Example \ref{Ex: non twisted Coxeter} and Remark \ref{Rem:surgery}.
From now on, we call a combinatorial AR-quiver associated to twisted adapted class of type $A_{2n+1}$ as a {\it twisted AR-quiver} for simplicity.

\begin{algorithm} \label{Alg: surgery} \hfill
\begin{enumerate}
\item[{\rm (a)}] For a vertex $\Gamma_Q \cap \Upsilon_{[\ii_0]}$ whose coordinate in $\Gamma_Q$ is $(i,p)$ $(i \in I_{2n})$, we assign a new coordinate $(i^+,p)$, where $i^+=i+1$ if $i>n$ and $i^+=i$ if $i\leq n$.
\item[{\rm (b)}] For a vertex $ \Upsilon_{[\ii_0]} \setminus \Gamma_Q$ whose residue is $n+1$, we assign the coordinate :
\begin{itemize}
\item[{\rm (i)}] $(n+1,p-\frac{1}{2})$ if the vertex is a sink and $(j,p)$ is a vertex
which is connected with the vertex by an arrow in $\Upsilon_{[\ii_0]}$.
\item[{\rm (ii)}] $(n+1,p+\frac{1}{2})$ if the vertex is a source and $(j,p)$ is a vertex
which is connected with the vertex by an arrow in $\Upsilon_{[\ii_0]}$.
\item[{\rm (iii)}] $(n+1,p+\frac{1}{2})$ if the vertex is neither a sink nor a source and
$$ (j,p) \to \text{ the vertex } \to  (j',p+1) \text{ in } \Upsilon_{[\ii_0]}$$
where $\{ j,j'\}=\{ n,n+2 \}$.
\end{itemize}
\end{enumerate}
For $\be \in \Phi^+$, we denote by $\Omega_{[\ii_0]}(\be) \in I \times \Z/2$ the coordinate for $\be$ in $\Upsilon_{[\ii_0]}$
\end{algorithm}

\begin{example} \label{ex: Q<Q>}
The coordinates for $\Upsilon_{[Q^<]}$ and $\Upsilon_{[Q^>]}$ in Example \ref{Ex: twisted Coxeter} can be depicted as follows:
\[ \ \ \ \scalebox{0.5}{\xymatrix@C=2ex@R=1ex{
& 1 & 1\frac{1}{2} &2 & 2\frac{1}{2} &3 & 3\frac{1}{2} &4 & 4\frac{1}{2} &5 & 5\frac{1}{2} &
6 & 6\frac{1}{2} &7 & 7\frac{1}{2} &8 \\
1 && [Q^<] &&  \bullet\ar@{->}[ddrr] && &&  \bullet\ar@{->}[ddrr]  && && \bullet\ar@{->}[ddrr]  && && \\
& &&&&& &&&& &&&& &&\\
2 && && &&\bullet\ar@{->}[ddrr] \ar@{->}[uurr]  && && \bullet \ar@{->}[ddrr]\ar@{->}[uurr] && && \bullet &&  \\
&&&&&& &&&&&&&&\\
3 & & && \bullet \ar@{->}[dr] \ar@{->}[uurr] && && \bullet\ar@{->}[dr] \ar@{->}[uurr]  && && \bullet\ar@{->}[dr] \ar@{->}[uurr]  && && \\
4& \bigstar\ar@{->}[dr]&&\bigstar\ar@{->}[ur]&&\bigstar\ar@{->}[dr]& &\bigstar\ar@{->}[ur]&&\bigstar\ar@{->}[dr]& &\bigstar\ar@{->}[ur]&&\bigstar\ar@{->}[dr]& &&\\
5 && \bullet\ar@{->}[ddrr] \ar@{->}[ur]  && && \bullet\ar@{->}[ddrr] \ar@{->}[ur]  && && \bullet \ar@{->}[ddrr] \ar@{->}[ur] && && \bullet\ar@{->}[ddrr]   &&\\
&&&&&&&&&&&&&&\\
6 && && \bullet\ar@{->}[ddrr] \ar@{->}[uurr]  && && \bullet\ar@{->}[ddrr] \ar@{->}[uurr]  && && \bullet\ar@{->}[ddrr] \ar@{->}[uurr]  && && \bullet  \\
 &&&&&& &&&&&&&&&&\\
7 && \bullet \ar@{->}[uurr]  && && \bullet \ar@{->}[uurr]  && && \bullet  \ar@{->}[uurr] && && \bullet  \ar@{->}[uurr] &&
}} \ \ \ \
 \scalebox{0.5}{\xymatrix@C=2ex@R=1ex{
 & 1 & 1\frac{1}{2} &2 & 2\frac{1}{2} &3 & 3\frac{1}{2} &4 & 4\frac{1}{2} &5 & 5\frac{1}{2} &
6 & 6\frac{1}{2} &7 & 7\frac{1}{2} &8 \\
1 & [Q^>] &&  \bullet\ar@{->}[ddrr] && &&  \bullet\ar@{->}[ddrr]  && && \bullet\ar@{->}[ddrr]  && && \\
& &&&&& &&&& &&&& &&\\
2 & && &&\bullet\ar@{->}[ddrr] \ar@{->}[uurr]  && && \bullet \ar@{->}[ddrr]\ar@{->}[uurr] && && \bullet &&  \\
&&&&&& &&&&&&&&\\
3  & && \bullet \ar@{->}[dr] \ar@{->}[uurr] && && \bullet\ar@{->}[dr] \ar@{->}[uurr]  && && \bullet\ar@{->}[dr] \ar@{->}[uurr]  && && \\
4 &&\bigstar\ar@{->}[ur]&&\bigstar\ar@{->}[dr]& &\bigstar\ar@{->}[ur]&&\bigstar\ar@{->}[dr]& &\bigstar\ar@{->}[ur]&&\bigstar\ar@{->}[dr]& &  \bigstar&\\
5 & \bullet\ar@{->}[ddrr] \ar@{->}[ur]  && && \bullet\ar@{->}[ddrr] \ar@{->}[ur]  && && \bullet \ar@{->}[ddrr] \ar@{->}[ur] && && \bullet\ar@{->}[ddrr] \ar@{->}[ur]  &&\\
&&&&&&&&&&&&&&\\
6 & && \bullet\ar@{->}[ddrr] \ar@{->}[uurr]  && && \bullet\ar@{->}[ddrr] \ar@{->}[uurr]  && && \bullet\ar@{->}[ddrr] \ar@{->}[uurr]  && && \bullet  \\
 &&&&&& &&&&&&&&&&\\
7 & \bullet \ar@{->}[uurr]  && && \bullet \ar@{->}[uurr]  && && \bullet  \ar@{->}[uurr] && && \bullet  \ar@{->}[uurr] &&
}}
 \]
\end{example}

\subsection{Labeling}

\begin{definition} \cite[Definition 1.6]{Oh14A} Fix any class $[\jj_0]$ of $w_0$ of type $A_n$.
\begin{enumerate}
\item[{\rm (a)}] A path in $\Upsilon_{[\jj_0]}$ is {\it $N$-sectional} (resp. {\it $S$-sectional})
if it is a concatenation of upward arrows (resp. downward arrows).
\item[{\rm (b)}] An $N$-sectional (resp. $S$-sectional) path $\rho$ is {\it maximal} if there is no longer
$N$-sectional (resp. $S$-sectional) path containing $\rho$.
\item[{\rm (c)}] For a sectional path $\rho$, the length of $\rho$ is the number of all arrows in $\rho$.
\end{enumerate}
\end{definition}

Recall that, for every $1 \le a \le b \le n$,
$\beta = \sum_{a \le k \le b} \alpha_k$ is a positive root in $\Phi^+$ and every positive root in $\Phi^+$ is of the form.
Thus we sometimes identify $\beta \in \Phi^+$ (and hence vertex in $\Upsilon_{[\jj_0]}$) with the segment $[a,b]$.
For $\beta=[a,b]$, we say $a$ is the {\it first component} of $\beta$ and $b$ is the {\it second component} of $\beta$. If $\beta$ is simple, we write $\beta$ as $[a]$.

\begin{proposition} \label{pro: section shares}
\cite[Proposition 4.5]{OS15} Fix any class $[\jj_0]$ of $w_0$ of type $A_n$.
Let $\rho$ be an $N$-sectional $($resp. $S$-sectional$)$ path in $\Upsilon_{[\jj_0]}$. Then every positive roots contained in
$\rho$ has the same first $($resp. second$)$ component.
\end{proposition}

\begin{theorem} \label{thm: labeling GammaQ}
\cite[Corollary 1.12]{Oh14A} Fix any Dynkin quiver $Q$ of type $A_n$.
For $1 \le i \le n$, $\Gamma_Q$
contains a maximal $N$-sectional path of length $n-i$ once and exactly once whose vertices share $i$ as the first component.
At the same time, $\Gamma_Q$ contains a maximal $S$-sectional path of length $i-1$ once and exactly
once whose vertices share $i$ as the second component.
\end{theorem}

With the above theorem, we can label the vertices of $\Gamma_Q$ without computing like \eqref{eq: computing for label} (see \cite[Remark 1.14]{Oh14A}).
In this subsection, we give an algorithm on the labeling $\Upsilon_{[\ii_0]}$ for $[\ii_0] \in \lf \Qd \rf$,
by only using Proposition \ref{pro: section shares} and Theorem \ref{thm: labeling GammaQ}.

\begin{definition} \label{def: central}
Fix any class $[\ii_0]$ in $\lf \Qd \rf$ of type $A_{2n+1}$ such that $\PPi([\ii_0])=[Q]$.
\begin{enumerate}
\item[{\rm (a)}] A vertex is a {\it central vertex} of $\Upsilon_{[\ii_0]}$ (i) if there exist {\it two} sectional paths which contain the vertex and
the other vertex whose residue is $n+1$, or (ii) if it is contained in $\Upsilon_{[\ii_0]} \setminus \Gamma_Q$.
\item[{\rm (b)}] The full subquiver $\oUp_{[\ii_0]}$ of $\Upsilon_{[\ii_0]}$ consisting of all central vertices is called a
{\it center} of $\Upsilon_{[\ii_0]}$.
\item[{\rm (c)}] The full subquiver $\uUp^{{\rm NE}}_{[\ii_0]}$ (resp. $\uUp^{{\rm SE}}_{[\ii_0]}$,$\uUp^{{\rm NW}}_{[\ii_0]}$ and $\uUp^{{\rm SW}}_{[\ii_0]}$)
of $\Upsilon_{[\ii_0]}$ consisting of all vertices which are not contained in
$\oUp_{[\ii_0]}$ and located in the North-East (resp. South-East, North-West and South-West) part of $\Upsilon_{[\ii_0]}$.
\end{enumerate}
\end{definition}

\begin{remark} \label{rem: cetral in GammaQ}
By \cite[Lemma 3.2.1]{KKK13b}, \cite[(3.2)]{KKK13b} and Theorem \ref{thm: labeling GammaQ}, every induced central vertex in $\Gamma_Q$ is located
at the intersection of a maximal $S$-sectional path and a maximal $N$-sectional path of $\Gamma_Q$ whose lengths are lager than or equal to $n-1$.
Thus an induced central vertex in $\Gamma_Q$ corresponds to a positive root $\beta$ of type $A_{2n}$ which has $\alpha_n$ or $\alpha_{n+1}$ as its {\it support}.
Here the support of $\beta$, denoted by ${\rm supp}(\be)$, is defined by the following way:
$${\rm supp}(\be) \seteq \{ \al_k \ | \ a \le k \le b\} \quad \text{ where } \quad \be=\sum_{a \le k \le b} \al_k.$$
\end{remark}

\begin{definition} \label{def: multiplicity}
For any $\gamma \in \PR \setminus \Pi$,
{\it the multiplicity} of $\gamma$, denoted by $\mathsf{m}(\gamma)$ is a positive integer defined as follows:
$$\mathsf{m}(\gamma) = \max \{\, m_i\, | \, \sum_{i\in I} m_i\alpha_i=\ga\, \}.$$
If $\mathsf{m}(\gamma)=1$, we say that $\gamma$ is {\it multiplicity free}.
\end{definition}

\begin{example} For $[\ii_0]=[Q^<]$ in Example (\ref{ex: Q<Q>}), we can decompose $\Upsilon_{[\ii_0]}$ into
$\uUp^{{\rm NE}}_{[\ii_0]}(\heartsuit)$, $\uUp^{{\rm SE}}_{[\ii_0]}(\square)$,$\uUp^{{\rm NW}}_{[\ii_0]}(\diamondsuit)$, $\uUp^{{\rm SW}}_{[\ii_0]}(\triangle)$ and
$\oUp_{[\ii_0]}(\bullet,\bigstar)$ as follows:
$$  \scalebox{0.65}
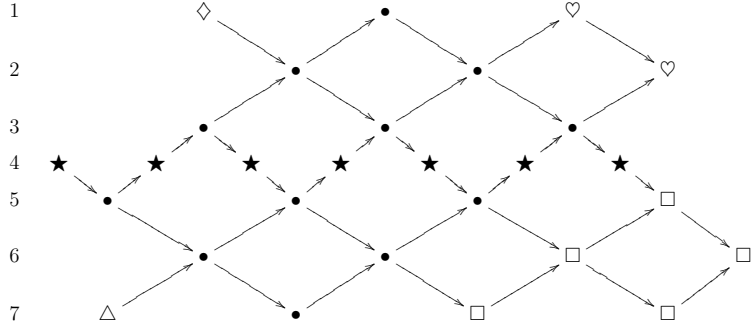
{\xymatrix@C=2ex@R=1ex{
1 &&&&  \diamondsuit\ar@{->}[ddrr] && &&  \bullet\ar@{->}[ddrr]  && && \heartsuit\ar@{->}[ddrr]  && && \\
& &&&&& &&&& &&&& &&\\
2 && && &&\bullet\ar@{->}[ddrr] \ar@{->}[uurr]  && && \bullet \ar@{->}[ddrr]\ar@{->}[uurr] && && \heartsuit &&  \\
&&&&&& &&&&&&&&\\
3 & & && \bullet \ar@{->}[dr] \ar@{->}[uurr] && && \bullet\ar@{->}[dr] \ar@{->}[uurr]  && && \bullet\ar@{->}[dr] \ar@{->}[uurr]  && && \\
4& \bigstar\ar@{->}[dr]&&\bigstar\ar@{->}[ur]&&\bigstar\ar@{->}[dr]& &\bigstar\ar@{->}[ur]&&\bigstar\ar@{->}[dr]& &\bigstar\ar@{->}[ur]&&\bigstar\ar@{->}[dr]& &&\\
5 && \bullet\ar@{->}[ddrr] \ar@{->}[ur]  && && \bullet\ar@{->}[ddrr] \ar@{->}[ur]  && && \bullet \ar@{->}[ddrr] \ar@{->}[ur] && && \square\ar@{->}[ddrr]   &&\\
&&&&&&&&&&&&&&\\
6 && && \bullet\ar@{->}[ddrr] \ar@{->}[uurr]  && && \bullet\ar@{->}[ddrr] \ar@{->}[uurr]  && && \square\ar@{->}[ddrr] \ar@{->}[uurr]  && && \square  \\
 &&&&&& &&&&&&&&&&\\
7 && \triangle \ar@{->}[uurr]  && && \bullet \ar@{->}[uurr]  && && \square  \ar@{->}[uurr] && && \square  \ar@{->}[uurr] &&
}}
$$
\end{example}

Note that, for any Dynkin quiver $Q$ of type $A_{2n}$, the number of $n$ and $n+1$ appear in $[Q]$ are the same as $n$ or $n+1$ whose sum is $2n+1$. Thus, the equation \ref{eq: 4cases} can be depicted in the $\{n,n+1,n+2\}$-th residues of
$\Upsilon_{[\ii_0]}$ $([\ii_0] \in  \lf \Qd \rf)$ by the one of four followings:
\begin{align}\label{eq: four situ}
\begin{cases}
\raisebox{1.2em}{\scalebox{0.5}{\xymatrix@C=2ex@R=1ex{
n &  &  &&\bullet \ar@{->}[dr]&&&& \cdots\cdots &&&&\bullet\ar@{->}[dr] \\
n+1 & \bigstar\ar@{->}[dr] && \bigstar \ar@{->}[ur] && \bigstar \ar@{->}[dr] &&&&&& \bigstar\ar@{->}[ur]&&\bigstar\ar@{->}[dr]\\
n+2 & & \bullet \ar@{->}[ur] &&&& \bullet && \cdots\cdots  && \bullet \ar@{->}[ur]&&&&\bullet \\
}}} & \text{(1) } \ \Upsilon_{[Q^<]} \ \text{ for } \ Q \ \text{ such that } \ {\xymatrix@R=3ex{ \bullet \ar@{-}[r]_<{ n \ } &  \bullet \ar@{->}[l]^<{ \ n+1} }}, \\
\raisebox{1.2em}{\scalebox{0.5}{\xymatrix@C=2ex@R=1ex{
n   &&&\bullet \ar@{->}[dr]&&&& \cdots\cdots &&&&\bullet\ar@{->}[dr] \\
n+1 && \bigstar \ar@{->}[ur] && \bigstar \ar@{->}[dr] &&&&&& \bigstar\ar@{->}[ur]&&\bigstar\ar@{->}[dr] && \bigstar \\
n+2 & \bullet \ar@{->}[ur] &&&& \bullet && \cdots\cdots  && \bullet \ar@{->}[ur]&&&&\bullet \ar@{->}[ur] \\
}}} & \text{(2) } \ \Upsilon_{[Q^>]} \ \text{ for } \ Q \ \text{ such that } \ {\xymatrix@R=3ex{ \bullet \ar@{-}[r]_<{ n \ } &  \bullet \ar@{->}[l]^<{ \ n+1} }},\\
\raisebox{1.2em}{\scalebox{0.5}{\xymatrix@C=2ex@R=1ex{
n   &&\bullet \ar@{->}[dr]&&&& \cdots\cdots &&&&\bullet\ar@{->}[dr] &&&&\bullet \\
n+1 & \bigstar \ar@{->}[ur] && \bigstar \ar@{->}[dr] &&&&&& \bigstar\ar@{->}[ur]&&\bigstar\ar@{->}[dr] && \bigstar\ar@{->}[ur] \\
n+2 &&&& \bullet && \cdots\cdots  && \bullet \ar@{->}[ur]&&&&\bullet \ar@{->}[ur] \\
}}} & \text{(3) } \ \Upsilon_{[Q^<]} \ \text{ for } \ Q \ \text{ such that } \ {\xymatrix@R=3ex{ \bullet \ar@{->}[r]_<{ n \ } &  \bullet \ar@{-}[l]^<{ \ n+1} }},\\
\raisebox{1.2em}{\scalebox{0.5}{\xymatrix@C=2ex@R=1ex{
n   &\bullet \ar@{->}[dr]&&&& \cdots\cdots &&&&\bullet\ar@{->}[dr] &&&&\bullet \ar@{->}[dr]\\
n+1 && \bigstar \ar@{->}[dr] &&&&&& \bigstar\ar@{->}[ur]&&\bigstar\ar@{->}[dr] && \bigstar\ar@{->}[ur] && \bigstar \\
n+2 &&& \bullet && \cdots\cdots  && \bullet \ar@{->}[ur]&&&&\bullet \ar@{->}[ur] \\
}}} & \text{(4) } \ \Upsilon_{[Q^>]} \ \text{ for } \ Q \ \text{ such that } \ {\xymatrix@R=3ex{ \bullet \ar@{->}[r]_<{ n \ } &  \bullet \ar@{-}[l]^<{ \ n+1} }}.
\end{cases}
\end{align}

\begin{remark} \label{rem: directions}
By the above observation, \cite[Lemma 3.2.1]{KKK13b} and Theorem \ref{thm: labeling GammaQ}, we can conclude the followings:
\begin{eqnarray} &&
  \parbox{80ex}{ $\uUp^{{\rm NE}}_{[\ii_0]}$ and $\uUp^{{\rm SW}}_{[\ii_0]}$ (resp. $\uUp^{{\rm SE}}_{[\ii_0]}$ and $\uUp^{{\rm NW}}_{[\ii_0]}$)
consist of totally induced maximal $S$-sectional (resp. $N$-sectional) paths whose length are less than $n$
and every totally induced maximal $S$-sectional  (resp. $N$-sectional) paths appear in $\uUp^{{\rm NE}}_{[\ii_0]}$ or $\uUp^{{\rm SW}}_{[\ii_0]}$
(resp. $\uUp^{{\rm SE}}_{[\ii_0]}$ or $\uUp^{{\rm NW}}_{[\ii_0]}$).
}\label{eq: totally induced}
\end{eqnarray}

Note that $\Upsilon_{[\ii_0]}$ consists of induced maximal $N$-sectional (resp. $S$-sectional) paths and non-induced vertices.
Also every non-induced vertex is contained in some induced maximal sectional path whose length is larger than or equal to $n$
and is {\it not} located in an intersection of an induced maximal $N$-sectional path and an induced maximal $S$-sectional path.
Furthermore, the lengths of induced maximal $N$-sectional (resp. $S$-sectional) paths are
\begin{eqnarray} &&
  \parbox{80ex}{
\begin{itemize}
\item[{\rm (a)}] $\overbrace{ 0,1,\ldots,n-1}^{\text{totally induced}},n+1,n+2,\ldots,2n$
(resp. $\overbrace{ 0,1,\ldots,n-2}^{\text{totally induced}},n,n+1\ldots,2n$) in (1) (resp. (4)) of \eqref{eq: four situ},
\item[{\rm (b)}] $\overbrace{ 0,1,\ldots,n-2}^{\text{totally induced}},n,n+1,\ldots,2n$
(resp. $\overbrace{ 0,1,\ldots,n-1}^{\text{totally induced}},n+1,n+2,\ldots,2n$) in (2) (resp. (3)) of \eqref{eq: four situ}.
\end{itemize}
}\label{eq: total non-total}
\end{eqnarray}
We also remark here that induced maximal $N$-sectional (resp. $S$-sectional) of length $k$ exists uniquely, if it exists.
\end{remark}

By Theorem \ref{thm: OS14}, we can get a reduced word $\ii_0' \in [\ii_0]$ by reading residues of vertices in $\Upsilon_{[\ii_0]}$
by the following order
\begin{align} \label{eq: order}
\{ \uUp^{{\rm NE}}_{[\ii_0]}, \ \uUp^{{\rm SE}}_{[\ii_0]} \}, \ \{ \oUp_{[\ii_0]} \} \text{ and } \{ \uUp^{{\rm NW}}_{[\ii_0]}, \uUp^{{\rm SW}}_{[\ii_0]} \}.
\end{align}

Note that
\begin{eqnarray} &&
  \parbox{80ex}{
\begin{enumerate}
\item[{\rm (a)}] all residues in $\uUp^{{\rm NE}}_{[\ii_0]}$ and $\uUp^{{\rm NW}}_{[\ii_0]}$ are less than or equal to $n$,
\item[{\rm (b)}] all residues in $\uUp^{{\rm SE}}_{[\ii_0]}$ and $\uUp^{{\rm SW}}_{[\ii_0]}$ are larger than or equal to $n+2$,
\item[{\rm (c)}] $\uUp^{{\rm NE}}_{[\ii_0]}$,$\uUp^{{\rm NW}}_{[\ii_0]}$, $\uUp^{{\rm SE}}_{[\ii_0]}$ $\uUp^{{\rm SW}}_{[\ii_0]} \subset \Gamma_Q \cap \Upsilon_{[\ii_0]}$
where $\PPi([\ii_0])=[Q]$.
\end{enumerate}
}\label{eq: from Q labelling}
\end{eqnarray}

By Remark \ref{Rem:surgery}, Theorem \ref{thm: labeling GammaQ},  Remark \ref{rem: directions} and \eqref{eq: from Q labelling}, we have the following lemma:

\begin{lemma} \label{lem: comp for length k le n} \hfill
\begin{enumerate}
\item[{\rm (1)}] If a vertex $v \in \uUp^{{\rm NE}}_{[\ii_0]}$ and the labeling of $v$ in $\Gamma_Q$ is the same as $[a,b]$ $(b \le n)$, then
the labeling of $v$ in $\Upsilon_{[\ii_0]}$ is $[a,b]$.
\item[{\rm (2)}] If a vertex $v \in \uUp^{{\rm SE}}_{[\ii_0]}$ and the labeling of $v$ in $\Gamma_Q$ is the same as $[a,b]$ $(a \ge n+1)$, then
the labeling of $v$ in $\Upsilon_{[\ii_0]}$ is $[a+1,b+1]$.
\item[{\rm (3)}] If a vertex $v \in \uUp^{{\rm NW}}_{[\ii_0]}$ and the labeling of $v$ in $\Gamma_Q$ is the same as $[a,b]$ $(a \ge n+1)$, then
the labeling of $v$ in $\Upsilon_{[\ii_0]}$ is $[a+1,b+1]$.
\item[{\rm (4)}] If a vertex $v \in \uUp^{{\rm SW}}_{[\ii_0]}$ and the labeling of $v$ in $\Gamma_Q$ is the same as $[a,b]$ $(b \le n)$, then
the labeling of $v$ in $\Upsilon_{[\ii_0]}$ is $[a,b]$.
\end{enumerate}
\end{lemma}

\begin{proof}
(1) By reading $\uUp^{{\rm NE}}_{[\ii_0]}$ first in \eqref{eq: order}, the labeling for $v$ in $\Upsilon_{[\ii_0]}$ should be the same as that in $\Gamma_Q$.
since $\RR$ is an identity map for $i<n$.

(2) By reading $\uUp^{{\rm SE}}_{[\ii_0]}$ first in \eqref{eq: order}, the labeling for $v$ in $\Upsilon_{[\ii_0]}$ should be shifted by one, since the image of $i>n+1$
via $\RR$ is the same as $i+1$.

The remained assertions follow from $\ii_0^{\rm rev}$ where $\ii_0^{\rm rev}=i_{l}i_{l-1} \cdots i_1$ for $\ii_0=i_1 i_2 \cdots i_{l}$.
\end{proof}

\begin{remark}
Note that all vertices in a maximal $N$-sectional (resp. $S$-sectional) path of length $k$ have labels $[2n-k,b]$ for $b \ge 2n-k$ (resp. $[b,k+1]$ for $b \le k+1$)
in $\Gamma_Q$. Hence every vertex has a label of the form $[2n+1-k,b]$ $(b \le n)$ $\uUp^{{\rm NW}}_{[\ii_0]}$ or $\uUp^{{\rm SE}}_{[\ii_0]}$
(resp. $[b,k]$ $(b \ge n)$ in $\uUp^{{\rm NE}}_{[\ii_0]}$ or $\uUp^{{\rm SW}}_{[\ii_0]}$). In other words,
we can label  all vertices in $\uUp^{{\rm NE}}_{[\ii_0]}$, $\uUp^{{\rm SW}}_{[\ii_0]}$, $\uUp^{{\rm NW}}_{[\ii_0]}$ or $\uUp^{{\rm SE}}_{[\ii_0]}$
by using \cite[Remark 1.14]{Oh14A}. Furthermore, we can say that every totally induced sectional path keeps the rule in Theorem \ref{thm: labeling GammaQ}:
\begin{eqnarray*} &&
\parbox{80ex}{
If a maximal $N$-sectional path of length $k < n$ exits in $\Upsilon_{[\ii_0]}$, its vertices share $2n+1-k$ as the first component. \\
If a maximal $S$-sectional path of length $k < n$ exits in $\Upsilon_{[\ii_0]}$, its vertices share $k+1$ as the second component.
}
\end{eqnarray*}
\end{remark}

\begin{proposition} \hfill \label{prop: comp for length k ge n}
\begin{enumerate}
\item[{\rm (1)}] Every vertex in a maximal $N$-sectional path whose length is $k\ge n$ has $2n+1-k$ as the first component.
\item[{\rm (2)}] Every vertex in a maximal $S$-sectional path whose length is $k\ge n$ has $k+1$ as the second component.
\end{enumerate}
\end{proposition}

\begin{proof}
Note that we have a maximal $N$-sectional path whose length is $2n$. By Proposition \ref {pro: section shares}, its first component should be $1$.
Then, by the induction on length, our assertion easily follows. The second assertion follows in the same way.
\end{proof}

By Lemma \ref{lem: comp for length k le n} and Proposition \ref{prop: comp for length k ge n}, we have the following theorem:
\begin{theorem} \hfill \label{them: comp for length k ge 0}
\begin{enumerate}
\item[{\rm (1)}] Every induced maximal $N$-sectional path whose length is $k$ shares $2n+1-k$ as a first component.
\item[{\rm (2)}] Every induced maximal $S$-sectional path whose length is $k$ shares $k+1$ as a second component.
\end{enumerate}
Hence, for every vertex in $\Upsilon_{[\ii_0]}$, we can label it as $[a,b] \in \Phi^+$ for some $1 \le a \le b \le 2n+1$.
\end{theorem}

\begin{proof}
(1) and (2) are immediate consequences of Lemma \ref{lem: comp for length k le n} and Proposition \ref{prop: comp for length k ge n}.
The last assertion follows from the argument:
(a) By Remark \ref{rem: cetral in GammaQ} and Remark \ref{rem: directions}, every induced central vertex in $\oUp_{[\ii_0]}$ is located at the intersection
of two maximal induced (but not totally induced) sectional paths whose lengths are larger $n$ and hence we can label them as $[a,b]$ for some $1 \le a \le b \le 2n+1$.
(b) By (a), only one vertex for each sectional path whose length is lager $k>0$ have not determined.
Due to the system of $\Phi^+$, we can label the remained ones also.
\end{proof}

\begin{corollary} \label{cor: label for non-induced} \hfill
\begin{enumerate}
\item[{\rm (a)}] For the case of (1) or (4) in \eqref{eq: four situ}, every positive root whose second component (resp. first component) is $n$ (resp. $n+1$) corresponds to a non-induced vertex
contained in a maximal $N$-sectional of length $k \ge n$ (resp. a maximal $S$-sectional of length $k \ge n+1$).
\item[{\rm (a)}]   For the case of (2) or (3) in \eqref{eq: four situ}, every positive root whose second component (resp. first component) is $n+1$ (resp. $n+2$) corresponds to a non-induced vertex
contained in a maximal $N$-sectional of length $k \ge n+1$ (resp. a maximal $S$-sectional of length $k \ge n+2$).
\end{enumerate}
\end{corollary}

\begin{proof}
{\rm (a)} By \eqref{eq: total non-total} and Lemma \ref{lem: comp for length k le n}, all labels of the form $[a,b]$ $(b \le n-1)$
(resp. $[a,b]$ $(a \ge n+2)$) appear in $\uUp^{{\rm NE}}_{[\ii_0]}$ or $\uUp^{{\rm SW}}_{[\ii_0]}$ (resp. $\uUp^{{\rm NW}}_{[\ii_0]}$ or $\uUp^{{\rm SE}}_{[\ii_0]}$).
By Theorem \ref{them: comp for length k ge 0} and \eqref{eq: total non-total}, every induced central vertex is located
at an intersection of maximal $S$-sectional path of length $k \ge n$ and an intersection of maximal $N$-sectional path of length $l \ge n+1$.
Hence the set of all induced central vertices has one to one correspondence with the set of all positive roots of the form $[2n+1-l,k+1]$ $(k \ge n, \ l \ge n+1)$.
Hence our assertion follows from Theorem \ref{them: comp for length k ge 0}.

The assertion for {\rm (b)} follows in the similar way.
\end{proof}

The following corollary follows from {\rm (a)} in Definition \ref{def: central}
and Corollary \ref{cor: label for non-induced}:

\begin{corollary} \label{cor: supports for induced central} \hfill
\begin{enumerate}
\item[{\rm (a)}]  For the case of (1) or (4) in \eqref{eq: four situ}, every induced central vertex in $\Upsilon_{[\ii_0]}$ corresponds $\be$ having
$\al_n$ and $\al_{n+1}$ as its support if (1) or (4) in \eqref{eq: four situ}.
\item[{\rm (b)}]  For the case of (2) or (3) in \eqref{eq: four situ}, every induced central vertex in $\Upsilon_{[\ii_0]}$ corresponds $\be$ having
$\al_{n+1}$ and $\al_{n+2}$ as its support if (2) or (3) in \eqref{eq: four situ}.
\end{enumerate}
\end{corollary}

\begin{definition} \label{def: folded multiplicity}
For any $\gamma \in \PR$,
{\it the folded multiplicity} of $\gamma$, denoted by $\overline{\mathsf{m}}(\gamma)$ is a positive integer defined as follows:
$$\overline{\mathsf{m}}(\gamma) = \max\left\{\, \sum_{j \in \overline{i}} m_j\, \big| \, \overline{i} \in \overline{I} \text{ and } \sum_{i\in I} m_i\alpha_i=\ga\, \right\}.$$
If $\overline{\mathsf{m}}(\gamma)=1$, we say that $\gamma$ is {\it folded multiplicity free}.
\end{definition}

\begin{remark} \label{rem: corr}
By Corollary \ref{cor: supports for induced central}, every folded multiplicity non-free positive root corresponds to an induced central vertex.
However, every induced central vertex does not correspond to a folded multiplicity non-free positive root.
\end{remark}

For an induced vertex in $\Upsilon_{[\ii_0]}$, we can summarize as follows:
\begin{corollary} \label{cor:label1}
Consider the map $\iota^+: I_{2n}\to I_{2n+1}$ such that $\iota^+(i)=i$ for $i=1, \cdots, n$ and $\iota^+(i)= i+1$ for $i=n+1, \cdots, 2n.$
Then the labeling for the induced vertex $v$ in $\Upsilon_{[\ii_0]}$ corresponding to $[a,b]$ in $\Gamma_Q$ is determined as follows:
\begin{align}\label{eq: induced label}
\left\{ \begin{array}{ll}  \sum_{i=a}^{b} \alpha_{\iota^+(i)} +\alpha_{n+1} & \text{ if }  v \in  \oUp_{[\ii_0]}, \\
 \sum_{i=a}^{b} \alpha_{\iota^+(i)} & \text{ otherwise.}    \end{array}  \right.
\end{align}
\end{corollary}

Now we shall give an example on the labeling without computation, for readers convenience:

\begin{example} \label{ex: label}
For a Dynkin quiver $Q = {\xymatrix@R=3ex{ \bullet
\ar@{->}[r]_<{ \ 1} &  \bullet
\ar@{<-}[r]_<{ \ 2}  &  \bullet
\ar@{->}[r]_<{ \ 3} &\bullet \ar@{->}[r]_<{ \ 4}
&\bullet \ar@{<-}[r]_<{ \ 5}
& \bullet \ar@{-}[l]^<{\ \ \ \ \ \ 6} }},$
let us consider the combinatorial AR-quiver for $\Upsilon_{[Q^<]}$
$$ \scalebox{0.6}{\xymatrix@C=1ex@R=1ex{
1 && &&  \bullet\ar@{->}[ddrr] && &&  \bullet\ar@{->}[ddrr]  && && \bullet\ar@{->}[ddrr]  && && \\
& &&&&& &&&& &&&& &&\\
2 && && &&\bullet\ar@{->}[ddrr] \ar@{->}[uurr]  && && \bullet \ar@{->}[ddrr]\ar@{->}[uurr] && && \bullet &&  \\
&&&&&& &&&&&&&&\\
3 & & && \bullet \ar@{->}[dr] \ar@{->}[uurr] && && \bullet\ar@{->}[dr] \ar@{->}[uurr]  && && \bullet\ar@{->}[dr] \ar@{->}[uurr]  && && \\
4& \bigstar\ar@{->}[dr]&&\bigstar\ar@{->}[ur]&&\bigstar\ar@{->}[dr]& &\bigstar\ar@{->}[ur]&&\bigstar\ar@{->}[dr]& &\bigstar\ar@{->}[ur]&&\bigstar\ar@{->}[dr]& &&\\
5 && \bullet\ar@{->}[ddrr] \ar@{->}[ur]  && && \bullet\ar@{->}[ddrr] \ar@{->}[ur]  && &&\bullet \ar@{->}[ddrr] \ar@{->}[ur] && && \bullet\ar@{->}[ddrr]   &&\\
&&&&&&&&&&&&&&\\
6 && && \bullet\ar@{->}[ddrr] \ar@{->}[uurr]  && && \bullet\ar@{->}[ddrr] \ar@{->}[uurr]  && && \bullet\ar@{->}[ddrr] \ar@{->}[uurr]  && &&\bullet  \\
 &&&&&& &&&&&&&&&&\\
7 && \bullet \ar@{->}[uurr]  && && \bullet \ar@{->}[uurr]  && && \bullet \ar@{->}[uurr] && &&\bullet \ar@{->}[uurr]
}}.$$
which is the case {\rm (1)} in \eqref{eq: four situ}. By  Theorem \ref{them: comp for length k ge 0},
we can complete finding labels for $\Upsilon_{[Q^<]}$ in three steps as follows:
\begin{align*}
& \scalebox{0.45}{\xymatrix@C=1ex@R=1ex{
1 && &&  [7]\ar@{->}[ddrr] && &&  [*,6]\ar@{->}[ddrr]  && && [*,2]\ar@{->}[ddrr]  && && \\
& &&&&& &&&& &&&& &&\\
2 && && &&[*,7]\ar@{->}[ddrr] \ar@{->}[uurr]  && && [*,6] \ar@{->}[ddrr]\ar@{->}[uurr] && && [2] &&  \\
&&&&&& &&&&&&&&\\
3 & & && [*,5] \ar@{->}[dr] \ar@{->}[uurr] && && [*,7]\ar@{->}[dr] \ar@{->}[uurr]  && && [*,6]\ar@{->}[dr] \ar@{->}[uurr]  && && \\
4& [*,4] \ar@{->}[dr]&& \bigstar \ar@{->}[ur]&&[*,5]\ar@{->}[dr]& & \bigstar \ar@{->}[ur]&&[*,7]\ar@{->}[dr]& &\bigstar \ar@{->}[ur]&&[*,6]\ar@{->}[dr]& &&\\
5 && [*,4]\ar@{->}[ddrr] \ar@{->}[ur]  && && [*,5]\ar@{->}[ddrr] \ar@{->}[ur]  && &&[*,7] \ar@{->}[ddrr] \ar@{->}[ur] && && [*,6]\ar@{->}[ddrr]   &&\\
&&&&&&&&&&&&&&\\
6 && && [*,4]\ar@{->}[ddrr] \ar@{->}[uurr]  && && [*,5]\ar@{->}[ddrr] \ar@{->}[uurr]  && && [*,7]\ar@{->}[ddrr] \ar@{->}[uurr]  && &&[*,6]  \\
 &&&&&& &&&&&&&&&&\\
7 && [*,1] \ar@{->}[uurr]  && && [*,4] \ar@{->}[uurr]  && && [*,5] \ar@{->}[uurr] && &&[*,7] \ar@{->}[uurr] &&
}}
 \scalebox{0.45}{\xymatrix@C=1ex@R=1ex{
1 && &&  [7]\ar@{->}[ddrr] && &&  [3,6]\ar@{->}[ddrr]  && && [1,2]\ar@{->}[ddrr]  && && \\
& &&&&& &&&& &&&& &&\\
2 && && &&[3,7]\ar@{->}[ddrr] \ar@{->}[uurr]  && && [1,6] \ar@{->}[ddrr]\ar@{->}[uurr] && && [2] &&  \\
&&&&&& &&&&&&&&\\
3 & & && [3,5] \ar@{->}[dr] \ar@{->}[uurr] && && [1,7]\ar@{->}[dr] \ar@{->}[uurr]  && && [2,6]\ar@{->}[dr] \ar@{->}[uurr]  && && \\
4& [*,4] \ar@{->}[dr]&&[3,*]\ar@{->}[ur]&&[*,5]\ar@{->}[dr]& &[1,*]\ar@{->}[ur]&&[*,7]\ar@{->}[dr]& &[2,*]\ar@{->}[ur]&&[*,6]\ar@{->}[dr]& &&\\
5 && [3,4]\ar@{->}[ddrr] \ar@{->}[ur]  && && [1,5]\ar@{->}[ddrr] \ar@{->}[ur]  && &&[2,7] \ar@{->}[ddrr] \ar@{->}[ur] && && [5,6]\ar@{->}[ddrr]   &&\\
&&&&&&&&&&&&&&\\
6 && && [1,4]\ar@{->}[ddrr] \ar@{->}[uurr]  && && [2,5]\ar@{->}[ddrr] \ar@{->}[uurr]  && && [5,7]\ar@{->}[ddrr] \ar@{->}[uurr]  && &&[6]  \\
 &&&&&& &&&&&&&&&&\\
7 && [1] \ar@{->}[uurr]  && && [2,4] \ar@{->}[uurr]  && && [5] \ar@{->}[uurr] && &&[6,7] \ar@{->}[uurr] &&
}} \\
& \qquad\qquad\qquad
\scalebox{0.6}{\xymatrix@C=1ex@R=1ex{
& \frac{1}{2} &\frac{2}{2} &\frac{3}{2} &\frac{4}{2} &\frac{5}{2} &\frac{6}{2} &\frac{7}{2} &\frac{8}{2} &\frac{9}{2}
&\frac{10}{2} &\frac{11}{2} &\frac{12}{2} &\frac{13}{2} & \frac{14}{2} &\frac{15}{2} &\frac{16}{2} \\
1&& &&  [7]\ar@{->}[ddrr] && &&  [3,6]\ar@{->}[ddrr]  && && [1,2]\ar@{->}[ddrr]  && && \\
& &&&&& &&&& &&&& &&\\
2&& && &&[3,7]\ar@{->}[ddrr] \ar@{->}[uurr]  && && [1,6] \ar@{->}[ddrr]\ar@{->}[uurr] && && [2] &&  \\
&&&&&& &&&&&&&&\\
3& & && [3,5] \ar@{->}[dr] \ar@{->}[uurr] && && [1,7]\ar@{->}[dr] \ar@{->}[uurr]  && && [2,6]\ar@{->}[dr] \ar@{->}[uurr]  && && \\
4& [4]\ar@{->}[dr]&&[3]\ar@{->}[ur]&&[4,5]\ar@{->}[dr]& &[1,3]\ar@{->}[ur]&&[4,7]\ar@{->}[dr]& &[2,3]\ar@{->}[ur]&&[4,6]\ar@{->}[dr]& &&\\
5&& [3,4]\ar@{->}[ddrr] \ar@{->}[ur]  && && [1,5]\ar@{->}[ddrr] \ar@{->}[ur]  && &&[2,7] \ar@{->}[ddrr] \ar@{->}[ur] && && [5,6]\ar@{->}[ddrr]   &&\\
&&&&&&&&&&&&&&\\
6&& && [1,4]\ar@{->}[ddrr] \ar@{->}[uurr]  && && [2,5]\ar@{->}[ddrr] \ar@{->}[uurr]  && && [5,7]\ar@{->}[ddrr] \ar@{->}[uurr]  && &&[6]  \\
 &&&&&& &&&&&&&&&&\\
7&& [1] \ar@{->}[uurr]  && && [2,4] \ar@{->}[uurr]  && && [5] \ar@{->}[uurr] && &&[6,7] \ar@{->}[uurr] &&
}}
\end{align*}
\end{example}

\section{Properties of twisted AR-quivers} In this section, we introduce and investigate the twisted additive property on
twisted AR-quiver $\Upsilon_{[\ii_0]}$ for $[\ii_0] \in  \lf \Qd \rf$. Then we study the distance of a pair
$(\al,\be)$ with respect to the bi-lexicographical order $\prec^\tb_{[\ii_0]}$ induced from $\prec_{[\ii_0]}$,
which were studied in \cite{Oh15E} for $\prec^\tb_{[Q]}$.

\subsection{Twisted additive property} Throughout this subsection, $[\ii_0]$ denotes a class in $\lf \Qd \rf$ of type $A_{2n+1}$. The following is an immediate consequence of
Theorem \ref{them: comp for length k ge 0}:

\begin{corollary}\label{cor: big mesh} Assume we have the following rectangle subquiver in $\Upsilon_{[\ii_0]} \colon$
\begin{align} \label{eq:rectangle}
\scalebox{0.9}{{\xy (0,0)*++{\rectangle}; \endxy}}
\end{align}
where $\al,\be,\ga,\eta \in \Phi^+$. Then $\al+\be=\ga+\eta$ More precisely, if we denote $\al=[a,b]$ and $\beta=[a',b']$ then $\gamma=[a,b']$ and $\eta=[a',b].$
\end{corollary}

The following proposition is another immediate consequence of Corollary \ref{cor: label for non-induced}.

\begin{corollary}\label{cor: triangle mesh} Assume we have the following subquivers in $\Upsilon_{[\ii_0]} \colon$
$$
\scalebox{0.6}{\xymatrix@C=2ex@R=1ex{
k & &&& \ga \ar@{->}[dr] &&&&\\
\vdots & && \bullet \ar@{->}[ur]&& \bullet \ar@{->}[dr]&&& \\
\vdots & & \bullet\ar@{.>}[ur]&&&& \bullet \ar@{->}[dr]&&&\\
n+1 & \alpha \ar@{->}[ur] &&&&&& \be
}} \
\scalebox{0.6}{\xymatrix@C=2ex@R=1ex{
n+1 & \alpha \ar@{->}[dr] &&&&&& \be\\
\vdots & & \bullet\ar@{.>}[dr]&&&& \bullet \ar@{->}[ur]&&&\\
\vdots & && \bullet \ar@{->}[dr]&& \bullet \ar@{->}[ur]&&& \\
l & &&& \ga \ar@{->}[ur] &&&&
}}
$$
where $k < n+1 < l$,  the root $\ga$ corresponds to an induced central vertex and $(\al,\be)$ corresponds to a couple of non-induced vertices.
Then we have $\al+\be=\ga.$
\end{corollary}

Thus we have a {\it twisted additive} property, which can be understood as a generalization of the additive property  \eqref{eq: addtive property} of $\Gamma_Q$,
as follows: (cf. Proposition \ref{Prop:twisted_adapted_2})
\begin{theorem} \label{thm:twisted additive}
 For $\beta \in \Phi^+$ with $\Omega_{[\ii_0]}(\be)=(i,p) \in I \times \Z/2$, assume that $\Omega_{[\ii_0]}(\al)=(i, p-2+\delta_{i,n+1})$.
Then we have
\begin{align}\label{eq: twisted addtive property}
\be+\al =  \sum_{\Omega_{[\ii_0]}(\ga)=(j ,q)}  \ga,
\end{align}
where
$$ q=p-\frac{1}{1+\delta_{i,n+1}} \quad \text{ and } \quad  j=\begin{cases}
n \text{ or } n+3 & \text{ if } i =n+2, \\
n-1 \text{ or } n+2 & \text{ if } i =n, \\
i \pm 1 & \text{ otherwise}.
\end{cases} $$
\end{theorem}

\begin{example}
In combinatorial AR-quiver of Example \ref{ex: label}, we have $n=3$. It is not hard to see the twisted additive property holds, for examples,
\[ \, (1)\, [1,7]+[3,5]= [1.5]+[3,7], \quad(2)\,  [1,6]+[3,7]=[3,6]+[1,7], \quad (3)\, [1,3]+[4,5]=[1,5]. \]
Note that if $i$ is the index of $\beta$ in Theorem \ref{thm:twisted additive}, then (1) is an example for $i=n$, (2) is for $i\neq n,n+1, n+2$ and (3) is for $i=n+1.$
\end{example}

From the fact that $^*=^{\vee(2n+1,2)}$ and \eqref{eq: dual Coxeter}, the reflection functor $r_i:[\ii_0] \mapsto [\ii_0]r_i$ can be understood
as a map $\Upsilon_{[\ii_0]} \mapsto \Upsilon_{[\ii_0]r_i}$ described by coordinates (cf. Algorithm \ref{alg: Ref Q}):

\begin{algorithm} \label{alg: tRef Q}
Let $\mathsf{h}^\vee=2n+1$ be a dual Coxeter number of $A_{2n}$ or $B_{n+1}$ and $i$ be a sink of $[\ii_0] \in \lf \Qd \rf$.
\begin{enumerate}
\item[{\rm (A1)}] Remove the vertex $(i,p)$ such that $\Omega_{[\ii_0]}(\al_i)=(i,p)$ and the arrows entering into $(i,p)$ in $\Upsilon_{[\ii_0]}$.
\item[{\rm (A2)}] Add the vertex $(i^*,p-\mathsf{h}^\vee)$ and the arrows to all vertices
whose coordinates are $(\, j\, ,\, p-\mathsf{h}^\vee+1/(1+\delta_{j,n+1}+\delta_{i^*,n+1})\, ) \in \Upsilon_{[\ii_0]}$ where $j$ is an index adjacent to $i^*$ in he Dynkin diagram
$\Delta$ of type $A_{2n}$.
\item[{\rm (A3)}] Label the vertex $(i^*,p-\mathsf{h}^\vee)$ with $\al_i$ and change the labels $\be$ to $s_i(\be)$ for all $\be \in \Upsilon_{[\ii_0]} \setminus \{\al_i\}$.
\end{enumerate}
\end{algorithm}

\subsection{Notions on sequences} In this section, we briefly review the notions on sequences of positive roots
which were mainly introduced in \cite{Mc12,Oh15E}.

\begin{convention} \label{conv: sequence}
Let us choose a reduced expression $\jj_0=i_1i_2 \cdots i_{\N}$ of $w_0 \in W$ and fix a convex total order $\le_{\jj_0}$
induced by $\jj_0$ and a labeling of $\Phi^+$ as follows:
$$\beta^{\jj_0}_k \seteq s_{i_1}\cdots s_{i_{k-1}}\alpha_{i_k} \in \Phi^+, \   \beta^{\jj_0}_k \le_{\jj_0} \beta^{\jj_0}_l \text{ if and only if } k \le l.$$

\begin{enumerate}
\item[{\rm (i)}] We identify a sequence $\um_\tw=(m_1,m_2,\ldots,m_{\N}) \in \Z_{\ge 0}^{\N}$ with
$$(m_1\beta^{\jj_0}_1 ,m_2\beta^{\jj_0}_2,\ldots,m_{\N} \beta^{\jj_0}_{\N}) \in (\rl^+)^{\N},$$
where $\rl^+$ is the positive root lattice.
\item[{\rm (ii)}] For a sequence $\um_{\jj_0}$ and another reduced expression ${\jj_0}'$ of $w_0$, the sequence $\um_{{\jj_0}'}\in\Z_{\ge 0}^{\N}$ is induced from $\um_{\jj_0}$ as follows : (a) Consider $\um_{\jj_0}$  as a sequence of positive roots, (b)
rearranging with respect to $<_{\jj_0'}$, (c)  applying the convention
{\rm (i)}.
\end{enumerate}
For simplicity of notations, we usually drop the script ${\jj_0}$ if there is no fear of confusion.
\end{convention}

\begin{definition} \cite{Mc12} \label{def: redezletb}
For sequences $\um$, $\um' \in \Z_{\ge 0}^{\N}$, we define an order $\le^\tb_{\jj_0}$ as follows:
\begin{eqnarray*}&&
\parbox{65ex}{
$\um'=(m'_1,\ldots,m'_{\N}) <^\tb_{\redex}  \um=(m_1,\ldots,m_{\N}) $ if and only if there exist integers $k$ and $s$ such that $1 \le k \le s \le \N$ satisfying the properties:  \\
 (a) $m_t'=m_t \ ( 1\leq t<k \text{ or } s<t\leq\N),$  (b) $m'_k  <  m_k$,  (c) $m'_{s}  <  m_{s}.$
}
\end{eqnarray*}
\end{definition}

\begin{definition} \cite{Oh15E} \label{def: redezprectb}
For sequences $\um$, $\um' \in \Z_{\ge 0}^{\N}$, we define an order $\prec^\tb_{[\jj_0]}$ as follows:
\begin{eqnarray}&&
\parbox{65ex}{
$\um'=(m'_1,\ldots,m'_{\N}) \prec^\tb_{[\jj_0]}  \um=(m_1,\ldots,m_{\N}) $ if and only if  $\um'_{\jj_0'} <^\tb_{\jj_0'}  \um_{\jj_0'}$
for all reduced words $\jj_0' \in [\jj_0]$.
}\label{eq: redezprectb}
\end{eqnarray}
\end{definition}

\begin{definition} \cite[Definition 1.10]{Oh15E} \hfill
\begin{enumerate}
\item[{\rm (i)}] A sequence $\um=(m_1,m_2,\ldots,m_{\N}) \in \Z_{\ge 0}^{\N}$ is called  {\it basic} if $m_i \le 1$  for all $i$.
\item[{\rm (ii)}] A sequence $\um$ is called a {\it pair} if it is basic and $|\um|\seteq \sum_{i=1}^{\N} m_i=2$.
\item[{\rm (iii)}] A {\it weight} $\wt(\um)$ of a sequence $\um$ is defined by $\sum_{i=1}^{\N} m_i\beta_i \in \rl^+$.
\end{enumerate}
\end{definition}

In this paper, we use the notation $\up$ for a pair sequence. For brevity, we write a pair $\up$ as $(\alpha,\beta) \in (\PR)^2$ or $(\up_{i_1},\up_{i_2})$ such that
$\beta_{i_1}=\alpha$, $\beta_{i_2}=\beta$ and $i_1 < i_2$.

\begin{definition} \cite[Definition 1.13, Definition 1.14]{Oh15E} \hfill
\begin{enumerate}
\item[{\rm (i)}] A pair $\up$ is called $[\jj_0]$-simple if there exists no sequence $\um \in \Z^{\N}_{\ge 0}$ such that
\begin{align}\label{eq: simple pair}
\um \prec^\tb_{[\jj_0]} \up \quad \text{ and } \quad \wt(\um)=\wt(\up).
\end{align}
\item[{\rm (ii)}] A sequence $\um=(m_1,m_2,\ldots,m_{\N}) \in \Z^{\N}_{\ge 0}$ is called $[\jj_0]$-simple if \ber $\um=(m_k \beta^{\jj_0}_k)$ \er or all pairs $(\up_{i_1},\up_{i_2})$
such that $m_{i_1},m_{i_2} >0$ are $[\jj_0]$-simple pairs.
\end{enumerate}
\end{definition}

\begin{definition} \cite[Definition 1.15]{Oh15E} For a given $[\jj_0]$-simple sequence $\us=(s_1,\ldots,s_{\N}) \in \Z^{\N}_{\ge 0}$,
we say that a sequence $\um \in \Z^{\N}_{\ge 0}$ is called a {\it $[\jj_0]$-minimal sequence of $\us$} if
\begin{enumerate}
\item[{\rm (i)}] $\us \prec^\tb_{\jj_0} \um$ and $\wt(\um)=\wt(\us)$,
\item[{\rm (ii)}] there exists no sequence $\um'\in \Z^{\N}_{\ge 0}$ such that it satisfies the conditions in {\rm (i)} and
$$  \us \prec^\tb_{[\jj_0]} \um' \prec^\tb_{[\jj_0]}  \um.$$
\end{enumerate}
\end{definition}

\begin{definition} \cite[Definition 4.8]{Oh15E} \label{def: gdist}
We say that a sequence $\um$ has {\it generalized $[\jj_0]$-distance $k$} $(k \in\Z_{\ge 0})$, denoted by $\gdist_{[\jj_0]}(\um)=k$, if $\um$ is {\it not} $[\jj_0]$-simple and
\begin{enumerate}
\item[{\rm (i)}] there exists a set of non $[\jj_0]$-simple sequences $\{ \um^{(s)} \ | \ 1 \le s \le k, \ \wt(\um^{(s)})=\wt(\um) \}$ such that
\begin{align}\label{eq: gdist seq}
\um^{(1)} \prec^\tb_{[\jj_0]} \cdots \prec^\tb_{[\jj_0]} \um^{(k)}=\um,
\end{align}
\item[{\rm (ii)}] the set of non $[\jj_0]$-simple sequences $\{ \um^{(s)} \}$ has maximal cardinality among sets of sequences satisfying \eqref{eq: gdist seq}.
\end{enumerate}
If $\um$ is $[\jj_0]$-simple, we define $\gdist_{[\jj_0]}(\um)=0$.
\end{definition}

\begin{definition}\cite[Definition 1.19]{Oh15E}
For a non-simple positive root $\gamma \in \PR \setminus \Pi$, the {\it $[\jj_0]$-radius} of $\gamma$, denoted by $\rds_{[\jj_0]}(\gamma)$,
is the integer defined as follows:
$$\rds_{[\jj_0]}(\gamma)=\max({\rm gdist}_{[\jj_0]}(\up) \ | \ \gamma  \prec_{[\jj_0]}^\tb \up \text{ and } \wt(\up)=\gamma).$$
\end{definition}

\begin{definition}\cite[Definition 1.21]{Oh15E} \label{def: jj_0-socle}
For a pair $\up$, the {\it $[\jj_0]$-socle} of $\up$, denoted by $\soc_{[\jj_0]}(\up)$, is a $[\jj_0]$-simple sequence $\us$ such that
$$\wt(\up)=\wt(\us) \quad \text{and} \quad \us \preceq^\tb_{[\jj_0]} \up,$$
if such $\us$ exists uniquely.
\end{definition}

Note that $\mathsf{m}(\gamma) =1$ for every positive root $\ga \in \PR$ of type $A_m$.

\begin{theorem}\cite[Theorem 4.15, Theorem 4.20]{Oh15E}\cite[Theorem 3.4]{Oh14A} \label{thm: known for Q} Let $Q$ be any Dynkin quiver of type $A_m$.
\begin{enumerate}
\item[{\rm (1)}] For any pair $(\al,\be) \in \PR$,  we have $$\gdist_{[Q]}(\al,\be) \le \max\{ \mathsf{m}(\gamma) \ | \ \ga \in \PR \}=1.$$
\item[{\rm (2)}] For any $\gamma \in \PR \setminus \Pi$, we have $$\rds_{[Q]}(\gamma)=\mathsf{m}(\gamma).$$
\item[{\rm (3)}] For any $\up$, $\soc_{[Q]}(\up)$ is well-defined.
\end{enumerate}
\end{theorem}

\subsection{Distance and radius with respect to $[\ii_0] \in \lf \Qd \rf$} \label{subsec: dis radi} In this section we will prove the following theorem:
\begin{theorem} \label{thm: dist upper bound Qd} Take any $[\ii_0] \in \lf\Qd\rf$ of type $A_{2n+1}$.
\begin{enumerate}
\item[{\rm (1)}] For any pair $(\al,\be) \in \PR$, we have $$\gdist_{[\ii_0]}(\al,\be) \le \max\{ \overline{\mathsf{m}}(\gamma) \ | \ \ga \in \PR \}=2.$$
\item[{\rm (2)}] For any $\gamma \in \PR \setminus \Pi$, we have  $$\rds_{[\ii_0]}(\gamma)\le 2.$$ In particular,
$\rds_{[\ii_0]}(\gamma)= 2$ implies that $\overline{\mathsf{m}}(\ga)=2$.
\item[{\rm (3)}] For any $\up$, $\soc_{[\ii_0]}(\up)$ is well-defined.
\end{enumerate}
\end{theorem}

\begin{lemma}
For a non-simple root $\gamma$ corresponding to non-central vertex in $\Upsilon_{[\ii_0]}$, we have
$$\rds_{[\ii_0]}(\gamma)=\overline{\mathsf{m}}(\ga)=1.$$
\end{lemma}

\begin{proof}
Let us assume that $\ga \in \uUp^{{\rm NE}}_{[\ii_0]} \sqcup \uUp^{{\rm SW}}_{[\ii_0]}$. Then, by Lemma \ref{lem: comp for length k le n} and Theorem \ref{them: comp for length k ge 0},
every nonzero component of a  sequence $\um$ with $\wt(\um)=\gamma$ should appear in $\uUp^{{\rm NE}}_{[\ii_0]} \sqcup \uUp^{{\rm SW}}_{[\ii_0]}$. Hence our assertion immediately follows from
Remark \ref{Rem:surgery}, Lemma \ref{cor:label1} and Theorem \ref{thm: known for Q}. We can prove for
$\ga \in \uUp^{{\rm SE}}_{[\ii_0]} \sqcup \uUp^{{\rm NW}}_{[\ii_0]}$ in the similar way.
\end{proof}

\begin{lemma} \label{lem: less than eq 2}
For any $\gamma \in \PR \setminus \Pi$ corresponding to an induced central vertex, $$\rds_{[\ii_0]}(\gamma)\le \overline{\mathsf{m}}(\ga).$$
\end{lemma}

\begin{proof}
By Corollary \ref{cor: label for non-induced} and Corollary \ref{cor: supports for induced central},
there exists a unique pair $(\al^\star,\be^\star)$ corresponding to non-induced central vertices such that $\al^\star+\be^\star=\ga$.
Also, by Remark \ref{Rem:surgery}, Lemma \ref{cor:label1} and \cite[Proposition 4.24]{Oh15E}, other pairs $(\al,\be)$ such that $\al+\be=\ga$
corresponds to induced vertices. Moreover, the pairs of induced vertices are not comparable with respect to $\prec^\tb_{[\ii_0]}$ by Theorem \ref{thm: known for Q}.
Thus our assertion follows from the fact that sometimes $(\al^\star,\be^\star)$ is comparable with a pair which consisting of
induced vertices. In Example \ref{ex: label}, we can see
$$  [3,5] \prec^{\tb}_{[\ii_0]} (\al^\star,\be^\star)=([4,5],[3]) \prec^{\tb}_{[\ii_0]} ([5],[3,4]).$$
By \cite[Theorem 3.2]{Oh14A}, the  non-induced vertices pair of weight $\gamma$ is less than other induced vertices pairs of weight
$\gamma$ with respect to $\prec^{\tb}_{[\ii_0]}$ whenever they are comparable.
Note that there exists a unique induced central vertex which is a folded multiplicity free $\ga$ of height $2$. In that case, $\rds_{[\ii_0]}(\ga)=1$.
\end{proof}

\begin{lemma}
For any $\gamma \in \PR \setminus \Pi$ corresponding to a non-induced central vertex, $$\rds_{[\ii_0]}(\gamma)=\overline{\mathsf{m}}(\ga)=1.$$
\end{lemma}

\begin{proof}
Let us assume that $[\ii_0]$ is the case $(1)$ in \eqref{eq: four situ}. Then $\gamma$ is $[a,n]$ or $[n+1,b]$. We assume further that
$\gamma=[a,n]$. Then every pair for $\ga$ is of the form $\{ [a,b-1],[b,n] \}$. Without loss of generality,  we assume that the pair consisting of $\{ [a,c-1],[c,n] \}$ is less than the pair consisting of  $\{ [a,d-1],[d,n] \}$ with respect to $ \prec^\tb_{[\ii_0]} $. Then there is a path between $[c,n]$ and $[d,n]$.

 Suppose  $[d,n] \prec_{[\ii_0]}[c,n]$. Then there is a path from $[c,n]$ to $[d,n]$. We can take a path goes through two special vertices : (i) the nearest vertex $V_d=[d,*]$ from $[d,n]$ which lies in the $n+2$-th layer, (ii) $V_c=[c,*]$. (i.e. the path $[c,n] \to V_c \to V_d \to [d,n]$.) We know $[d,n] \prec_{[\ii_0]}  [a,c-1]$ so that $V_d \prec_{[\ii_0]} [a,c-1], V_c$. Also, we have $[a, c-1], V_c \prec_{[\ii_0]} [a,d-1]$ by the fact that $V_c \prec_{[\ii_0]} [c,n]$.  Hence the following facts hold:  (i) the pair with $\{[a,c-1], V_c\}$ is less than the pair with $\{[a,d-1],V_d\}$, (ii) the two pairs have the same weight $\gamma\in \Phi^+\backslash \Pi$, (iii) every vertices in both pairs are induced vertices. However, it contradicts to Theorem \ref{thm: known for Q} (2).

 Also, when there is a path from $[d,n]$ to $[c,n]$, we can prove by similar arguments.
 \end{proof}

\begin{proof} [{\bf The first step for  Theorem \ref{thm: dist upper bound Qd}}]
From the above three lemmas, the second assertion of Theorem \ref{thm: dist upper bound Qd} follows. Furthermore, the first and the third assertions
for $\al+\be \in \PR$ also hold.
\end{proof}

\begin{proposition} \cite[Proposition 4.5]{Oh15E}
For a Dynkin quiver $Q$ of type $A_m$ and $(\al,\be)$ with $\al+\be \not\in \PR$ and $\gdist_{[Q]}(\al,\be)=1$, there exists a unique rectangle in $\Gamma_Q$
given as follows:
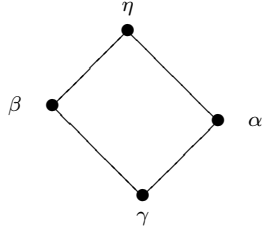
\begin{align}\label{eq:rectangle2}
\rectangle
\end{align}
where $\al+\be=\ga+\eta$ and $(\ga,\eta) \prec^\tb_{Q} (\al,\be)$. Furthermore,
\begin{enumerate}
\item[{\rm (a)}] there is no pair $(\al',\be') \ne (\ga,\eta)$ such that
$(\al',\be') \prec^\tb_{Q} (\al,\be)$ and $\al+\be=\al'+\be'$,
\item[{\rm (b)}] ${\rm supp}(\al) \cap {\rm supp}(\be) \ne \emptyset$.
\end{enumerate}
\end{proposition}

\begin{proposition} \label{prop: less than eq to 2}
For any pair  $(\al,\be)$ of type $A_{2n+1}$ such that $\al+\be \not\in \PR$,
$$\gdist_{[\ii_0]}(\al,\be) \le 2.$$
\end{proposition}

\begin{proof}
For $(\al,\be)$ which satisfies one of the following properties:
\begin{itemize}
\item ${\rm supp}(\al) \cap {\rm supp}(\be) = \emptyset$,
\item their first (resp. second) components are the same,
\item they are incomparable with respect to $\prec_{[\ii_0]}$,
\end{itemize}
one can prove easily that $\gdist_{[\ii_0]}(\al,\be)=0$ by using the convexity of $\prec_{[\ii_0]}$ and Theorem \ref{them: comp for length k ge 0}.
Thus $\gdist_{[\ii_0]}(\al,\be)>0$  implies that there exists a rectangle alike \eqref{eq:rectangle2} or one of $\alpha$ and $\beta$ is a non-induced vertex.
By Remark \ref{Rem:surgery} if there is a  rectangle alike \eqref{eq:rectangle2} then  the rectangle can be classified with the followings:
\begin{itemize}
\item[{\rm (i)}] the rectangle without non-induced vertices on it,
\item[{\rm (ii)}] the rectangle with two non-induced vertices whose first or second component are the same,
\item[{\rm (iii)}] the rectangle with two non-induced vertices whose sum is contained in $\PR$ by Lemma \ref{cor: label for non-induced},
\end{itemize}
For {\rm (i)}, the proofs are the same as in \cite[Proposition 4.5]{Oh15E}. The the cases {\rm (ii)} and {\rm (iii)} can be depicted as follows.
$$
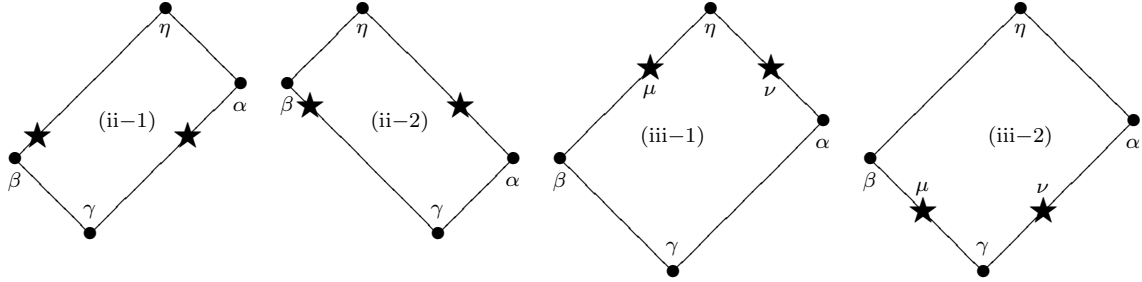
{\xy (0,0)*{}="T1"; (10,-10)*{}="R1";(-20,-20)*{}="L1"; (-10,-30)*{}="B1";
"T1"; "R1" **\dir{-}; "T1"; "L1" **\dir{-}; "L1"; "B1" **\dir{-};,"R1"; "B1" **\dir{-};
"T1"*{\bullet}; "B1"*{\bullet};,"L1"*{\bullet};"R1"*{\bullet};
"L1"+(3,3)*{\bigstar};"B1"+(13,13)*{\bigstar};
"T1"+(-5,-15)*{_{{\rm (ii-1)}}};
"L1"+(0,-3)*{_{\be}};"R1"+(0,-3)*{_{\al}};
"T1"+(0,-3)*{_{\eta}};"B1"+(0,3)*{_{\ga}};
\endxy} \quad
{\xy (0,0)*{}="T2"; (-10,-10)*{}="L2";(20,-20)*{}="R2"; (10,-30)*{}="B2";
"T2"; "R2" **\dir{-}; "T2"; "L2" **\dir{-}; "L2"; "B2" **\dir{-};,"R2"; "B2" **\dir{-};
"T2"*{\bullet}; "B2"*{\bullet};,"L2"*{\bullet};"R2"*{\bullet};
"L2"+(3,-3)*{\bigstar};"T2"+(13,-13)*{\bigstar};
"T2"+(5,-15)*{_{{\rm (ii-2)}}};
"L2"+(0,-3)*{_{\be}};"R2"+(0,-3)*{_{\al}};
"T2"+(0,-3)*{_{\eta}};"B2"+(0,3)*{_{\ga}};
\endxy}
\quad
{\xy (0,0)*{}="T3"; (-20,-20)*{}="L3";(15,-15)*{}="R3"; (-5,-35)*{}="B3";
"T3"; "R3" **\dir{-}; "T3"; "L3" **\dir{-}; "L3"; "B3" **\dir{-};,"R3"; "B3" **\dir{-};
"T3"*{\bullet}; "B3"*{\bullet};,"L3"*{\bullet};"R3"*{\bullet};
"T3"+(-8,-8)*{\bigstar};"T3"+(8,-8)*{\bigstar};
"T3"+(-8,-11)*{_{\mu}};"T3"+(8,-11)*{_{\nu}};
"T3"+(-5,-17)*{_{{\rm (iii-1)}}};
"L3"+(0,-3)*{_{\be}};"R3"+(0,-3)*{_{\al}};
"T3"+(0,-3)*{_{\eta}};"B3"+(0,3)*{_{\ga}};
\endxy}
\quad
{\xy (0,0)*{}="T4"; (-20,-20)*{}="L4";(15,-15)*{}="R4"; (-5,-35)*{}="B4";
"T4"; "R4" **\dir{-}; "T4"; "L4" **\dir{-}; "L4"; "B4" **\dir{-};,"R4"; "B4" **\dir{-};
"T4"*{\bullet}; "B4"*{\bullet};,"L4"*{\bullet};"R4"*{\bullet};
"B4"+(-8,8)*{\bigstar};"B4"+(8,8)*{\bigstar};
"T4"+(0,-17)*{_{{\rm (iii-2)}}};
"L4"+(0,-3)*{_{\be}};"R4"+(0,-3)*{_{\al}};
"T4"+(0,-3)*{_{\eta}};"B4"+(0,3)*{_{\ga}};
"B4"+(-8,11)*{_{\mu}};"B4"+(8,11)*{_{\nu}};
\endxy}
$$
where $\bigstar$ denotes a non-induced vertex.

Let the set of vertices $V$ satisfy (a) every vertex in $V$  is in or on the rectangle (see the above pictures) and (b) the sum of all elements  is $\alpha+\beta.$ Denote by $V_N$ the set of non-induced vertices in $V$ lies on a N-sectional path and by $V_S$  the set of non-induced vertices in $V$ lies on a S-sectional path. Then we note that
\begin{enumerate}
\item[(IND-1)] $|V_N|=|V_S|$,
\item[(IND-2)] if $\alpha_1 \in V_N$ and $\alpha_2\in V_S$ then $\alpha_1+\alpha_2$ is the root in or on the rectangle,
\item[(IND-3)]  if $V$ consists of induced vertices then $V=\{\alpha, \beta\}$ or $V=\{\gamma, \eta\}$.
\end{enumerate}
\vskip 2mm

\noindent {\rm (ii)}  Suppose $V_N\neq \emptyset$ and $V_S \neq \emptyset$. Then $\alpha_1 \in V_N$ and $\alpha_2\in V_S$ satisfy (IND-2) and, moreover, the root  $\alpha'=\alpha_1+\alpha_2\neq \alpha, \beta, \gamma$ or $\delta.$ Consider the set $V\backslash\{\alpha_1, \alpha_2\} \cup \{\alpha'\}$ is another set of vertices satisfying (a) and (b). Hence it contradicts to (IND-3).  In this case, by (IND-3),
$\gdist_{[\ii_0]}(\al,\be)=1$

\vskip 2mm
\noindent {\rm (iii)} Suppose $V_N \cup V_S = \{\mu, \nu\}$ then, for the case {\rm (iii-1)}, $V=\{\gamma, \mu, \nu\}$ (resp.  for the case {\rm (iii-2)}, $V=\{\eta, \mu, \nu\}$). Otherwise the $V\backslash\{\mu, \nu\} \cup \{\eta\}$ (resp. $V\backslash\{\mu, \nu\} \cup \{\gamma\}$) is a contradiction to (IND-3). Conversely, we can prove that $V_N\cup N_S$ should be $\emptyset$ or $\{\mu, \nu\}$ by the same argument. In this case, since the sequence determined by $\{\gamma, \mu, \nu\}$ (resp. $\{\eta, \mu, \nu\}$) is between the pair $(\gamma, \eta)$  and $(\alpha, \beta)$ with respect to $\prec^{\tb}_{[\ii_0]}$, we have $\gdist_{[\ii_0]}=2$.

\vskip 2mm

 Now let $\alpha$ be an induced vertex and $\beta$ be a non-induced vertex.

$$
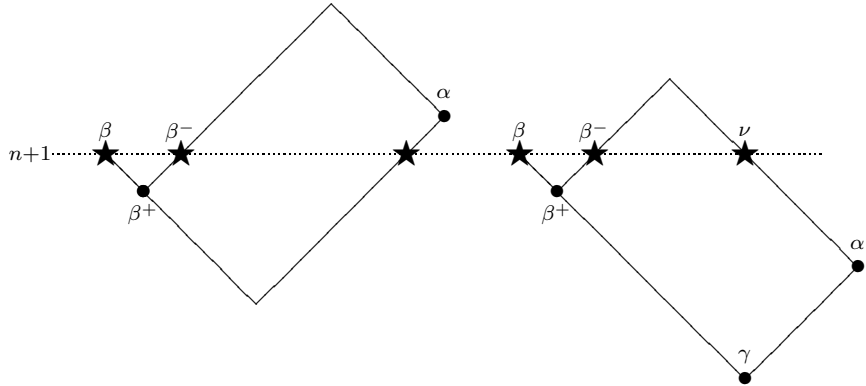
{\xy (0,0)*{}="T1"; (15,-15)*{}="R1";(-25,-25)*{}="L1"; (-10,-40)*{}="B1";
"L1"+(55,0)*{}="L2"; "L2"+(15,15)="T2";
"T2"+(25,-25)="R2";"R2"+(-15,-15)="B2";
"T2"; "R2" **\dir{-}; "T2"; "L2" **\dir{-}; "L2"; "B2" **\dir{-};,"R2"; "B2" **\dir{-};
"L2"+(-5,5); "L2" **\dir{-}; "L2"+(-5,8)*{_{\be}};"L2"+(5,8)*{_{\be^-}};
"L2"+(-5,5)*{\bigstar};"L2"+(5,5)*{\bigstar};
"T1"; "R1" **\dir{-}; "T1"; "L1" **\dir{-}; "L1"; "B1" **\dir{-};,"R1"; "B1" **\dir{-};
"L1"+(-5,5); "L1" **\dir{-}; "L1"+(-5,8)*{_{\be}};"L1"+(5,8)*{_{\be^-}};
"L1"+(-5,5)*{\bigstar};"L1"+(5,5)*{\bigstar};
"L1"+(0,-3)*{_{\be^+}};"L1"*{\bullet};
"R1"*{\bullet};"R1"+(0,3)*{_{\al}};
"R1"+(-5,-5)*{\bigstar};
"T2"+(10,-10)*{\bigstar};
"T2"+(10,-7)*{_{\nu}};
"B2"*{\bullet};"B2"+(0,3)*{_{\ga}};
"R2"*{\bullet};"R2"+(0,3)*{_{\al}};
"L2"+(0,-3)*{_{\be^+}};"L2"*{\bullet};
"L1"*{\bullet};"L1"+(-15,5)*{_{n+1}};
"L1"+(-12,5); "L1"+(90,5) **\dir{.};
\endxy}
$$

 Suppose $\beta \prec_{[\ii_0]} \alpha$ and take the non-induced vertex $\beta^-$ such that  $\beta^-$ is the largest  non-induced root (with respect to $\prec_{[\ii_0]}$) satisfying  $\beta^-\prec_{[\ii_0]}\beta$ . Now let $\beta^+= \beta+\beta^-.$ Then $\beta^+$ satisfies
 (1) $ \beta^+ \prec_{[\ii_0]} \beta$, (2) $\beta'\prec_{[\ii_0]} \beta$ if and only if $\beta'\prec_{[\ii_0]} \beta$ or $\beta'=\beta^+.$
 Since $\alpha$ and $\beta^+$ are induced vertices, we can consider the set $V^+$ of vertices  satisfying (a) and (b) associated to $\alpha$ and $\beta^+$.
 In order to see $\gdist_{[\ii_0]}(\alpha, \beta)$, we need to find a set $W$ of vertices such that (a') every element is in or on the rectangle determined by $\beta^+$
 and $\alpha$ (b') the sum of elements is $\alpha+\beta$. One can easily see $W \cup \{\beta^-\}$ satisfies the conditions for $V^+$. Hence depending on whether
  $(\alpha, \beta^+)$ is a case of  {\rm(ii)} or {\rm(iii)}, we get $\gdist_{[\ii_0]}(\alpha, \beta)=0$ or $1.$
 For the latter case, we have $W= \{\gamma, \nu\}$ and $V^+=\{ \gamma, \beta^-, \nu\}.$ The same argument works for the case $\alpha \prec_{[\ii_0]} \beta.$
\end{proof}

\begin{proof} [{\bf The second step for  Theorem \ref{thm: dist upper bound Qd}}]
From the above propositions, the first and the third assertions complete.
\end{proof}

\begin{remark}
In \cite{Oh14A}, it was shown that for a Dynkin quiver $Q$ of type $A_m$ and a root $\gamma \in \PR \setminus \Pi$, we have
$$\rds_{[Q]}(\gamma)=\mathsf{m}(\gamma).$$
Analogously, if a twisted adapted class $[\ii_0]$ of type $A_{2n+1}$ is associated to a twisted Coxeter element then we have
\begin{equation}\label{Eqn:rds_folded mult}
 \rds_{[\ii_0]}(\gamma)= \overline{\mathsf{m}}(\gamma)
\end{equation}
for any $\gamma\in \PR\setminus \Pi.$  However, if a twisted adapted class $[\ii_0]$ of type $A_{2n+1}$ which is not associated to a twisted Coxeter element then (\ref{Eqn:rds_folded mult}) can be either true or not. More precisely, if
$ \PPi([\ii_0])=[Q] $ for  a Dynkin quiver $Q$ with the subquiver
${\xymatrix@R=3ex{ \bullet \ar@{->}[r]_<{ \ n-1} &\bullet \ar@{->}[r]_<{ \ n}
&\bullet \ar@{->}[r]_<{ \ n+1}
& \bullet \ar@{-}[l]^<{\ \ \ \ \ \ n+2} }}$ or ${\xymatrix@R=3ex{ \bullet \ar@{<-}[r]_<{ \ n-1} &\bullet \ar@{<-}[r]_<{ \ n}
&\bullet \ar@{<-}[r]_<{ \ n+1}
& \bullet \ar@{-}[l]^<{\ \ \ \ \ \ n+2} }}$ then (\ref{Eqn:rds_folded mult}) is not true. If $Q$ does not have such a subquiver then  (\ref{Eqn:rds_folded mult}) is true.
\end{remark}

\begin{remark}
Note that, for each pair $(\al,\be)$ with $\gdist_{[\ii_0]}(\al,\be)=2$, there exists a {\it unique} chain of sequences
$$ \soc_{[\ii_0]}(\al,\be) \prec^\tb_{[\ii_0]} \um \prec^\tb_{[\ii_0]} (\al,\be)$$
which tells $\gdist_{[\ii_0]}(\al,\be)=2$. Furthermore,
\begin{enumerate}
\item[{\rm (1)}] if $\al+\be \in \ga$, then $\um$ is a pair consisting of non-induced central vertices,
\item[{\rm (2)}] if $\al+\be \not \in \ga$, $\um$ is a triple $( \mu,\nu,\eta)$ such that
\begin{itemize}
\item[{\rm (i)}] $\mu+\nu \in \PR$, $(\mu,\nu)$ is an $[\ii_0]$-minimal pair of $\mu+\nu$ and $\al-\mu,\be-\nu \in \PR$,
\item[{\rm (ii)}] $\eta$ is not comparable to $\mu$ and $\nu$ with respect to $\prec_{[\ii_0]}$,
\item[{\rm (iii)}] $\eta=(\al-\mu)+(\be-\nu)$ and $((\al-\mu),(\be-\nu))$ is an $[\ii_0]$-minimal pair for $\eta$,
\item[{\rm (iv)}] $(\al-\mu,\mu)$, $(\nu,\be-\nu)$  are $[\ii_0]$-minimal pairs for $\al$ and $\be$, respectively.
\end{itemize}
In Example \ref{ex: label},
$$ ([1,7],[2,5]) \prec^\tb_{[\ii_0]} ([4,7],[1,3],[2,5]) \prec^\tb_{[\ii_0]} ([2,7],[1,5]).$$
\end{enumerate}
\end{remark}

\subsection{Minimal pairs for $\ga \in \PR \setminus \Pi$ and their coordinates in $\Upsilon_{[\ii_0]}$} \label{subsec: minimal twisted}
In this subsection, we shall see the coordinates of minimal pairs for $\ga \in \PR \setminus \Pi$ in $\Upsilon_{[\ii_0]}$. We briefly
recall one of the main results in \cite{Oh14A}:

\begin{theorem} \cite[Theorem 3.2, Theorem 3.4]{Oh14A} \label{thm upper lower minimal}
For every pair $(\alpha,\beta)$ of $\alpha+\beta=\gamma\in \Phi^+_{A_{2n}}$, we write
$$\phi_Q(\alpha)=(i,p), \quad \phi_Q(\be)=(j,q) \quad \text{ and } \quad \phi_Q(\ga)=(k,z).$$
Then $(\alpha,\beta)$ is $[Q]$-minimal and
\begin{eqnarray} &&
\parbox{79ex}{
\begin{enumerate}
\item[{\rm (i)}] $p-z=|i-k|$ and $q-z=-|j-k|$,
\item[{\rm (ii)}] $i+j=k$ or $(2n+1-i)+(2n+1-j)=(2n+1-k).$
\end{enumerate}
}\label{eq: upper and lower}
\end{eqnarray}
\end{theorem}

Define
$$
i^- =
\begin{cases}
i-1 & \text{ if } i>n+1, \\
i & \text{ if } i \le n+1.
\end{cases}
$$

\begin{lemma} \label{lem: minimal for rds2}
For $\ga \in \PR \setminus \Pi$ with $\Omega_{[\ii_0]}(\ga)=(k,r)$ and $\overline{\mathsf{m}}(\ga)=2$, an $[\ii_0]$-minimal pair $(\al,\be)$ for $\ga$ satisfies
one of the following conditions: Set $\Omega_{[\ii_0]}(\al)=(i,p)$ and $\Omega_{[\ii_0]}(\be)=(j,q)$.
\begin{enumerate}
\item[{\rm (i)}] $i=j=n+1$ such that $p+q=2r$,
\item[{\rm (ii)}] $q=|k^--i^-|+p, \ r=p-|k^--j^-|$, $i,j \ne n+1$ and one of the following holds:
\begin{align} \label{eq: induced minimal}
\begin{cases}
{\rm (a)} \ i+j=k \text{ and } k \le n, \\
{\rm (b)} \ (2n+1-i^-)+(2n+1-j^-)=2n+1-k^-, \  k \le n \text{ and } \min\{ i,j\} \le n, \\
{\rm (c)} \ i^-+j^-=k^-, \  k \ge n+2 \text{ and } \max\{ i,j \} \ge n+2, \\
{\rm (d)} \ (2n+2-i)+(2n+2-j)=2n+2-k \text{ and } k \ge n+2.
\end{cases}
\end{align}
\end{enumerate}
\end{lemma}

\begin{proof}
By Lemma \ref{lem: less than eq 2}, $\ga$ with $\overline{\mathsf{m}}(\ga)=2$ has a unique pair $(\al,\be)$
which consists of non-induced central vertices and is an $[\ii_0]$-minimal of $\ga$. Then the first assertion follows. The other $[\ii_0]$-minimal pairs of $\ga$
are induced from $\Gamma_Q$ and incomparable with the non-induced pair. Then one can easily check that the other $[\ii_0]$-minimal pairs
satisfy one of the four conditions in \eqref{eq: induced minimal} by Theorem \ref{thm upper lower minimal}.
\end{proof}

\begin{lemma} \label{lem: minimal for induced rds1}
For an induced vertex $\ga \in \PR \setminus \Pi$ with $\Omega_{[\ii_0]}(\ga)=(k,r)$ and $\overline{\mathsf{m}}(\ga)=1$,
an $[\ii_0]$-minimal pair $(\al,\be)$ for $\ga$ satisfies one of the following conditions $:$
Set $\Omega_{[\ii_0]}(\al)=(i,p)$ and  $\Omega_{[\ii_0]}(\be)=(j,q)$.
\begin{enumerate}
\item[{\rm (i)}] $i=j=n+1$ such that $p+q=2r$,
\item[{\rm (ii)}] $p-r=|k^- -i^-|, \ q-r=-|k^--j^-|$ and
$$i^-+j^-=k^- \  \text{ or } \ (2n+1-i^-)+(2n+1-j^-)=2n+1-k^-.$$
\end{enumerate}
\end{lemma}
\begin{proof}
Note that the existence of the non-induced pair for $\ga$ is not guaranteed (see Remark \ref{rem: corr}).
In this case, each pair for $\ga$ induced from $\Gamma_Q$ is not comparable with the non-induced pair for $\ga$. Then we can apply the same argument of
the previous lemma.
\end{proof}

\begin{lemma} \label{lem: minimal for non-induced central}
For a non-induced central vertex $\ga \in \PR \setminus \Pi$ with $\Omega_{[\ii_0]}(\ga)=(n+1,r)$,
an $[\ii_0]$-minimal pair $(\al,\be)$ for $\ga$ satisfies one of the following conditions:
Set $\Omega_{[\ii_0]}(\al)=(i,p)$ and  $\Omega_{[\ii_0]}(\be)=(j,q)$.
\begin{enumerate}
\item[{\rm (i)}] $(i,p)=(n+1,r+2\ell)$ and $(j,q)=(\ell,r-\frac{1}{2}-(n-\ell))$,
\item[{\rm (ii)}] $(i,p)=(n+1,r+2\ell)$ and $(j,q)=(2n+2-\ell,r-\frac{1}{2}-(n-\ell))$,
\item[{\rm (iii)}] $(i,p)=(\ell,r+\frac{1}{2}+(n-\ell))$ and $(j,q)=(n+1,r-2\ell)$,
\item[{\rm (iv)}] $(i,p)=(2n+2-\ell,r+\frac{1}{2}+(n-\ell))$ and $(j,q)=(n+1,r-2\ell)$,
\end{enumerate}
for some $\ell\in \Z_{\ge 1}$.
\end{lemma}

\begin{proof}
Let us assume that $\ga=[a,n+1]$ for some $a \le n$ and is contained in the $N$-Sectional path $N[a]$ of length $2n+1-a$. Note that
$\{ \al,\be \} = \{ [a,k], [k+1,n+1] \}$ for some $a \le k <n+1$. We assume further that $\be=[k+1,n+1]$. Equivalently, $\Omega_{[\ii_0]}(\be)=(n+1,r-2\ell)$
for some $\ell \in \Z_{\ge 1}$ by Corollary \ref{cor: label for non-induced}. Note that there exists an $S$-sectional path $S[k]$ of length $k-1$
\begin{itemize}
\item  whose vertices shares $k$ as their second component,
\item  which intersects with $N[a]$.
\end{itemize}
Furthermore, the vertex located at the intersection of $N[a]$ and $S[k]$ is $[a,k]$. By the assumption that $[a,k] \prec_{[\ii_0]} [a,n+1]$,
the $[\ii_0]$-residue $i$ of $[a,k]$ is strictly less than $n+1$ by Theorem \ref{them: comp for length k ge 0}. By applying \cite[Corollary 1.15]{Oh14A}
, Algorithm \ref{Alg: surgery} and Theorem \ref{them: comp for length k ge 0}, we have the following in $\Upsilon_{[\ii_0]}$:
$$
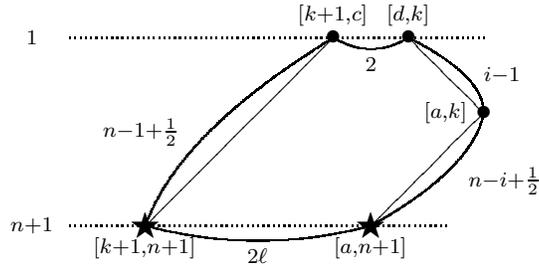
{\xy (0,0)*{}="T1"; (-25,-25)*{}="L1"; (10,0)*{}="T2"; (20,-10)*{}="R1";  (5,-25)*{}="R2";
"T1"; "L1" **\dir{-};"T2"; "R1" **\dir{-};"R1"; "R2" **\dir{-};
"L1"+(-10,0); "R2"+(10,0) **\dir{.}; "L1"+(-15,0)*{_{n+1}};
"T1"+(-35,0); "T2"+(10,0) **\dir{.}; "T1"+(-40,0)*{_{1}};
"T1"*{\bullet};"T2"*{\bullet};"R1"*{\bullet};
"L1"*{\bigstar};"R2"*{\bigstar};
"T1"+(0,3)*{_{[k+1,c]}};"T2"+(0,3)*{_{[d,k]}};
"R1"+(-5,0)*{_{[a,k]}};
"L1"+(0,-3)*{_{[k+1,n+1]}};"R2"+(0,-3)*{_{[a,n+1]}};
"T1"; "T2"; **\crv{(5,-3)}?(.5)+(0,-2)*{\scriptstyle 2};
"L1"; "R2"; **\crv{(-10,-29)}?(.5)+(0,-2)*{_{2\ell}};
"T1"; "L1"; **\crv{(-20.5,-12.5)}?(.5)+(-9,0)*{_{n-1+\frac{1}{2}}};
"R2"; "R1"; **\crv{(18.5,-17.5)}?(.6)+(9,0)*{_{n-i+\frac{1}{2}}};
"T2"; "R1"; **\crv{(20,-5)}?(.5)+(5,0)*{_{i-1}};
\endxy}
$$
Hence we can obtain that $i=\ell$
which yields our third assertion. For the remained cases, one can prove by applying the similar argument.
\end{proof}
\section{Folded AR-quivers}

\subsection{Folding the twisted AR-quiver}
For any Dynkin quiver $Q$ of type $A_{2n}$, we know that $i \not\equiv i^{\vee(2n,2)}\ {\rm mod} \ 2$. Thus,
for any $\al,\be \in \PR$ such that $\Omega_Q(\al)=(i,p)$ and $\Omega_Q(\be)=(2n+1-i,q)$, we have
\begin{align} \label{eq: not equiv}
p \not\equiv q \ {\rm mod} \ 2.
\end{align}
Thus we can {\it fold} a twisted AR-quiver by assigning new coordinates in $\overline{I} \times \Z$ in the following way:
\begin{eqnarray} &&
\parbox{79ex}{
For a positive root (or vertex) $\be$ with $\Omega_{[\ii_0]}(\be)=(i, \frac{p}{2}) \in I \times \Z/2$,
we assign a new coordinate $\widehat{\Omega}_{[\ii_0]}(\be)=(\ov{i},p) \in \ov{I} \times \Z$ without duplication; that means, there exist no
$\be \ne \be' \in \PR$ such that $\widehat{\Omega}_{[\ii_0]}(\be)=\widehat{\Omega}_{[\ii_0]}(\be')$. As a quiver, the folded quiver $\widehat{\Upsilon}_{[\ii_0]}$
is isomorphic to the twisted AR-quiver $\Upsilon_{[\ii_0]}$. We call $\widehat{\Upsilon}_{[\ii_0]}$ {\it the folded quiver} of $\Upsilon_{[\ii_0]}$.
We also call $\ov{i}$ {\it the folded residue} of $\be$ when $\widehat{\Omega}_{[\ii_0]}(\be)=(\ov{i},p)$.
}\label{eq: folded quiver}
\end{eqnarray}

\begin{example} For $\Upsilon_{[\ii_0]}$ in \eqref{ex: label}, its folded quiver $\widehat{\Upsilon}_{[\ii_0]}$ can be drawn as follows:
$$
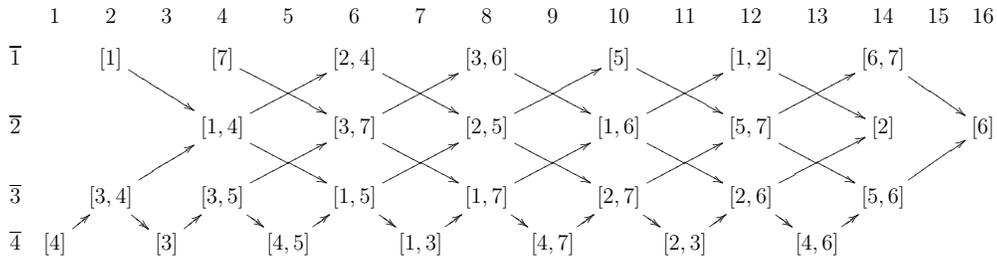
\scalebox{0.7}{\xymatrix@C=1ex@R=1ex{
& 1 & 2 & 3 & 4 & 5 & 6 & 7 & 8 &9 & 10 & 11 & 12 & 13 & 14 & 15 & 16\\
\ov{1}&&[1]\ar@{->}[ddrr]&&  [7]\ar@{->}[ddrr] &&[2,4]\ar@{->}[ddrr]&&  [3,6]\ar@{->}[ddrr]  &&[5]\ar@{->}[ddrr]&& [1,2]\ar@{->}[ddrr]  &&[6,7]\ar@{->}[ddrr]\\
& &&&&& &&&& &&&& &&\\
\ov{2}&& &&[1,4]\ar@{->}[ddrr] \ar@{->}[uurr] &&[3,7]\ar@{->}[ddrr] \ar@{->}[uurr]  &&[2,5]\ar@{->}[ddrr] \ar@{->}[uurr]&&
[1,6] \ar@{->}[ddrr]\ar@{->}[uurr] &&[5,7]\ar@{->}[ddrr] \ar@{->}[uurr] && [2] && [6] \\
&&&&&& &&&&&&&&\\
\ov{3}& &[3,4]\ar@{->}[uurr] \ar@{->}[dr]  && [3,5] \ar@{->}[dr] \ar@{->}[uurr] && [1,5]\ar@{->}[uurr] \ar@{->}[dr]
&& [1,7]\ar@{->}[dr] \ar@{->}[uurr]  &&[2,7]\ar@{->}[uurr] \ar@{->}[dr] && [2,6]\ar@{->}[dr] \ar@{->}[uurr]  &&[5,6]\ar@{->}[uurr]\\
\ov{4}& [4]\ar@{->}[ur]&&[3]\ar@{->}[ur]&&[4,5]\ar@{->}[ur]& &[1,3]\ar@{->}[ur]&&[4,7]\ar@{->}[ur]& &[2,3]\ar@{->}[ur]&&[4,6]\ar@{->}[ur]
}}
$$
\end{example}

\vskip 2mm
Recall that we identify $\bar{i}$ with $j= \big| \, \big\{\, \bar{i'}\in\bar{I}\ |\ \text{min}\{j'\in \bar{i}'\} \leq \text{min}\{ j\in \bar{i}\}\, \big\}\, \big|$. Hence $\bar{I}=\{ 1,2,\ldots, n,n+1 \}$ via $\bar{i}= \bar{i}^\vee= i$ for $i\leq n+1.$

\begin{definition} For $\be$ with $\widehat{\Omega}_{[\ii_0]}(\be)=(i,p)$ and $(i,p-4) \in \widehat{\Upsilon}_{[\ii_0]}$, we denote by
${}^{\wpl 1}\be$ the positive root $\al$ such that $\widehat{\Omega}_{[\ii_0]}(\al)=(i,p-4)$. More generally,
${}^{\wpm k}\be$ denotes the positive root $\al$ such that $\widehat{\Omega}_{[\ii_0]}(\al)=(i,p \mp  4k)$ for $k \in \Z_{\ge 0}$.
\end{definition}

\begin{lemma} \cite[Lemma 3.2.3]{KKK13b} \label{lem: Leclerc}
For any Dynkin quiver $Q$ of type $AD_n$, the positions of simple roots inside of $\Gamma_Q$ are on the boundary of $\Gamma_Q$;
that is, either
\begin{itemize}
\item[{\rm (i)}] $\al_i$ are a sink or a source of $\Gamma_Q$, or
\item[{\rm (ii)}] residue of $\al_i$ is
$\begin{cases}
\text{$1$ or $n$} & \text{ if $Q$ is of type $A_n$,} \\
\text{$1$, $n-1$ or $n$} & \text{ if $Q$ is of type $D_n$.}
\end{cases}$
\end{itemize}
Furthermore, all sinks and sources of $\Gamma_Q$ have their labels as simple roots.
\end{lemma}

Now we have an analogue of the above lemma for all folded AR-quivers:

\begin{lemma} For any folded AR-quiver of type $A_{2n+1}$, positions of simple roots inside of $\widehat{\Upsilon}_{[\ii_0]}$ are on the boundary of
$\widehat{\Upsilon}_{[\ii_0]}$; that is, either
\begin{itemize}
\item[{\rm (i)}] $\al_i$ are a sink or a source $\widehat{\Upsilon}_{[\ii_0]}$, or
\item[{\rm (ii)}] $\al_i$ has $1$ or $n+1 \in \ov{I}$ as its folded residue.
\end{itemize}
Furthermore, all sinks and sources of $\widehat{\Upsilon}_{[\ii_0]}$ have their labels as simple roots.
\end{lemma}

\begin{proof}
It is an immediate consequence of Remark \ref{rem: cetral in GammaQ}, Corollary \ref{cor: label for non-induced}, Corollary \ref{cor:label1}
and Lemma \ref{lem: Leclerc}.
\end{proof}

\begin{remark} \label{rem: Bn}
  Note that $\ov{I}$ can be thought as an index set of $B_{n+1}$. Let us introduce an interesting observation (Algorithm \ref{alg: fRef Q})
  in the sense that the action of reflection functors on twisted AR-quivers of type $A_{2n+1}$ can be
  described by the notions on the finite type $B_{n+1}$ :  In order to do this, we fix notations as follows.
\begin{enumerate}
\item[{\rm (a)}] Let $\ov{\mathsf{D}}={\rm diag}(d_{i} \ | \ i \in \ov{I} )={\rm diag}(2,2,\ldots,2,1)$ be the diagonal matrix
which diagonalizes the Cartan matrix $\cmA=(a_{ij})$ of type $B_{n+1}$.
\item[{\rm (b)}] We denote by $\ov{\mathsf{d}}={\rm lcm}(d_{i} \ | \ i \in \ov{I})=2$.
\end{enumerate}
Recall that the involution $^*$ induced by the longest element
$w_0$ of type $B_{n+1}$ is an identity map $\ov{i} \mapsto \ov{i}$.
\end{remark}
Now we shall describe the algorithm which shows a way of obtaining $\widehat{\Upsilon}_{[\ii_0]r_i}$ from
$\widehat{\Upsilon}_{[\ii_0]}$ by using the notations on $B_{n+1}$. (cf. Algorithm \ref{alg: tRef Q}.)

\begin{algorithm} \label{alg: fRef Q}
Let $\mathsf{h}^\vee=2n+1$ be a dual Coxeter number of $B_{n+1}$ and $\al_i$ $(i \in I)$ be a sink of $[\ii_0] \in \lf \Qd \rf$.
\begin{enumerate}
\item[{\rm (A1)}] Remove the vertex $(\ov{i},p)$ such that $\widehat{\Omega}_{[\ii_0]}(\al_i)=(\ov{i},p)$ and the arrows entering into $(\ov{i},p)$ in
$\widehat{\Upsilon}_{[\ii_0]}$.
\item[{\rm (A2)}] Add the vertex $(\ov{i},p-\ov{\mathsf{d}} \times \mathsf{h}^\vee)$ and
the arrows to all vertices whose coordinates are $(\ov{j}, p-\ov{\mathsf{d}} \times \mathsf{h}^\vee+\min(d_{\ov{i}},d_{\ov{j}})) \in \widehat{\Upsilon}_{[\ii_0]}$,
where $\ov{j}$ is adjacent to $\ov{i}$ in Dynkin diagram of type $B_{n+1}$.
\item[{\rm (A3)}] Label the vertex $(\ov{i},p-\ov{\mathsf{d}} \times \mathsf{h}^\vee)$ with $\al_i$ and change the labels $\be$ to $s_i(\be)$ for all $\be \in \widehat{\Upsilon}_{[\ii_0]}
\setminus \{\al_i\}$.
\end{enumerate}
\end{algorithm}

Also, using coordinates in $\widehat{\Upsilon}_{[\ii_0]}$, the twisted additive property  can be restated as follows. (cf. Theorem \ref{thm:twisted additive})

\begin{proposition} \label{Prop:twisted_adapted_2}
Let $\alpha$ and $\beta$ be positive roots with coordinates $(i, 2p-2^{|\ov{i}|})$ and  $(i, 2p)$ in
$\widehat{\Upsilon}_{[\ii_0]}$. Here $i \in \ov{I}$ and $|i|$ denotes the number of indices in the orbit $\ov{i}$. Then
\[ \alpha+\beta= \sum_{\gamma \in \mathcal{J}} \gamma\]
where $\gamma\in \mathcal{J}$ consists of $\gamma$ which are on paths from $\alpha$ to $\beta$ and have the coordinates $(j, 2p-2^{|i|-1})$ for $j\in \bar{I}.$
\end{proposition}

\begin{remark}
Proposition \ref{Prop:twisted_adapted_2} is equivalent to Theorem \ref{thm:twisted additive}. However, Proposition \ref{Prop:twisted_adapted_2} is better to see relations between additive properties of $\Gamma_Q$ and twisted additive properties of $\widehat{\Upsilon}_{[\Qd]}$. In Proposition \ref{Prop:twisted_adapted_2}, if we take $|\bar{i}|=1$ for all $i\in I$ then we directly obtain additive properties of $\Gamma_Q$.
\end{remark}

\subsection{Minimal pairs for $\ga \in \PR \setminus \Pi$ and their coordinates in $\widehat{\Upsilon}_{[\ii_0]}$} In this subsection, we record
coordinates of minimal pairs for $\ga \in \PR \setminus \Pi$ in $\widehat{\Upsilon}_{[\ii_0]}$. The following proposition is an immediate
consequence of Lemma \ref{lem: minimal for rds2}, Lemma \ref{lem: minimal for induced rds1} and Lemma \ref{lem: minimal for non-induced central}:

\begin{proposition}
For $\al,\be,\ga \in \PR$ with $\widehat{\Omega}_{[\ii_0]}(\al)=(i,p)$, $\widehat{\Omega}_{[\ii_0]}(\be)=(j,q)$
$\widehat{\Omega}_{[\ii_0]}(\ga)=(k,r)$ and  $\al+\be =\ga$ $(i,j,k \in \ov{I})$, $(\al,\be)$ is an $[\ii_0]$-minimal pair of $\ga$
if and only if one of the following conditions holds:
\begin{eqnarray}&&
\left\{\hspace{1ex}\parbox{75ex}{
\begin{enumerate}
\item[{\rm (i)}]
$\ell \seteq \max(i,j,k) \le n$, $i+j+k=2\ell$
and
$$ \left( \dfrac{q-r}{2},\dfrac{p-r}{2} \right) =
\begin{cases}
\big( -i,j \big), & \text{ if } \ell = k,\\
\big( i-(2n+1),j \big), & \text{ if } \ell = i,\\
\big( -i,2n+1-j  \big), & \text{ if } \ell = j.
\end{cases}
$$
\item[{\rm (ii)}] %Two of $(i,j,k)$ are $n+1$,
$s \seteq \min(i,j,k) \le n$, the others are the same as $n+1$ and
%and
$$ (q-r,p-r) =
\begin{cases}
\big( -2(n-k)+1,2(n-k)-1 ), & \text{ if } s = k,\\
\big( -4i-4,2(n-i)-1  ), & \text{ if } s = i,\\
\big( -2(n-j)+1, 4j+4), & \text{ if } s = j.
\end{cases}
$$
\end{enumerate}
}\right. \label{eq: Dorey folded coordinate Bn+1}
\end{eqnarray}
\end{proposition}

\section{Distance polynomial and Folded distance polynomial}
In this section, we briefly review the distance polynomial defined on the adapted cluster point $\lf Q \rf$ which was studied in \cite{Oh15E}.
Here $Q$ is any Dynkin quiver of type $ADE_m$. Then we introduce and study the folded distance polynomial
which is well-defined on the twisted adapted cluster point $\lf \Qd \rf$ of type $A_{2n+1}$ and is deeply related to quantum affine algebra of
type $B^{(1)}_{n+1}$. This section can be understood as a twisted analogue of \cite[Section 6]{Oh15E}. For this section and next section, we refer to and follow
\cite[Section 6]{Oh15E} and \cite{Kas02} for quantum affine algebras, their integrable representations and denominator formulas.

\bigskip

Let us take a base field $\ko$ the algebraic closure of $\C(q)$ in $\cup_{m >0} \C((q^{1/m}))$.

\subsection{Distance polynomial} For our purpose, we assume that the Dynkin quiver $Q$ is of type $A_{2n}$ with its index set $I_{2n}$.

\begin{definition} \cite[Definition 6.11]{Oh15E}
For a Dynkin quiver $Q$, indices $k,l \in I_{2n}$ and an integer $t \in \mathbb{N}$, we define the subset $\Phi_{Q}(k,l)[t]$ of $\PR \times \PR$ as follows:

A pair $(\alpha,\beta)$ is contained in $\Phi_{Q}(k,l)[t]$ if $\alpha \prec_Q \beta$ or $\be \prec_Q \al$ and
$$\{ \Omega_Q(\al),\Omega_Q(\be) \}=\{ (k,a), (l,b)\} \quad \text{ such that } \quad |a-b|=t.$$
\end{definition}

\begin{lemma} \cite[Lemma 6.12]{Oh15E} \label{lem: o well}
For any  $(\alpha^{(1)},\beta^{(1)})$ and $(\alpha^{(2)},\beta^{(2)})$ in $\Phi_{Q}(k,l)[t]$, we have
$$ o^{Q}_t(k,l) := \gdist_Q(\alpha^{(1)},\beta^{(1)})=\gdist_{Q}(\alpha^{(2)},\beta^{(2)}). $$
\end{lemma}

We denoted by $Q^{{\rm rev}}$ the quiver obtained
by reversing all arrows of $Q$ and $Q^*$ the quiver obtained from $Q$ by replacing vertices of $Q$ from $i$ to $i^*$.
Here $^*: i \leftrightarrow 2n+1-i$ $(i \le n)$ is an involution on $I_{2n}$ induced by the longest element $w_0$ of $W_{2n}$.

\begin{definition}\cite[Definition 6.15]{Oh15E} \label{def: Dist poly Q 1A2n}
For $k,l \in I_{2n}$ and a Dynkin quiver $Q$, we define a polynomial $D^Q_{k,l}(z) \in \ko[z]$ as follows:
$$D^Q_{k,l}(z) \seteq  \prod_{ t \in \Z_{\ge 0} } (z-(-1)^t q^{t} )^{\mathtt{o}^{\overline{Q}}_t(k,l)},$$
where
\begin{align} \label{def: non-fold O}
\mathtt{o}^{\overline{Q}}_t(k,l) \seteq  \max( o^{Q}_t(k,l),o^{Q^\rev}_t(k,l) ).
\end{align}
\end{definition}

\begin{proposition} \cite[Proposition 6.16]{Oh15E} \label{prop: DQ DQ' 1A2n}
For $k,l \in I_{2n}$ and any Dynkin quivers $Q$ and $Q'$, we have
\begin{enumerate}
\item[{\rm (a)}] $D^Q_{k,l}(z)=D^Q_{l,k}(z)=D^Q_{k^*,l^*}(z)=D^Q_{l^*,k^*}(z)$.
\item[{\rm (b)}] $D^Q_{k,l}(z)=D^{Q'}_{k,l}(z)$.
\end{enumerate}
Hence $D_{k,l}(z)$ is well-defined for $\lf Q \rf$.
\end{proposition}

Now, we recall the denominator formulas $d^{A^{(1)}_{2n}}_{k,l}(z)=d^{A^{(1)}_{2n}}_{l,k}(z)$ $(1 \le k ,l \le 2n)$ for $U_q'(A^{(1)}_{2n})$:

\begin{theorem} \cite{AK}  \label{thm: denom 1A2n} For any $1 \le k $ and $l \le 2n$, we have
$$ d^{A^{(1)}_{2n}}_{k,l}(z)= \prod_{s=1}^{ \min(k,l,2n+1-k,2n+1-l)} \big(z-(-1)^{k+l}q^{2s+|k-l|}\big).$$
\end{theorem}

\begin{theorem} \cite[Theorem 6.18]{Oh15E}\label{eq: dist denom 1A2n}
For any adapted class $[Q]$ of type $A_{2n}$, the denominator formulas for the quantum affine algebra $U'_q(A_{2n}^{(1)})$ can be read
from $\Gamma_Q$ and $\Gamma_{Q^\rev}$ as follows:
\begin{align*}
d^{A^{(1)}_{2n}}_{k,l}(z) & =D_{k,l}(z) \times
(z+q^{\mathtt{h}^\vee})^{\delta_{l,k^*}}
\end{align*}
where $\mathtt{h}^\vee=2n+1$ is the dual Coxeter number of $A_{2n}$.
\end{theorem}

\subsection{Folded distance polynomial} In this subsection, we follow the framework of the previous section to prove the well-definedness of
the folded distance polynomial.

\bigskip

{\it In this subsection, we mainly deal with $\ov{I}$ and hence fix the indices of $\ov{I}$ as follows:}
$$ \ov{I}=\{ 1,2,\ldots, n,n+1\}.$$

\begin{definition}
For a folded AR-quiver $\wUp_{[\ii_0]}$, indices $\ov{k},\ov{l} \in \ov{I}$ and an integer $t \in \mathbb{N}$,
we define the subset $\Phi_{[\ii_0]}(\ov{k},\ov{l})[t]$ of $\PR \times \PR$ as follows:

A pair $(\alpha,\beta)$ is contained in $\Phi_{[\ii_0]}(\ov{k},\ov{l})[t]$ if $\alpha \prec_Q \beta$ or $\be \prec_Q \al$ and
$$\{ \widehat{\Omega}_{[\ii_0]}(\al),\widehat{\Omega}_{[\ii_0]}(\be) \}=\{ (\ov{k},a), (\ov{l},b)\} \quad \text{ such that } \quad |a-b|=t.$$
\end{definition}

\begin{lemma}
For any  $(\alpha^{(1)},\beta^{(1)})$ and $(\alpha^{(2)},\beta^{(2)})$ in $\Phi_{[\ii_0]}(\ov{k},\ov{l})[t]$, we have
$$ o^{[\ii_0]}_t(\ov{k},\ov{l}) := \gdist_{[\ii_0]}(\alpha^{(1)},\beta^{(1)})=\gdist_{[\ii_0]}(\alpha^{(2)},\beta^{(2)}). $$
\end{lemma}

\begin{proof}
{\rm (1)} Assume that $\ov{k},\ov{l} \in \ov{I} \setminus \{ n+1 \}$. By the observation in \eqref{eq: not equiv}, the set $\Phi_{[\ii_0]}(\ov{k},\ov{l})[t]$
is induced from one of
\begin{align} \label{eq: sets}
\Phi_{Q}(k,l)[t/2] \sqcup  \Phi_{Q}(k^*,l^*)[t/2] \text{ and } \Phi_{Q}(k^*,l)[t/2] \sqcup  \Phi_{Q}(k,l^*)[t/2].
\end{align}
where $\PPi([\ii_0])=[Q]$ and $i  \overset{*}{\leftrightarrow} 2n+1-i$. Note that one of the sets in \eqref{eq: sets} must be empty.

In each case, if there exists a path from $\beta^{(1)}$ to $\al^{(1)}$ passing through two non-induced vertices, then so is $(\alpha^{(2)},\beta^{(2)})$.
Then our assertion for this case follows from Corollary \ref{cor:label1}, Theorem \ref{thm: dist upper bound Qd} and Lemma \ref{lem: o well}.

{\rm (2)} Assume that $\ov{k}=\ov{l}=\{ n+1 \}$. By Lemma \ref{cor: label for non-induced}
either {\rm (i)} $\al^{(j)}+\be^{(j)} \in \PR$ or {\rm (ii)} $\al^{(j)}+\be^{(j)} \not \in \PR$ for all $j$.
Then, for all $j$, we have
$$
\gdist_{[\ii_0]}(\al^{(j)},\be^{(j)})=
\begin{cases}
1 & \text{  if {\rm (i)}}, \\
0 & \text{  if {\rm (ii)}},
\end{cases}
$$
by Lemma \ref{lem: less than eq 2}.

{\rm (3)} Assume that one of $\ov{k}$ and $\ov{l}$ is $n+1$ and the other is not. Let us fix $\ov{k}=n+1$ and $\widehat{\Omega}_{[\ii_0]}(\al)=(n+1,a)$.
As in the same reason of {\rm (i)}, $\be$ is induced from some vertex whose residue is one of $l$ and $l^*$.
Note that for $\al \prec_{[\ii_0]} \ga \prec_{[\ii_0]} \be$,
$$\text{${}^{\wpl 1}\be \in \PR$ implies ${}^{\wpl 1}\ga \in \PR$ and ${}^{\wmi 1} \al \in \PR$ implies ${}^{\wmi 1} \ga \in \PR$.}$$

Thus there exists
$k \in \Z$ such that ${}^{\wpl k}\al^{(1)} =\al^{(2)}$ and ${}^{\wpl k}\be^{(1)} =\be^{(2)}$. Hence we have
$$\gdist_{[\ii_0]}(\alpha^{(1)},\beta^{(1)})=\gdist_{[\ii_0]}(\alpha^{(2)},\beta^{(2)}),$$
by Theorem \ref{them: comp for length k ge 0} and Corollary \ref{cor: label for non-induced}.
\end{proof}

Note that
\begin{align} \label{eq: Q< Qrev<}
\wUp_{[Q^<]} \simeq \wUp_{[{Q^\rev}^<]}
\end{align}
as quivers with the consideration on their {\it folded} coordinate system.

\medskip

We keep the notations on $B_{n+1}$ in Remark \ref{rem: Bn}.

\begin{definition} \label{def: Dist poly Q}
For $\ov{k},\ov{l} \in \ov{I}$ and a folded AR-quiver $\wUp_{[\ii_0]}$, we define a polynomial $\widehat{D}^{[\ii_0]}_{\ov{k},\ov{l}}(z) \in \ko[z]$ as follows:
$$\widehat{D}^{[\ii_0]}_{\ov{k},\ov{l}}(z) \seteq \prod_{ t \in \Z_{\ge 0} } (z- (-1)^{\ov{k}+\ov{l}}(q_s)^{t})^{\mathtt{o}^{[\ii_0]}_t(\ov{k},\ov{l})} ,$$
where
$$ q_s^{\ov{\mathsf{d}}}=q_s^2=q \quad \text{ and } \quad  \mathtt{o}^{[\ii_0]}_t(\ov{k},\ov{l}) \seteq
\left\lceil \dfrac{o^{[\ii_0]}_t(\ov{k},\ov{l})}{ \ov{\mathsf{d}} } \right\rceil.$$
\end{definition}
\begin{remark}
By the observation in \eqref{eq: Q< Qrev<}, it is enough to consider only one folded quiver $\widehat{\Upsilon}_{[\ii_0]}$ to define
$\mathtt{o}^{[\ii_0]}_t(\ov{k},\ov{l})$ while we need to consider two AR-quivers for defining $\mathtt{o}^{\overline{Q}}_t(k,l)$ in \eqref{def: non-fold O}.
\end{remark}

\begin{proposition} \label{prop: DQ DQ'}
For $\ov{k},\ov{l} \in \ov{I}$ and any twisted adapted classes $[\ii_0]$ and $[\ii_0']$ in $\lf \Qd \rf$, we have
$$\widehat{D}^{[\ii_0]}_{\ov{k},\ov{l}}(z)=\widehat{D}^{[\ii_0']}_{\ov{k},\ov{l}}(z).$$
\end{proposition}

\begin{proof}
It is enough to consider when $[\ii'_0]=[\ii_0]r_i$. Then our assertion is obvious for $i \ne n+1$
by Remark \ref{Rem:surgery} and Proposition \ref{prop: DQ DQ' 1A2n}.
When $i =n+1$ is also obvious from the fact that $[Q^>]r_{n+1}=[Q^<]$.
\end{proof}

From the above proposition, we can define $\widehat{D}_{\ov{k},\ov{l}}(z)$ for the twisted adapted cluster point $\lf \Qd \rf$ in a natural way
and call it {\it the folded distance polynomial} at $\ov{k}$ and $\ov{l}$.

\medskip

Now, we recall the denominator formulas $d^{B^{(1)}_{n+1}}_{k,l}(z)=d^{B^{(1)}_{n+1}}_{l,k}(z)$ for the quantum affine algebra  $U_q'(B^{(1)}_{n+1})$:

\begin{theorem} \cite{Oh14} \hfill \label{thm: denom 1}
$$
\begin{cases}
d^{B^{(1)}_{n+1}}_{k,l}(z) =  \displaystyle \prod_{s=1}^{\min (k,l)} \big(z-(-1)^{k+l}q^{|k-l|+2s}\big)\big(z-(-1)^{k+l}q^{2n+1-k-l+2s}\big) & \text{ if } 1 \le k,l \le n, \\
d^{B^{(1)}_{n+1}}_{k,n+1}(z) = \displaystyle  \prod_{s=1}^{k}\big(z-(-1)^{n+1+k}q_s^{2n-2k+1+4s}\big) & \text{ if } 1 \le k \le n, \\
d^{B^{(1)}_{n+1}}_{n+1,n+1}(z)=\displaystyle \prod_{s=1}^{n} \big(z-(q_s)^{4s-2}\big) \times (z-q^{\mathtt{h}^\vee}) & \text{ if } k=l=n+1.
\end{cases}
$$
\end{theorem}

By considering the denominator formulas $d^{A^{(1)}_{2n}}_{k,l}(z)$, $d^{B^{(1)}_{n+1}}_{k,l}(z)$ and the distance polynomial $D_{k,l}(z)$ for $1 \le k,l \le n$,
we have an interesting interpretation as follows:
\begin{equation} \label{eq: interesting interpretation}
\begin{aligned}
\dfrac{d^{B^{(1)}_{n+1}}_{k,l}(z)}{(z-q^{\mathtt{h}^\vee})^{\delta_{kl}}} & = D_{k,l}(z) \times D_{k,l^*}(z) = D_{k,l}(z) \times D_{k^*,l}(z)  \\
& = D_{k^*,l^*}(z) \times D_{k,l^*}(z)  =  D_{k^*,l^*}(z) \times D_{k^*,l}(z)
\end{aligned}
\end{equation}
where $D_{k,l}(z)=D_{l,k}(z)=D_{k^*,l^*}(z)=D_{l^*,k^*}(z)$ and $i  \overset{*}{\leftrightarrow} 2n+1-i$.

\begin{theorem} \label{thm: dist denom}
For any twisted adapted class $[\ii_0]$, the denominator formulas for $U'_q(B_{n+1}^{(1)})$ can be read
from $\wUp_{[\ii_0]}$ as follows:
\begin{align*}
d^{B^{(1)}_{n+1}}_{\ov{k},\ov{l}}(z) & =\widehat{D}_{\ov{k},\ov{l}}(z) \times (z-q^{\mathtt{h}^\vee})^{\delta_{\ov{l},\ov{k}}}
\end{align*}
where $\mathtt{h}^\vee$ is the dual Coxeter number of $B_{n+1}$.
\end{theorem}

\begin{proof} Fix $Q$ such that $\PPi([\ii_0])=[Q]$.

(1) Assume that $\ov{k}, \ov{l} \in \ov{I}\setminus \{n+1\}$. Then one of
$o^{Q}_t(k,l)$ and $o^{Q}_t(k^*,l)$ is positive implies that $o^{[\ii_0]}_t(\ov{k},\ov{l})>0$ and hence
$\mathtt{o}^{[\ii_0]}_t(\ov{k},\ov{l})=1$ by Proposition \ref{prop: less than eq to 2}. Thus our assertion for this case
follows from \eqref{eq: interesting interpretation}.

(2) Assume that $\ov{k} = \ov{l}=n+1$. Then our assertion is obvious by Corollary \ref{cor: label for non-induced} and Corollary \ref{cor: triangle mesh}.

(3) In general, it suffices to consider only one folded AR-quiver $\wUp_{[\ii_0]}$ by Proposition \ref{prop: DQ DQ'}. We take the Dynkin quiver
$$Q \ : \ \xymatrix@R=3ex{ *{ \bullet }<3pt> \ar@{<-}[r]_<{1}  &*{\bullet}<3pt>
\ar@{<-}[r]_<{2} &\cdots\ar@{<-}[r] &*{\bullet}<3pt>
\ar@{<-}[r]_<{2n-1} &*{\bullet}<3pt>
\ar@{-}[l]^<{\ \ 2n}} \quad \text{ of type $A_{2n}$}$$
and $[\ii_0]$ as $[Q^<]$. Then \cite[(6.20)]{Oh15E} and Corollary \ref{cor:label1}, we can draw $\wUp_{[\ii_0]}$ with its labels. Then one can check that the assertion for
$\ov{k}=n+1$ and $\ov{l} \ne n+1$ holds by reading $\wUp_{[\ii_0]}$. We skip the proof and provide a particular example for $[Q^<]$ of type $A_5$ instead.
\end{proof}

\begin{example} Here are $\Gamma_Q$, $\Upsilon_{[Q^<]}$ and $\wUp_{[Q^<]}$ for
$Q \ : \ \xymatrix@R=3ex{ *{ \bullet }<3pt> \ar@{<-}[r]_<{1}  &*{\bullet}<3pt>
\ar@{<-}[r]_<{2} &*{\bullet}<3pt>
\ar@{<-}[r]_<{3} &*{\bullet}<3pt>
\ar@{-}[l]^<{\ \ 4}}$.
\begin{align*}
& \scalebox{0.6}{\xymatrix@C=1ex@R=1ex{
1 & [4]\ar@{->}[ddrr] &&&& [3]\ar@{->}[ddrr] && \Gamma_Q && [2]\ar@{->}[ddrr] &&&& [1] \\
&&&&&& &&\\
2 &&&[3,4]\ar@{->}[ddrr]\ar@{->}[uurr] &&&&[2,3]\ar@{->}[ddrr]\ar@{->}[uurr] &&&& [1,2] \ar@{->}[uurr]\\
&&&&&&&&\\
3 &&&&&[2,4] \ar@{->}[ddrr]\ar@{->}[uurr]&&&& [1,3] \ar@{->}[uurr]\\
&&&&&&&\\
4 &&&&&&&[1,4]\ar@{->}[uurr]
}}
 \scalebox{0.6}{\xymatrix@C=1ex@R=1ex{
1 & [5] \ar@{->}[ddrr]&&&& [3,4]\ar@{->}[ddrr] && \Upsilon_{[Q^<]}&& [2]\ar@{->}[ddrr] &&&& [1] \\
&&&&&& &&\\
2 &&&[3,5]\ar@{->}[dr]\ar@{->}[uurr] &&&&[2,4]\ar@{->}[dr]\ar@{->}[uurr] &&&& [1,2]\ar@{->}[uurr] \\
3 &&[3]\ar@{->}[ur]&&[4,5]\ar@{->}[dr]&&[2,3]\ar@{->}[ur]&&[4]\ar@{->}[dr]&&[1,3]\ar@{->}[ur] \\
4 &&&&&[2,5]\ar@{->}[ur]\ar@{->}[ddrr] &&&& [1,4]\ar@{->}[ur] \\
&&&&&&&\\
5 &&&&&&&[1,5]\ar@{->}[uurr]
}} \\
& \qquad\qquad\qquad \scalebox{0.8}{\xymatrix@C=1ex@R=1ex{
\ov{1} & [5] \ar@{->}[ddrr]&& \wUp_{[Q^<]}&& [3,4]\ar@{->}[ddrr] &&[1,5]\ar@{->}[ddrr]&& [2]\ar@{->}[ddrr] &&&& [1] \\
&&&&&& &&\\
\ov{2} &&&[3,5]\ar@{->}[dr]\ar@{->}[uurr] &&[2,5]\ar@{->}[uurr]\ar@{->}[dr]&&[2,4]\ar@{->}[dr]\ar@{->}[uurr] &&[1,4]\ar@{->}[dr]&& [1,2]\ar@{->}[uurr] \\
\ov{3} &&[3]\ar@{->}[ur]&&[4,5]\ar@{->}[ur]&&[2,3]\ar@{->}[ur]&&[4]\ar@{->}[ur]&&[1,3]\ar@{->}[ur]
}}
\end{align*}
\end{example}

\section{Dorey's rule for $U_q'(B^{(1)}_{n+1})$}

In this section, we shall interpret Dorey's rule for $U_q'(B^{(1)}_{n+1})$, studied in \cite{CP96}, via $[\ii_0]$-minimal pairs
$(\al,\be)$ of $\ga \in \PR_{A_{2n+1}}$ when $[\ii_0]$ is twisted adapted.
Such interpretation is analogous to the arguments for $U_q'(A^{(i)}_{n})$ and $U_q'(D^{(i)}_{n})$ in \cite{KKKO14D,Oh14A,Oh14D} $(i=1,2)$. We first recall
Dorey's rule for $U_q'(B^{(1)}_{n+1})$.

\begin{proposition} \cite[Theorem 8.1]{CP96} $($see also \cite{Oh14}$)$
Let $(i,x)$, $(j,y)$, $(k,z) \in \ov{I} \times \ko^\times$. Then
$$ {\rm Hom}_{U_q'(B^{(1)}_{n+1})}\big( V(\varpi_{j})_y \otimes V(\varpi_{i})_x , V(\varpi_{k})_z  \big) \ne 0 $$
if and only if one of the following conditions holds:
\begin{eqnarray}&&
\left\{\hspace{1ex}\parbox{75ex}{
\begin{enumerate}
\item[{\rm (i)}] $\ell \seteq \max(i,j,k) \le n$, $i+j+k=2\ell$
and
$$ \left( y/z,x/z \right) =
\begin{cases}
\big( (-1)^{j+k}q^{-i},(-1)^{i+k}q^{j} \big), & \text{ if } \ell = k,\\
\big( (-1)^{j+k}q^{i-(2n+1)},(-1)^{i+k}q^{j} \big), & \text{ if } \ell = i,\\
\big( (-1)^{j+k}q^{-i},(-1)^{i+k}q^{2n+1-j}  \big), & \text{ if } \ell = j.
\end{cases}
$$
\item[{\rm (ii)}] $s \seteq \min(i,j,k) \le n$, the others are the same as $n+1$ and
$$ (y/z,x/z) =
\begin{cases}
\big( (-1)^{n+1+k}q_s^{-2(n-k)+1},(-1)^{n+1+k}q_s^{2(n-k)-1} ), & \text{ if } s = k,\\
\big(  q_s^{-4i-4},(-1)^{i+n+1} q_s^{2(n-i)-1}  ), & \text{ if } s = i,\\
\big( (-1)^{j+n+1}q_s^{-2(n-j)+1}, q_s^{4j+4}), & \text{ if } s = j.
\end{cases}
$$
\end{enumerate}
}\right. \label{eq: Dorey B}
\end{eqnarray}
Here $V(\varpi_i)$ is the unique simple $U_q'(B_{n+1}^{(1)})$-module which is finite dimensional integrable with its dominant weight $\varpi_i$ $(i \in \ov{I})$
$($see \cite{Kas02} for more detail$)$.
\end{proposition}

\begin{definition} \label{def: [ii0] module category} (cf. \cite[Definition 2.4]{KKKO14D})
Let $[\ii_0]$ be a twisted adapted class of type $A_{2n+1}$.
For any positive root $\beta$ contained in $\Phi^+$, we set the $U_q'(B^{(1)}_{n+1})$-module $V_{[\ii_0]}(\beta)$ as follows:
\begin{align} \label{eq: Vii0(beta)}
V_{[\ii_0]}(\beta) \seteq V(\varpi_{i})_{(-1)^{i}(q_s)^{p}} \quad \text{ where } \quad \widehat{\Omega}(\beta)=({i},p).
\end{align}
We let the smallest abelian full subcategory $\mathscr{C}_{[\ii_0]}$ consist of finite dimensional integrable $U_q'(B^{(1)}_{n+1})$-modules such that
\begin{itemize}
\item[{\rm (a)}] it is stable by taking subquotient, tensor product and extension,
\item[{\rm (b)}] it contains $V_{[\ii_0]}(\beta)$ for all $\beta \in \Phi^+$.
\end{itemize}
\end{definition}

\begin{theorem}\label{thm: Dorey} $($cf. \cite[Theorem 5.2]{Oh14A}$)$
Let $(i,x)$, $(j,y)$, $(k,z) \in \ov{I} \times \ko^\times$. Then
$$ {\rm Hom}_{U_q'(B^{(1)}_{n+1})}\big( V(\varpi_{j})_y \otimes V(\varpi_{i})_x , V(\varpi_{k})_z  \big) \ne 0 $$
if and only if there exists a twisted adapted class $[\ii_0]$ and $\al,\be,\ga \in \Phi_{A_{2n+1}}^+$ such that
\begin{itemize}
\item[{\rm (i)}] $(\al,\be)$ is an $[\ii_0]$-minimal pair of $\ga$,
\item[{\rm (ii)}] $V(\varpi_{j})_y  = V_{[\ii_0]}(\be)_t, \ V(\varpi_{i})_x  = V_{[\ii_0]}(\al)_t, \ V(\varpi_{k})_z  = V_{[\ii_0]}(\ga)_t$
for some $t \in \ko^\times$.
\end{itemize}
\end{theorem}

\begin{proof}
By comparing \eqref{eq: Dorey folded coordinate Bn+1} and \eqref{eq: Dorey B},
our assertion is a consequence of \eqref{eq: Vii0(beta)} in Definition \ref{def: [ii0] module category}.
\end{proof}

\begin{corollary} \label{Cor: stable}
The condition {\rm (b)} in {\rm Definition \ref{def: [ii0] module category}} can be restated as follows:
\begin{itemize}
\item[{\rm (b)}$'$] It contains $V_{[\ii_0]}(\alpha_k)$ for all $\alpha_k \in \Pi$.
\end{itemize}
\end{corollary}

\appendix
\section*{Appendix: Foldable cluster $r$-points for exceptional cases}
In this appendix, we show the existence of foldable cluster $r$-points for exceptional cases for $E_6$ associated with $\vee$ in \eqref{eq: F_4} and
$D_4$ associated with $\vee$ in \eqref{eq: G_2}.

\vskip 3mm

For type $E_6$, there exists a $\vee$-foldable cluster $r$-point $\lf \ii_0 \rf$ where $\vee$ is the one in \eqref{eq: F_4}:
$$ \ii_0 =\prod_{k=0}^{8} (1\ 2\ 6\ 3)^{k\vee}.$$

The $r$-cluster point $\lf \ii_0 \rf$ is called the {\it twisted adapted cluster point} of type $E_6$ and a class in
$[\ii_0'] \in \lf \ii_0 \rf$ is called a {\it twisted adapted class} of type $E_6$. Furthermore, there are $32$ distinct
twisted adapted classes in $\lf \ii_0 \rf$ while there are only $24$ distinct twisted Coxeter elements. Here the number $32$
coincides with the number of distinct Dynkin quivers of type $E_6$.

\vskip 3mm

For type $D_4$, there exists a unique $\vee$-foldable cluster $r$-point $\lf \ii_0 \rf$
and a unique $\vee^2$-foldable cluster $r$-point $\lf \jj_0 \rf$ where $\vee$ is the one in \eqref{eq: G_2}:
$$ \ii_0 =\prod_{k=0}^{5} (2\ 1)^{k\vee} \quad \text{and} \quad \jj_0 =\prod_{k=0}^{5} (2\ 1)^{2k\vee}.$$

The $r$-cluster points $\lf \ii_0 \rf$ and $\lf \jj_0 \rf$ are called the {\it triply twisted adapted cluster points} of type $D_4$ and a class in
$[\ii_0'] \in \lf \ii_0 \rf \sqcup \lf \jj_0 \rf$  is called a {\it triply twisted adapted class} of type $D_4$. Each triply twisted adapted class
consists of a unique reduced expression and there are $6$ distinct twisted adapted classes in each triply twisted adapted cluster point.
Recall $12$ is the number of distinct triply twisted Coxeter elements.

\end{document}